\documentclass[reqno, 11pt]{amsart}

\usepackage{
amssymb,
amsmath,
amsthm,
eucal,
dsfont,
multicol,
multirow,
mathrsfs,
tikz,
graphicx,
hyperref,
makecell
}
\usepackage{lipsum}
\makeatletter\renewcommand*{\eqref}[1]{%
\hyperref[{#1}]{\textup{\tagform@{\!\!\ref*{#1}}}}%
}\makeatother 

\makeatletter
 
  \@addtoreset{equation}{section}
 \makeatother
\theoremstyle{plain}
\newtheorem{theorem}{Theorem}[section]
\newtheorem{lemma}[theorem]{Lemma}
\newtheorem{proposition}[theorem]{Proposition}

\newtheorem{corollary}[theorem]{Corollary}
\theoremstyle{definition}
\newtheorem{definition}[theorem]{Definition}
\newtheorem{remark}[theorem]{Remark}

\newtheorem{assumption}[theorem]{Assumption}

\def\supp{\mathop{\mathrm{supp}}\nolimits}

\def\Re{\mathop{\mathrm{Re}}\nolimits}
\def\Im{\mathop{\mathrm{Im}}\nolimits}

\def\sgn{\mathop{\mathrm{sgn}}\nolimits}

\def\R{{\mathbb{R}}}

\def\N{{\mathbb{N}}}

\def\k{{ \textbf{k}}}
\def\i{{\rm i}}
\def\d{{\rm d}}
\def\<{{\langle}}
\def\>{{\rangle}}

\title[$L^p$-boundedness of wave operators for higher order Schr\"odinger operator ]
{The $L^p$-boundedness of wave operators for  higher order Schr\"odinger operator with zero  singularities in low odd dimensions }
\author{Han Cheng, Avy Soffer, Zhao Wu and  Xiaohua Yao}

\address {Han Cheng, Institute of Applied Physics and Computational Mathematics, Beijing, 100088, People's Republic of China}
\email{chmathh@163.com}
\address{Avy Soffer, Mathematics Department, Rutgers University, New Brunswick, NJ, 08903, USA}
\email{soffer@math.rutgers.edu}
\address {Zhao Wu, Faculty of Mathematics and Statistics, Hubei University, Wuhan, 430062, People's Republic of China}
\email{wuzhao@hubu.edu.cn}
\address {Xiaohua Yao, School of Mathematics and Statistics, Key Laboratory of Nonlinear Analysis and Applications (Ministry of Education), and Hubei Key Laboratory of Mathematical Sciences, Central China Normal University, Wuhan, 430079, P.R. China}
\email{yaoxiaohua@ccnu.edu.cn}
\date{\today}
\begin{document}
\begin{abstract}

This paper investigates the $L^p$-bounds of wave operators for higher-order Schrödinger operators
$H = (-\Delta)^m + V$ on $\mathbb{R}^n$, with $m \ge 2$ and real-valued decaying potentials $V$.


In light of the recent works of Erdo\u{g}an and Green \cite{EG22, EG23} for the regular case with $n > 2m$, and of Erdo\u{g}an–Green–LaMarster \cite{EGL24} addressing the zero eigenvalue case when $n > 4m$, our main objective is to establish the sharp $L^p$-boundedness of the wave operators
$W_\pm(H; (-\Delta)^m)$ in the presence of all types of zero-resonance singularities, for all odd dimensions $1 \le n \le 4m - 1$.


Specifically, for odd $n$ with $1 \le n \le 4m - 1$, there exist $m_n$ types of zero resonances for $H$, along with a critical type $k_c$ (both depending on $n$ and $m$). If zero is a regular point of $H$ or a $\mathbf{k}$-th kind resonance with $1 \le \mathbf{k} \le k_c$, the wave operators
$W_\pm(H; (-\Delta)^m)$
are bounded on $L^p(\mathbb{R}^n)$ for all $1 < p < \infty$. If zero is a $\mathbf{k}$-th kind resonance with $k_c < \mathbf{k} \le m_n$, we show that the range of $p$-boundedness for $W_\pm(H; (-\Delta)^m)$ narrows to $1 < p < p_{\mathbf{k}}$, where
$$p_{\mathbf{k}} = \frac{n}{n - 2m + \mathbf{k} + k_c - 1}.$$
Additionally, if zero is an eigenvalue of $H$ (i.e., $\mathbf{k} = m_n + 1$), then
$W_\pm(H; (-\Delta)^m)$ are bounded  on $L^p(\mathbb{R}^n)$
for all
$1 < p < \frac{2n}{n - 1}$.


Furthermore, it is shown that the wave operators $W_\pm(H; (-\Delta)^m)$ are unbounded on $L^p(\mathbb{R}^n)$ for all $p_{\mathbf{k}} < p \le \infty$ if $k_c < \mathbf{k} \le m_n$, and for all $\frac{2n}{n - 1} < p \le \infty$ if zero is an eigenvalue of $H$ with a non-zero solution $\phi$ to
$H\phi = 0$ in $\bigcap_{s < -\frac{1}{2}} L^{2}_{s}(\mathbb{R}^n) \setminus L^2(\mathbb{R}^n)$
(referred to as a $p$-wave resonance). The key idea of the proof is to reduce the $L^p$-unboundedness to establishing the optimality of time-decay estimates for
$e^{itH}P_{ac}(H)$ in weighted $L^2$ spaces.


Finally, we emphasize that the $L^p$-boundedness of wave operators is closely tied to the asymptotic expansions of the resolvent of
$H = (-\Delta)^m + V$
near the zero-energy threshold. Since zero energy of $(-\Delta)^m$ is degenerate for all $m \ge 2$, with the degeneracy becoming stronger as $m$ increases, the classification of zero resonances becomes considerably more intricate compared to the case $m = 1$, requiring new analytical ideas.

\end{abstract}
\maketitle
\tableofcontents

\section{Introduction}
\setcounter{equation}{0}
\subsection{Backgrounds}
Let $m\ge 2$ be an integer,  $H=(-\Delta)^m+V(x)$ be the  higher order Schr\"{o}dinger operator on $\mathbb{R}^n$, and $V(x)$ be a real-valued bounded potential satisfying  $|V(x)|\lesssim\langle x\rangle^{-\beta}$ for some specific $\beta>0$ and certain regularity on $V$ (depending on $n,m$), where $\langle x\rangle=\sqrt{1+|x|^2}$.  

In this paper we study the $L^p$-bounds of the wave operators $W_\pm(H;(-\Delta)^m)$ associated with $H$, which are defined by the following strong limits in $L^2(\mathbb{R}^n)$:
\begin{equation}\label{eq-wave operator-defi}
W_\pm=W_\pm(H;(-\Delta)^m)=:s\text{-}\lim_{t\rightarrow\pm\infty}e^{itH}e^{-it(-\Delta)^m}. 
\end{equation}
Here, $e^{itH}$ and $e^{it(-\Delta)^m}$ are Schr\"odinger groups generated by  $H$ and $(-\Delta)^m$, respectively.   

As $\beta>1,$ it was known that $W_\pm (H;(-\Delta)^m)$ exist on $L^2(\mathbb{R}^n)$ as partial isometric operators and are asymptotically complete ( see e.g.  Kuroda \cite{Kuroda}). In particular,  the wave operators $W_\pm$ satisfy the identities
$W_\pm W_\pm^*=P_{ac}(H)$, $ W_\pm^*W_\pm=I$
and the following intertwining formula  (see e.g. Reed and Simon \cite{Reed-Simon-III}):
\begin{equation}\label{eq-intertwing identity}
f(H)P_{ac}(H)=W_\pm f((-\Delta)^m)W^*_\pm,
\end{equation}
where $f$  is  any Borel measurable function of  $\mathbb{R}$ and $P_{ac}(H)$ denotes the projection onto the absolutely continuous spectral subspace of $H.$ 
By the identity \eqref{eq-intertwing identity} and the $L^p-L^q$ estimates for the free operator $f((-\Delta)^m),$ the boundedness of $W_\pm(H;(-\Delta)^m)$ can immediately be used to establish the $L^p-L^q$ estimates for the perturbed operator $f(H)P_{ac}(H)$ by
\begin{equation}\label{eq-disperesive estimate}
  \|f(H)P_{ac}(H)\|_{L^p\rightarrow L^q}\le \|W_\pm\|_{L^q\rightarrow L^q}\|f((-\Delta)^m)\|_{L^p\rightarrow L^q}\|W_\pm^*\|_{L^p\rightarrow L^p}.
\end{equation}
In fact, under a suitable condition on $f,$ it is comparably easy to establish the $L^p-L^q$ estimates for the free operator $f((-\Delta)^m)$ by Fourier multiplier methods. Thus, in order to obtain the estimate \eqref{eq-disperesive estimate}, it suffices to prove the following possible boundedness of $W_\pm(H;(-\Delta)^m)$ and $W_\pm^*(H;(-\Delta)^m):$
\begin{equation}\label{eq-Lpbounds}
  \|W_\pm f\|_{L^q\rightarrow L^q}\lesssim\|f\|_{L^q},\ \ \ \ \|W_\pm^* f\|_{L^p\rightarrow L^p}\lesssim\|f\|_{L^p}.
\end{equation}
Due to such feature, the $L^p$ boundedness of wave operators for the Schr\"odinger operator $-\Delta + V$ (i.e. $m = 1$) has been extensively developed as a fundamental tool for studying various nonlinear dispersive equations, such as the non-linear Schr\"odinger and Klein-Gordon equations with potentials. As a result, it is natural and important to extend this study to more general higher-order elliptic operators $(-\Delta)^m + V(x)$, which has garnered increasing attention in the mathematics and mathematical physics communities.

Despite the recent emergence of many related works (see e.g. \cite{EG22, GG21, GY23, MWY22, MWY23-1, MWY23-2}), they remain relatively scarce compared to the case of the Schr\"odinger operator $-\Delta + V(x)$. In particular, to the best of the author's knowledge, there are no results on the $L^p$ boundedness of wave operators for higher-order Schr\"odinger operators $H = (-\Delta)^m + V$ on $\mathbb{R}^n$ when $1\leq n \leq 2m$ and $m \geq 3$, unlike the cases for $m = 1, 2$. More background and references can be found in Subsection \ref{subsec:1.3}.

Inspired by these findings, we turn our attention to the remaining cases. Notably, Erdo\u{g}an and Green \cite{EG22, EG23} established the $L^p$-boundedness of wave operators $W_\pm(H; (-\Delta)^m)$ for all $1 \leq p \leq \infty$ when $n > 2m$ in the zero regular case of $H$, and more recently, Erdo\u{g}an-Green-LaMarster \cite{EGL24} further considered the corresponding $L^p$-boundedness of $W_\pm(H; (-\Delta)^m)$ when zero is an eigenvalue of $H$ for $n > 4m$. Hence, the primary goal of this paper is to address the remaining cases. Specifically, we show that $W_\pm(H; (-\Delta)^m)$ are bounded on $L^p(\mathbb{R}^n)$, with a sharp range of $p$ depending on all possible types of zero-energy resonances of $H$ when $n$ is odd and $1 \leq n \leq 4m - 1$. The cases for even dimensions, $2 \leq n \leq 4m$, will be addressed in another work \cite{CWY24}.

The existing research works suggest that the $L^p$-behavior of wave operators can be broadly categorized into three scenarios: $n < 2m$, $n = 2m$, and $n > 2m$. When $n < 2m$, as first observed by Goldberg and Green \cite{GG21} for $(n, m) = (3, 2)$, the resolvent exhibits a singularity appearing in the stationary representation of the low-energy part of the wave operators. It is expected that wave operators are generically unbounded on $L^p$ for $p = 1, \infty$ in this scenario (cf. Mizutani-Wan-Yao \cite{MWY22, MWY23-1, MWY23-2} for $n = 1, 3$ and $m = 2$). 

On the other hand, when $n > 2m$, the singularity at zero energy is relatively mild, but the high-energy part becomes more complicated than in the case $n < 2m$ since the resolvent does not decay in the high-energy limit (cf. Erdo\u{g}an-Green \cite{EG22, EG23}). The case $n = 2m$ is critical, as it combines the difficulties of both the low-energy and the high-energy parts (see Galtbayar-Yajima \cite{GY23} for $(n, m) = (4, 2)$).

Finally, it is worth emphasizing that the $L^p$ range of wave operators is closely related to the zero resonance classification and the asymptotic expansions of the resolvent $R(z)$ of $H = (-\Delta)^m + V$ near the zero-energy threshold. Compared to the Schrödinger operator $-\Delta + V$ (i.e. $m = 1$) (cf. e.g. \cite{JK, J80, JN01}), the classification of resonances and the asymptotic expansions of $(-\Delta)^m + V$ are more challenging because the polynomial symbol $P(\xi) = |\xi|^{2m}$ of $(-\Delta)^m$ is degenerate at $\xi = 0$ for $m \geq 2$. Moreover, the degeneracy at zero energy increases as $m$ increases and $n$ remains smaller.

\subsection{The main results}
 In order to state our main results, we first need to introduce the definition of zero energy resonances for $H=(-\Delta)^{m}+V$. 
Recall that zero is an eigenvalue of $H$ if there exists some nonzero  $\phi\in L^2(\R^n)$ such that
\begin{equation}\label{equ1.2}
	((-\Delta)^{m}+V)\phi=0
\end{equation}
in distributional sense. However,  for some potentials,  it is well known that there may exist some non-trivial solutions satisfying \eqref{equ1.2} in the weighted spaces $$L_s^2(\R^n):=\{f\in L_{loc}^2(\mathbb{R}_n); \,\langle x\rangle ^sf\in L^2(\R^n)\}$$ for some $s<0$.  This is usually referred to as a zero resonance of $H$, which certainly leads to challenges in establishing dispersive estimates or  $L^p$-boundedness of wave operators associated with $H$.

Next, we introduce the concept of zero resonances for $H$. 

Let $m\ge 2$ and $n$ be an odd integer,  and set 
\begin{equation}\label{equ0.1}
	m_n:=\min\big\{m, \ 2m-\frac{n-1}{2} \big\}=
	\begin{cases}
		m,&\ \ \mbox{for}\ \  1\le n\le 2m-1,\\[0.3cm]
		\mbox{$2m-\frac{n-1}{2}$},&\ \ \mbox{for}\ \ 2m+1\le n\le 4m-1.
	\end{cases}
\end{equation}
Clearly, since $n$ is odd and $1\le n\le 4m-1$, $m_n$ is a positive integer that is  {\bf the total number of zero resonance types} of $H=(-\Delta)^{m}+V$ on $\R^n$ (see \cite{CHHZ-2024} and Subsection \ref{expansion} below for details).  

In particular, we observe that for all $1\leq n\leq 2m - 1$, the total number $m_n$ of zero resonance types is always a fixed value $m$. Moreover, when $2m + 1\leq n\leq 4m - 1$, $m_n$ is given by $m_n=2m-\frac{n - 1}{2}$, and it decreases as $n$ increases. Finally, until $n>4m$, the operator $H = (-\Delta)^{m}+V$ contains only a possible zero eigenvalue and has no zero resonance on $\mathbb{R}^n$ (cf. Feng et al. \cite{FSWY20}).

\begin{definition}\label{defi-resonance}
Let $1\le n\le 4m-1$ be odd and $m_n$ be the integer defined in \eqref{equ0.1}.  Given that $H = (-\Delta)^m + V$, where $V \in L^{\infty}(\mathbb{R}^n)$ is a real-valued potential. Then 
\vskip0.1cm
\indent(\romannumeral1) Zero is said to be the $\mathbf{k}$\text{-}th kind resonance of $H$ if $1\le \mathbf{k}\le m_n$ and the equation \eqref{equ1.2}  (i.e.  $H\phi=0$) has a non-trivial solution $\phi$ in $\bigcap_{s<-\frac{1}{2}-m_n+\mathbf{k}}L_{s}^2(\mathbb{R}^n)$, but doesn't have non-trivial solution $\phi$ in $\bigcap_{s<\frac{1}{2}-m_n+\mathbf{k}}L_{s}^2(\mathbb{R}^n)$.
\vskip0.1cm
\indent(\romannumeral2)  Zero is said to be the eigenvalue of $H$, if \eqref{equ1.2} has a non-trivial solution $\phi$ in $L^2(\R^n)$. 
\vskip0.1cm

\indent(\romannumeral3)  Zero is said to be the regular point of $H$, if zero is neither a resonance nor an eigenvalue. 

For convenience below, we also regard  {\it  the regular zero energy as  the $0$\text{-}th resonance of $H$ and the zero eigenvalue as the $(m_n+1)$\text{-}th kind resonance of $H$}. \end{definition}
\begin{remark} For all $n>2m$ and $m\ge 2$,  the zero resonance classification of $(-\Delta)^m+V$ was first found by  Feng et al. \cite{FSWY20}. Subsequently, the problem was further explored in \cite{EGT21, Green-Toprak-4D, LSY23, SWY22} for the cases where $n\leq 4$ and $m = 2$. These studies delved deeper into specific dimensions and orders of the operator. In particular, for all odd dimensions $n$ and $m\geq 2$, the zero resonance definition was introduced by \cite{CHHZ-2024}. The definition above  is equivalent to all previous ones (possibly in the different space).

It should be noted that zero resonance classification is closely linked to the asymptotic expansion of the resolvent $R(z)$ of $H=(-\Delta)^m+V$ near the zero threshold.  Compared with Schr\"odinger operator $-\Delta+V$ (i.e. $m=1$) (cf. \cite{JK, J80,JN01}),  as remarked before,  the zero resonance classification and the asymptotic expansions of $(-\Delta)^m+V$ are complicated, since $(-\Delta)^m$ is degenerate at zero energy when $m\ge 2$.  We refer readers to Theorem \ref{thm-M inverse-odd} for more details. 
\end{remark}

Throughout the paper, we make the following assumption.

\begin{assumption}\label{Assumption}
Let $m\ge 2$, $1\le n\le 4m-1$ be odd,  $0\le {\bf k}\le m_n+1$ and $H=(-\Delta)^m+V$ where $V$ is real valued. Additionally, we assume that the following three conditions hold:
\vskip0.1cm
\indent (\romannumeral1) Zero is the $\mathbf{k}$\text{-}th kind resonance of $H$ and $H$ has no embedded positive eigenvalues.
\vskip0.1cm
\indent (\romannumeral2) $\lvert V(x)\rvert\lesssim \langle x \rangle^{-\beta}$ for some $\beta$ satisfying
\begin{equation}\label{decayonV}
\beta>
\begin{cases}
4m-n+4\mathbf{k}+4, &\quad\text{if\ \ $1\le n\le 2m-1$},\\
n+4\mathbf{k}+5, &\quad\text{if\ \ $2m+1\le n\le 4m-1$}.
\end{cases}
\end{equation}

\indent (\romannumeral3) There exists a $\delta>0$  such that $\|\langle\cdot\rangle^{1+\delta}\langle\nabla \rangle ^{\delta} V(\cdot)\|_{L^2}<C$ when $n=4m-1$.
\end{assumption}
\begin{remark}
 We remark that the assumption about the absence of embedded positive eigenvalues is indispensable. The classical Kato theorem indicates that if $V(x)=o(|x|^{-1})$ as $|x|\to \infty$, $-\Delta +V$ has no positive eigenvalues. Compared to the second-order case,  there exists a $V$ in $C_0^{\infty}(\mathbb{R}^n)$ such that $H=(-\Delta)^m+V$ has a positive eigenvalue when $m$ is even (cf. Feng et al. \cite[Section 7.1]{FSWY20}). 

 Assumption (iii) is a smoothness condition specifically used to estimate the high-energy part, as in Erdo\u{g}an-Green \cite{EG22} when $n=4m-1$. Moreover, we note that the smoothness assumption of $V$ for $H$ is necessary when $n\ge4m$ by Erdo\u{g}an-Goldberg-Green\cite{EGG-JFA-2023}.
\end{remark}

Let $\mathbb{B}(X,Y)$ be the space of bounded operators from the Banach space $X$ to the Banach space $Y$, and $\mathbb{B}(X)=\mathbb{B}(X,X)$. Additionally, denote 
\begin{equation}\label{defkc}
k_c:=\max\{m-\frac {n-1}{2}, 0\}=\begin{cases}
m-\frac{n-1}{2},\ \ &\ \text{if}\ \ 1\le n\le2m-1,\\
0,\ \ & \ \text{if}\ \ 2m+1\le n\le 4m-1.
\end{cases}
\end{equation}
Note that when $1\leq n\leq 2m - 1$ is odd, $k_c$ is a positive integer, and $1\leq k_c\leq m_n$. 

In addition, as shown in the following theorem, $k_c$ represents the type of critical resonance of $H$. When $\k > k_c$, the $L^p$-boundedness range of the wave operators $W_\pm\big(H;(-\Delta)^m\big)$ shrinks from the full range $1 < p < \infty$ (excluding the endpoints $p = 1,\infty$) into a finite interval that depends on the resonance type $\k$.

\begin{theorem}\label{thm-main result}
Let $1\leq n\leq 4m - 1$ be odd, $0\leq \mathbf{k}\leq m_n + 1$, and let $H = (-\Delta)^m+V$ satisfy Assumption \ref{Assumption} (depending on the type $\mathbf{k}$ of zero resonance). Then the following statements hold:   

\vspace{5pt}
\indent\emph{(\romannumeral1)} If zero is the regular point of $H$  ( i.e. $\mathbf{k}=0$ ), then  $W_\pm\big(H;(-\Delta)^m\big)\in \mathbb{B}(L^p(\mathbb{R}^n))$ for all $1<p<\infty$;

\vspace{5pt}
\indent\emph{(\romannumeral2)} If zero is the $\bf{k}$-th kind resonance of $H$ and  $1\le \mathbf{k}\le k_c$, then $ W_\pm\big(H;(-\Delta)^m\big)\in \mathbb{B}(L^p(\mathbb{R}^n))$ for all $1<p<\infty;$

\vspace{5pt}
\indent\emph{(\romannumeral3)} If   zero is the $\bf{k}$-th kind resonance of  $H$ and  $k_c+1\le \mathbf{k}\le m_n$, then  $W_\pm\big(H;(-\Delta)^m\big)\in \mathbb{B}(L^p(\mathbb{R}^n))$ for all $1<p<\frac{n}{n-2m+\mathbf{k}+k_c-1}$;

\vspace{5pt}
\indent(\romannumeral4) If  zero is an eigenvalue  of $H$ ( i.e. $\mathbf{k}=m_n+1$ ), then $ W_\pm\big(H;(-\Delta)^m\big)\in \mathbb{B}(L^p(\mathbb{R}^n))$ for all $1<p<\frac{2n}{n-1}.$
\end{theorem}

\begin{remark}
Some comments are given as follows:
\begin{itemize}
    \item In Theorem \ref{thm-main result} (i) (i.e., the regular case with \(k = 0\)), if \(2m + 1\leq n\leq 4m - 1\) is odd, the wave operators \(W_\pm\big(H;(-\Delta)^m\big)\) are also bounded on \(L^p(\mathbb{R}^n)\) for the endpoints \(p = 1\) and \(p=\infty\) (cf. Erdoğan and Green \cite{EG22}).
    \vspace{4pt}
    \item When \(1\leq n\leq 2m - 1\), as first observed by Goldberg and Green \cite{GG21} for \((n, m) = (3, 2)\), the resolvent exhibits a singularity at zero energy. This singularity leads to the appearance of the Hilbert transform in the stationary representation of the low-energy part of the wave operators. The presence of the Hilbert transform and the zero-energy singularity causes difficulties in ensuring the boundedness of the wave operators. Hence, it is expected that, in general, wave operators are unbounded on \(L^p\) for \(p = 1, \infty\) when \( 1\leq n\leq2m-1\) (cf. Mizutani-Wan-Yao \cite{MWY22, MWY23-1, MWY23-2} for counterexamples in the cases \(n = 1, 3\) and \(m = 2\)). 
    \vspace{4pt}
    \item Moreover, in Theorem \ref{thm-main result} (ii), clearly \(k_c\geq 1\) is necessary, which requires that \(1\leq n\leq 2m - 1\) according to \eqref{defkc}. Hence, it follows from Theorem \ref{thm-main result} (iii) and (iv) that the range of \(p\) for which the \(L^p\)-boundedness of \(W_\pm\big(H;(-\Delta)^m\big)\) holds will be a finite interval if any kind of zero resonance or eigenvalue occurs in the cases of \(2m + 1\leq n\leq 4m - 1\).
\end{itemize}
\end{remark}

In the following result, we also show that the interval of $p$ in the above theorem is sharp, except for the endpoints when $\k>k_c$. 
\begin{theorem}\label{thm-unbound}
 Assume that the assumption of Theorem \ref{thm-main result} holds. Then the following holds:
 
 \vspace{5pt}
\indent\emph{(\romannumeral1)}
 \(W_\pm(H; (-\Delta)^m) \notin \mathbb{B}(L^p(\mathbb{R}^n))\) for all \(\frac{n}{n - 2m + \mathbf{k} + k_c - 1} < p \leq \infty\) if  \(k_c<\mathbf{k}\leq m_n\).

\vspace{5pt}
\indent\emph{(\romannumeral2)}
 \(W_\pm(H; (-\Delta)^m) \notin \mathbb{B}(L^p(\mathbb{R}^n))\) for all \(\frac{2n}{n - 1}< p \leq \infty\) if zero is an eigenvalue of \(H\) with a non-zero distributional solution of \(H\phi = 0\) in \(\big(\bigcap_{s < -\frac{1}{2}}L^{2}_{s}(\R^n)\big)\setminus L^2(\R^n)\).
\end{theorem}

\begin{remark} $\mbox{}$ \vspace{0.5pt}

\begin{itemize}
\item If zero is an eigenvalue of \(H\),  but there is no nonzero solution \(\phi\) to \(H\phi = 0\) in \(\bigcap_{s < -\frac{1}{2}}L^{2}_{s}(\mathbb{R}^n)\setminus L^2(\mathbb{R}^n)\), then the wave operators \(W_\pm\big(H;(-\Delta)^m\big)\) may be bounded on \(L^p\) with an interval of \(p\) values that strictly contains \((1,\, \frac{2n}{n - 1})\). 
\item For the specific case of \(m = 2\) and \(n = 3\), Mizutani-Wan-Yao in \cite{MWY23-1} call zero a \(p\)-wave resonance if there exists \(\phi\in \bigcap_{s < -\frac{1}{2}}L^{2}_{s}(\mathbb{R}^3)\setminus L^2(\mathbb{R}^3)\) such that \(H\phi = 0\). 
In particular, they have shown that \(W_\pm\big(H;(-\Delta)^2\big)\) are bounded on \(L^p(\mathbb{R}^3)\) for all \(1 < p< \infty\) if zero is an eigenvalue of \(H\), but not a \(p\)-wave resonance, and all zero eigenfunctions \(\phi\) of \(H\) satisfy
\begin{align}
\label{orthogonal}
\int_{\mathbb{R}^3}x_ix_jx_k V(x)\phi(x)dx = 0
\end{align}
for all \(i,j,k = 1,2,3\).
\item 
    We prove Theorem \ref{thm-unbound} by contradiction in Section \ref{sec-unb}. The key idea involves reducing the problem to establishing the optimality of time-decay estimates for $e^{itH}$ in weighted $L^2$ spaces (known as {\it  Kato-Jensen-type estimates}). These optimal decay estimates, in turn, rely on a detailed asymptotic expansion of the resolvent $R^\pm(\lambda^{2m})$ of $H = (-\Delta)^m + V$ near the zero threshold, accounting for the presence of zero resonances or eigenvalue singularities.  

\end{itemize}
\end{remark}

\begin{remark}

To provide a comprehensive overview of the \(L^p\)-boundedness of the wave operators $W_\pm\big(H;(-\Delta)^m\big)$, in Table \ref{table} below, we list the known $L^p$-boundedness results of the wave operators $W_\pm(H;(-\Delta)^m)$ when $m\geq1$ and $1\leq n\leq 4m - 1$ is odd. Our results agree with the existing ones for $m = 1,2$.
\end{remark} 
\begin{table}[h]
\renewcommand\arraystretch{2}
\centering
\begin{tabular}{|p{2cm}<{\centering}|p{4.3cm}<{\centering}|p{9cm}<{\centering}|}
 \hline 
  {\bf Operators}      &  {$\bf n$,\  $\bf m_n$ {\bf and} $\bf{k}$}     & {\bf $L^p$-bounds of  $W_\pm$ with zero resonance type ${\bf k }$} \\
  \hline
  \multirow{2}{*}{\makecell{$-\Delta+V$\\ $m=1$}}      &  \makecell{$n=1,\ m_n=1$  \vspace{0.05cm}\\ $0\le \mathbf{k} \le 2$  \vspace{0.05cm}\\no zero eigen.(${\bf k}=2$)}   \vspace{0.05cm}&  \makecell{$\mathbf{k}=0,1;$\ \ $1<p<\infty$, \vspace{0.05cm}\\
Weder \cite{Weder-99-CMP}, Galtbayar-Yajima \cite{Gal-Yaj-00},   \vspace{0.05cm}\\ and D'Ancona-Faneli \cite{AF06}.}  \vspace{0.05cm}\\
  \cline{2-3}
  & \makecell{$n=3,\ m_n=1$  \vspace{0.05cm}\\ $0\le {\bf k}\le 2$}   \vspace{0.05cm}& \makecell{$\mathbf{k}=0$:\ \ $1\le p\le \infty$,\vspace{0.05cm}\\  Yajima \cite{Yajima-JMSJ-95}, Beceanu-Schlag \cite{Beceanu-Schlag-20};   \vspace{0.05cm}\\ $\mathbf{k}=1,2$:\ \ $1<p<3$,  Yajima \cite{Yajima-2016-3d}.}  \vspace{0.05cm}\\
  \hline
  \multirow{3}{*}{\makecell{$\Delta^2+V$\\ $m=2$}}      & \makecell{$n=1,\ m_n=2$  \vspace{0.05cm}\\ $0\le \k\le 3$  \vspace{0.05cm}\\no zero eigen. ($\k=3$)}   \vspace{0.05cm}&  \makecell{$\mathbf{k}=0,1,2$:\ \ $1<p<\infty$,\vspace{0.05cm}\\
Mizutani-Wan-Yao\cite{MWY22}.}\\
  \cline{2-3}
  &\makecell{ $n=3,\ m_n=2$  \vspace{0.05cm}\\ $0\le \k\le 3$  
  } & \makecell{$\mathbf{k}=0$:\ \ $1<p<\infty$  \vspace{0.05cm}\\  Goldberg-Green \cite{GG21},  Mizutani-Wan-Yao \cite{MWY23-2};  \vspace{0.05cm}\\
  $\mathbf{k}=1$:\ \ $1<p<\infty$ and  $\mathbf{k}=2,3$:\ \ $1<p<3$,
  \\ Mizutani-Wan-Yao \cite{MWY23-1}.}  \vspace{0.05cm}\\
  \cline{2-3}
  & \makecell{$n=5,7$  \vspace{0.05cm}\\  $0\le \k \le m_n+1$  \vspace{0.05cm}} & \makecell{$\mathbf{k}=0$:\ \ $1\le p\le\infty$, Erdo\u{g}an-Green \cite{EG22};  \vspace{0.05cm}\\   $1\le \k \le m_n+1$,  Thm \ref{thm-main result}.
  }  \\
  \hline
   \multirow{2}{*}{\makecell{$(-\Delta)^m+V$\\ $m\ge3$}}     & \makecell{$1\le n\le 2m-1$   \vspace{0.05cm}\\ $k_c=m-\frac{n}{2}+\frac{1}{2}$\\ $m_n=m$   \vspace{0.05cm}
   \\ $0\le \k \le m_n+1$} & \makecell{$0\le \k\le k_c:\ 1<p<\infty,$\ \  Thm \ref{thm-main result} (i),(ii);  \vspace{0.08cm}\\ $k_c< \k\le m_n:\ 1<p<\frac{n}{n-2m+\mathbf{k}+k_c-1},$\\ Theorem \ref{thm-main result} (iii);  \vspace{0.08cm}\\ $k=m_n+1:\ 1<p<\frac{2n}{n-1}$,\ \ Thm \ref{thm-main result} (iv).}  \\
  \cline{2-3}
  & \makecell{$2m+1\le n\le 4m-1$  \vspace{0.08cm}\\ $k_c=0$\vspace{0.08cm}\\ $m_n=2m-\frac{n-1}{2}$   \vspace{0.08cm}\\ $0\le \k \le m_n+1$}&  \makecell{\\$\mathbf{k}=0$:\ \ $1\le p\le\infty$, \vspace{0.08cm}\\ Erdo\u{g}an-Green \cite{EG22, EG23};   \vspace{0.05cm}\\
  $1\le\k\le m_n:\ 1<p<\frac{n}{n-2m+\mathbf{k}-1},$\\ Theorem \ref{thm-main result} (iii);  \vspace{0.08cm}\\ $\k=m_n+1:\ 1<p<\frac{2n}{n-1}$,\ \ Thm \ref{thm-main result} (iv).} \\
  \cline{2-3}
  \hline
\end{tabular}
\vskip 0.4cm
\caption{
The $L^p$-boundedness for $W_\pm\big(H;(-\Delta)^m\big)$ when $1\le n\le 4m-1$ is odd}\label{table}
\end{table}

\subsubsection{Some applications}
In this subsection, we give a quick application to higher-order Schr\"odinger group $e^{itH}P_{ac}(H)$. Recall that as the following dispersive estimates for the free higher-order Schr\"{o}dinger group $e^{it(-\Delta)^m}$ hold  (see e.g. \cite{ZYF}) 
\begin{equation}\label{eq: for free}
\left\|e^{i t(-\Delta)^m}\right\|_{L^p-L^{p'}}\lesssim |t|^{-\frac{n}{m}(\frac{1}{p}-\frac {1}{2})}, \quad t\neq0,
\end{equation}
where $1\le p\le2$ and $\frac{1}{p'}+\frac{1}{p}=1$.  Then by the $L^p$-boundedness of  $W_\pm(H;-\Delta)^m$  in Theorem \ref{thm-main result} and intertwining formula \eqref{eq-intertwing identity} of wave operators, we can immediately  obtain the following corollary.

\begin{corollary}
Let $1\le n\le 4m-1$ be odd. Under Assumption \ref{Assumption} with resonance type $\mathbf{k}$,  one has 
\begin{equation}\label{decayest}
\Big\|e^{i tH}P_{ac}(H)\Big\|_{L^p-L^{p'}}\lesssim |t|^{-\frac{n}{m}(\frac{1}{p}-\frac 12)}, \quad t\neq0,
\end{equation}
where  $\frac{1}{p'}+\frac{1}{p}=1$ and the exponent $p'$  satisfies that
\begin{equation*}
\begin{cases}
2\le p'< \infty, &\quad \text{if}\ \ ~ 0\le \mathbf{k} \le k_c,\\
2\le p'< \frac{n}{n-2m+\mathbf{k}+k_c-1}, &\quad \text{if}\ \ ~  k_c+1\le \mathbf{k} \le m_n,\\
2\le p'< \frac{2n}{n-1}, &\quad \text{if}\ \ ~\mathbf{k}= m_n+1.
\end{cases}
\end{equation*}
\end{corollary}
\vskip0.2cm
\begin{remark}

As is well known, the decay estimate \eqref{decayest} for the Schr\"odinger operator $-\Delta + V$ ( i.e.  $m = 1$) has been an active research topic over the past three decades and has found broad applications in nonlinear Schr\"odinger equations. In particular, $L^1-L^\infty$ decay estimates have been extensively studied.  
For $n \geq 3$, Journé, Soffer, and Sogge \cite{JSS} first established the $L^1-L^\infty$ dispersive estimate of Schr\"odinger operator in the regular case using methods distinct from wave operator techniques. Subsequent contributions for lower dimensions were made by Weder \cite{Weder1}, Rodnianski and Schlag \cite{RodSchl}, Goldberg and Schlag \cite{Goldberg-Schlag}, and Schlag \cite{Schlag-CMP}. Further developments and references can be found in Schlag’s survey work \cite{Schlag}.  

Moreover, it is worth noting that there exist many recent advances in the study of dispersive estimates for $ e^{itH}$ in the context of higher-order Schrödinger operators $H = (-\Delta)^m + V$ with $m \geq 2$. For the case $m = 2$, Kato-Jensen estimates or $ L^1-L^\infty$ estimates for $e^{itH}$ have been established in works such as \cite{EGT21, FSY, Green-Toprak-4D, SWY22, LSY23}. For general $m \geq 2$, further results can be found in \cite{CHHZ-2024, EGG_ArXiv, FSWY20}.

\end{remark}

\subsection{Further related backgrounds.}\label{subsec:1.3}
 In this subsection, we review known results on the $L^p$ boundedness of wave operators for both the Schr\"odinger operator and higher-order elliptic operators.

For the classical Schr\"odinger operator \( H = -\Delta + V(x) \), due to its numerous applications in linear and nonlinear dispersive equations, there has been significant work on the $L^p$ boundedness of wave operators \( W_\pm(H; -\Delta) \). If zero energy is regular, Yajima \cite{Yajima-JMSJ-95,Yajima-1995-even} was the first to prove that, under certain smoothness and decay conditions on the potential \( V \), the wave operators \( W_\pm(H; -\Delta) \) are bounded on \( L^p(\mathbb{R}^n) \) for all \( 1 \leq p \leq \infty \) in dimensions \( n \geq 3 \). Using the Wiener algebra method, Beceanu \cite{Beceanu} derived sharp conditions on the potential for \( n = 3 \) (see also Beceanu-Schlag \cite{Beceanu-Schlag-20}). For \( n = 1, 2 \) with regular zero energy, it has been shown that the wave operators are bounded on \( L^p(\mathbb{R}^n) \) for all \( 1 < p < \infty \) (see \cite{Weder-99-CMP, Gal-Yaj-00, AF06, Yajima-99-CMP, Jensen-Yajima-02-CMP}). Notably, for \( n = 1 \), the range of \( p \) remains unchanged even when zero energy is not regular (cf. \cite{Weder-99-CMP}).

In addition, when zero energy is a resonance or eigenvalue of \( H = -\Delta + V \), the range of exponent \( p \) for which the wave operators \( W_\pm(H; -\Delta) \) are bounded on \( L^p(\mathbb{R}^n) \) may shrink. For \( n > 4 \), it was shown in \cite{Finco-Yajima-2006-even,Yajima-JMSJ-2006} that the wave operators are bounded on \( L^p(\mathbb{R}^n) \) for \( n/(n-2) < p < n/2 \). Yajima \cite{Yajima-2016} and Goldberg-Green \cite{Goldberg-Green-2016} independently extended this range to \( 1 < p < n \) (with \( p = 1 \) also covered in \cite{Goldberg-Green-2016}). For \( n = 4 \), Goldberg-Green \cite{Goldberg-Green-2017} showed that the wave operators are bounded on \( L^p(\mathbb{R}^4) \) for \( 1 \leq p < 4 \) if zero is an eigenvalue but no resonance is present, while Yajima \cite{Yajima-2022-4} showed that the range is \( 1 < p < 4 \) if there is a zero resonance.

For \( n = 3 \), Yajima \cite{Yajima-2016-3d} demonstrated that the wave operators are bounded on \( L^p(\mathbb{R}^3) \) for \( 1 \leq p < 3 \) if zero energy is an eigenvalue, and for \( 1 < p < 3 \) if zero energy is a resonance. For \( n = 2 \), Erdo\u{g}an-Goldberg-Green \cite{Erdogan-Goldberg-Green-JFA-2018} established that  wave operators are bounded on \( L^p(\mathbb{R}^2) \) for \( 1 < p < \infty \) if there is an \( s \)-wave resonance or eigenvalue at zero. Yajima \cite{Yajima-2022} showed that if there is a \( p \)-wave resonance at zero, the wave operators are bounded on \( L^p(\mathbb{R}^2) \) for \( 1 < p \leq 2 \) but unbounded for \( 2 < p < \infty \).

It is natural to extend the study to higher-order Schr\"odinger operators \( H = (-\Delta)^m + V \) with \( m \geq 2 \). In fact, based on recent developments in the resolvent \( (H - z)^{-1} \) and decay estimates for \( e^{-itH} \) (see, e.g. \cite{EGT21, FSWY20}), Goldberg-Green \cite{GG21} first proved that for \( n = 3 \) and \( m = 2 \), the wave operators \( W_\pm(H; \Delta^2) \) are bounded on \( L^p(\mathbb{R}^3) \) for \( 1 < p < \infty \) under the assumption that the zero energy is regular. This result was later extended to higher-order Schr\"odinger operators \( H = (-\Delta)^m + V \) on \( \mathbb{R}^n \) for \( n > 2m \) in the regular case by Erdo\u{g}an-Green \cite{EG22, EG23}.

Recently, significant progress has been made in understanding the $L^p$ boundedness of wave operators for the fourth-order Schr\"odinger operator \( \Delta^2 + V(x) \). For \( n = 1 \), Mizutani-Wan-Yao \cite{MWY22} proved that the wave operators \( W_\pm(H; \Delta^2) \) are bounded on \( L^p(\mathbb{R}) \) for all \( 1 < p < \infty \), while being unbounded at \( p = 1 \) and \( p = \infty \), regardless of whether the zero energy is a regular point or resonance.

For \( n = 3 \), Mizutani-Wan-Yao \cite{MWY23-2} showed that the wave operators \( W_\pm(H; \Delta^2) \) are unbounded on \( L^1(\mathbb{R}^3) \) and \( L^\infty(\mathbb{R}^3) \) and established weak \( (1,1) \) estimates in the regular case. In \cite{MWY23-1}, they further proved that the range of \( p \) is \( 1 < p < \infty \) if zero is a first-kind resonance, but \( 1 < p < 3 \) if zero is a second- or third-kind resonance. For \( n = 4 \), Galtbayar-Yajima \cite{GY23} demonstrated that the wave operators \( W_\pm(H; \Delta^2) \) are bounded on \( L^p(\mathbb{R}^4) \) for \( 1 < p < \infty \) if zero energy is regular or corresponds to only \( s \)-wave resonance. The range of \( p \) is \( 1 < p < 4 \) for cases involving \( s \)- and \( p \)-wave resonances and zero energy eigenfunctions, and \( 1 < p \leq 2 \) if a \( d \)-wave resonance occurs.

In this paper, as well as in a forthcoming sequel \cite{CWY24}, we will show that \(W_\pm(H; (-\Delta)^m)\) are bounded on \(L^p(\mathbb{R}^n)\), where the sharp range of the exponent \(p\) is determined by all possible types of zero energy resonance of \(H\), for \(1 \leq n \leq 4m\).

\subsection{The outline of the proof}
 In what follows, we briefly explain the ideas behind the proofs of the main theorem. The usual starting point is the following stationary representation of wave operators (see e.g. Kuroda \cite{Kuroda}):
\begin{equation}\label{eq-stationary formula}
W_{\pm}(H,(-\Delta)^m)=I-\frac{m}{\pi i}\int_0^{+\infty}\lambda^{2m-1}R^{\mp}(\lambda^{2m})V\big(R_0^+(\lambda^{2m})-R_0^-(\lambda^{2m})\big)d\lambda,
\end{equation}
where $R^{\pm}(\lambda^{2m}),\ R_0^\pm(\lambda^{2m})$ are the limiting resolvent operators defined by
$$R^{\pm}(\lambda^{2m}):=\lim\limits_{\varepsilon\downarrow0}R(\lambda^{2m}\pm i\varepsilon)=\lim\limits_{\varepsilon\downarrow0}\big(H-(\lambda^{2m}\pm i\varepsilon)\big)^{-1},$$
$$R^\pm_0(\lambda^{2m}):=\lim\limits_{\varepsilon\downarrow0}R_0(\lambda^{2m}\pm i\varepsilon)=\lim\limits_{\varepsilon\downarrow0}\big((-\Delta)^m-(\lambda^{2m}\pm i\varepsilon)\big)^{-1}.$$
We shall only prove the $L^p$-boundedness for $W_{-}$ here, since $ \overline{W_+f}=W_-\overline{f}$.
Because the identity operator is obviously bounded on $L^p(\mathbb{R}^n)$ for all $1\le p\le \infty$, it suffices to consider the second term involving the integral  in \eqref{eq-stationary formula}. For convenience, we denote this integral term by $W$.  Let $\chi\in C_0^\infty(\mathbb{R}^n)$ such that $\chi\equiv1$ in $(-\lambda_0/2,\,\lambda_0/2)$ and $\text{supp}\chi\subset[-3\lambda_0/4,\,3\lambda_0/4]$, and the fixed $\lambda_0$ is given in Theorem \ref{thm-M inverse-odd}. We then use the functions $\chi$ and $1-\chi$ to decompose the  $W$ into the low and high energy parts as follows: 
\begin{equation}\label{eq-wave operator-low-0}
W^L=\frac{m}{\pi i}\int_0^\infty\lambda^{2m-1}\chi(\lambda)R^+(\lambda^{2m})V\big(R_0^+(\lambda^{2m})-R_0^-(\lambda^{2m})\big)d\lambda,
\end{equation}
\begin{equation}\label{eq-wave operator-high-0}
W^H=\frac{m}{\pi i}\int_0^\infty\lambda^{2m-1}(1-\chi(\lambda))R^+(\lambda^{2m})V\big(R_0^+(\lambda^{2m})-R_0^-(\lambda^{2m})\big)d\lambda.
\end{equation}

\subsubsection{The low energy part $W^{L}$} 
We first sketch the proof for the low energy part \(W^{L}\).
We define \(v(x)=\sqrt{|V(x)|}\) and \(U(x)\) as the sign function of \(V(x)\), denoted by \(\text{sgn} V(x)\). Here, \(\text{sgn} V(x)\) represents the sign function of \(V(x)\), which equals \(1\) when \(V(x)\ge 0\), and \(-1\) when \(V(x)<0\). And then we define
\begin{equation} \label{defM}
M^\pm(\lambda)=U+vR_0^\pm(\lambda^{2m})v.	
\end{equation}
Due to the absence of embedded positive eigenvalue of \(H=(-\Delta)^m+V\), it is well-known that  \(M^\pm(\lambda)\) is invertible on \(L^2(\mathbb{R}^n)\) for all \(\lambda>0\) (see e.g. Agmon \cite{Agmon} and Kuroda \cite{Kuroda}), and  one has the symmetric second resolvent identity
\begin{equation}\label{eq: reso-identity}
R^\pm(\lambda^{2m})V=R_0^\pm(\lambda^{2m})v(M^\pm(\lambda))^{-1}v.	
\end{equation}
Using the above identity,  \(W^L\) can be expressed as
\begin{equation}\label{eq-wave operator-low-1}
	W^L=\frac{m}{\pi i}\int_0^\infty\lambda^{2m-1}\chi(\lambda)R_0^+(\lambda^{2m})v(M^+(\lambda))^{-1}v\big(R_0^+(\lambda^{2m})-R_0^-(\lambda^{2m})\big)d\lambda.
\end{equation}	
The main technical challenge is to understand, as \(\lambda\rightarrow 0^+\), the refined integral kernel structure of  the following composed operator:
$$R_0^+(\lambda^{2m})v(M^+(\lambda))^{-1}v\big(R_0^+(\lambda^{2m})-R_0^-(\lambda^{2m})\big).$$

One of the key tools in our proof is the asymptotic expansions of $(M^\pm(\lambda))^{-1}$ recently derived in \cite{CHHZ-2024}, which are presented in the following form as $\lambda\rightarrow0^+$ (also see Theorem \ref{thm-M inverse-odd}):
\begin{equation*}
	\big(M^\pm(\lambda)\big)^{-1}=\sum_{i,j\in J_\mathbf{k}}\lambda^{2m - n - i - j}Q_i(M_{i,j}^\pm+\Gamma_{i,j}^\pm(\lambda))Q_j,
\end{equation*}
where $M^\pm_{i,j}\in\mathbb{B}(L^2)$ are $\lambda$-independent operators and $\Gamma_{i,j}^\pm(\lambda)$ are $\lambda$-dependent operators, and the index set $J_{\mathbf{k}}$ is dependent on zero energy resonance type $\mathbf{k}$. 

Using this asymptotic expansion, the integral kernel $W^L(x,y)$ is a sum of finite such terms $W_{i,j}^L(x,y)$, which can be expressed as
\begin{equation}\label{eq-1}
\int_0^\infty\lambda^{4m - n - 1 - i - j}\chi(\lambda)\Big\langle \Big(M_{i,j}^++\Gamma_{i,j}^+(\lambda)\Big)Q_jv\Big(R_0^{+}-R_0^{-}\Big)(\lambda^{2m})(\cdot,y),\, Q_ivR_0^{-}(\lambda^{2m})(\cdot,x)\Big\rangle _{L^2}d\lambda,
\end{equation}
where $i, j \in J_{\mathbf{k}}$, and $\langle \cdot, \cdot\rangle_{L^2}$ denotes the inner product on $L^2(\mathbb{R}^n)$.

The novel and key step is that we can write $Q_jvR_0^{-}(\lambda^{2m})(\cdot,y)$ and $Q_jv (R_0^{+}-R_0^{-})(\lambda^{2m})(\cdot,y)$ as
\begin{equation*}
	Q_jvR_0^-(\lambda^{2m})(\cdot,\, y)=e^{-i\lambda |y|}\omega_{j}^{-}(\lambda,\cdot,y)=e^{- i\lambda |y|}\big(\omega_{j,1}^-(\lambda,\cdot,y)+\omega_{j,2}^-(\lambda,\cdot,y)+\omega_{j,3}^-(\lambda,\cdot,y)\big),
\end{equation*}
and
\begin{equation*}
	Q_jv(R_0^+-R_0^-)(\lambda^{2m})(\cdot,\,y)=e^{ i\lambda |y|}\tilde{\omega}_{j}^{+}(\lambda,\cdot,y)+e^{-i\lambda |y|}\tilde{\omega}_{j}^{-}(\lambda,\cdot,y),
\end{equation*}
where $\omega_{j}^{\pm}$, $\omega_{j,1}^{\pm}$, $\omega_{j,2}^{\pm}$, $\omega_{j,3}^{\pm}$ and $\tilde{\omega}_{j}^{\pm}$ satisfy some estimates dependent on the property of $Q_j$, as detailed in Lemma \ref{lemma-QjvR-odd} and Lemma \ref{lem3.10}.

By putting these expressions into integral \eqref{eq-1}, the kernels $W_{i,j}^L(x,y)$ of  $W_{i,j}^L$ can be written as 
\begin{equation*}
W_{i,j}^L(x,y)
=\int_0^\infty \Big(e^{i\lambda(|x|+|y|)}T_{i,j}^+(\lambda,x,y)+e^{i\lambda(|x|-|y|)}T_{i,j}^-(\lambda,x,y)\Big)\chi(\lambda)d\lambda,
\end{equation*}
where
\begin{equation*}
	T_{i,j}^{\pm}(\lambda,x,y)=\lambda^{4m-n-1-i-j}\Big\langle \big(M_{i,j}^++\Gamma_{i,j}^+(\lambda)\big)\widetilde{\omega}_{j}^{\pm}(\lambda,\cdot,y),\, \omega_{i}^{-}(\lambda,\cdot,x)\Big\rangle_{L^2}.
\end{equation*}

We shall demonstrate the $L^p$-boundedness of $W_{i,j}^L$ by analyzing its integral kernel for three cases:
$$|x|> 2|y|,\,\, |x|< \frac12|y|\, \,\,  {\rm and}\,\,\,  \frac12|y|\le |x|\le 2|y|.$$ The corresponding operators for these cases are $W_{i,j,1}^L$, $W_{i,j,2}^L$ and $W_{i,j,3}^L$, respectively.

(I) \underline{For the first two cases, we write $W_{i,j,1}^L(x,y)$ and $W_{i,j,2}^L(x,y)$} as
$$\mathbf{1}_{\{|x|> 2|y|\}}\int_0^\infty\chi(\lambda) \Big(e^{i\lambda(|x|+|y|)}T_{i,j}^+(\lambda,x,y)+e^{i\lambda(|x|-|y|)}T_{i,j}^-(\lambda,x,y)\Big)d\lambda,$$
$$\mathbf{1}_{\{|x|< \frac12|y|\}}\int_0^\infty\chi(\lambda) \Big(e^{i\lambda(|x|+|y|)}T_{i,j}^+(\lambda,x,y)+e^{i\lambda(|x|-|y|)}T_{i,j}^-(\lambda,x,y)\Big)d\lambda.$$
Here $T_{i,j}^\pm(\lambda;x,y)$ satisfy some pointwise estimates (see Lemma \ref{lemma-T-odd1} and Lemma \ref{lemma-T-odd2}).  Hence, using a result from the oscillatory integral estimate, we obtain that $|W_{i,j,1}^L(x,y)|\lesssim\langle x\rangle^{-n}$ in the case $|x|> 2|y|$. This implies that $W_{i,j,1}^L$ is in $\mathbb{B}(L^p(\mathbb{R}^n))$ for $1< p \le \infty$ (see Proposition \ref{propo-low energy-odd-1}). While in the case $|x|< \frac12|y|$, it follows that $W_{i,j,2}^L(x,y)$ satisfies $|W_{i,j,2}^L(x,y)|\lesssim\langle x\rangle^{-\rho_i}\langle y\rangle^{-n+\rho_i}$ for some  $\rho_i$ depending on $i\in J_{\k}$. Consequently, we establish that $W_{i,j,2}^L\in \mathbb{B}(L^p(\mathbb{R}^n))$ for $1\le p < \sigma_i$ ($\sigma_i$ depends on $i\in J_{\k}$) in Proposition \ref{propo-low energy-odd-2}. The restrictions on the range of $p$ in Theorem \ref{thm-main result} are based on the  results established in Propositions \ref{propo-low energy-odd-2} and \ref{propo-low energy-odd-2'} .

 (II) \underline{For the last term $W_{i,j,3}^L$ in the low energy part}, we divide $W_{i,j,3}^L$ into three parts $I_{i,j}^{\pm,b}$, $I_{i,j,1}^{\pm,g}$ and $I_{i,j,2}^{\pm,g}$, where the first is the  bad part of $W_{i,j,3}^L$, and the last two are good parts. 
We can obtain the $L^p$-boundedness of {\bf the good parts} $I_{i,j,1}^{\pm,g}$ and $I_{i,j,2}^{\pm,g}$ for all $1\le p\le \infty$  by a similar analysis to that for $W_{i,j,1}^L$ or $W_{i,j,2}^L$ (see Proposition \ref{lem.estforwij3g}). 

To estimate  {\bf the bad part}  $I_{i,j}^{\pm,b}$, we apply the dyadic decomposition to rule out an admissible error term $e(x,y)$ as follows:
\begin{equation*}
I_{i,j}^{\pm,b}(x,y)=\sum_{|s-h|<2}\chi_s(x)I_{i,j}^{\pm,b}(x,y)\chi_h(y)+ e(x,y),
\end{equation*}
where  $\chi_0(x)=\mathbf{1}_{\{|x|\le 1 \}}$ and $\chi_h(x)=\mathbf{1}_{\{2^{h-1}<|x|\le2^{h} \}}$ for $h\in \N^{+} $ on $\mathbb{R}^n$ ( see \eqref{sumofw3b}).  Therefore, to show that $I_{i,j}^{\pm,b}$ is bounded on $L^p(\mathbb{R}^n)$,  it suffices to prove that 
\begin{equation*}
\left\| \chi_s I_{i,j}^{\pm,b}\chi_h \right\|_{\mathbb{B}(L^p(\mathbb{R}^n))}\leqslant C_p	
\end{equation*}
holds uniformly for $s,h\in\mathbb{N}_0$ such that $|s-h|<2$. 

It can be further shown that the kernel $\chi_l(x) I_{i,j}^{\pm,b}(x,y)\chi_h(y)$ can be split into an admissible part and a singular part. The singular part is a sum of 
\begin{equation*}
C \chi_s(x)\chi_h(y)\frac{y^\alpha}{|y|^{\frac{n-1}{2}+\vartheta(j)}}	\frac{x^{\tilde{\alpha}}}{|x|^{\frac{n-1}{2}+\vartheta(i)}}\frac{1}{|x|\pm|y|},
\end{equation*}
with $|\alpha|=|\tilde{\alpha}|=\vartheta(j)$. Finally, we can reduce the singular part operator to the $L^p$-boundedness of {\bf the truncated Hilbert Transform} on the line, and show that $I_{i,j}^{\pm,b}$ is bounded on $L^p(\mathbb{R}^n)$ for $1<p<\infty$ (see Proposition \ref{lem.estforwij3b1} for details). 
 At this time, we have finished the proof of the low energy estimate of $W^{L}$.
\subsubsection{The high energy part $W^H$} 
We now sketch the proof of the high energy part.
Erdo\u{g}an-Green \cite{EG22,EG23} proved the $L^p$-boundedness of $W^H$ for all $1\le p\le\infty$ when $n>2m$ under a milder condition of $V$, which requires that $|V|\lesssim \langle x\rangle^{-\beta}$ with $\beta>n+5$, as well as $\|\langle x\rangle^{\frac{4m-1+n}{2}+\delta}V(x)\|_{L^2}\lesssim 1$ for $2m< n< 4m-1$ and  $\|\langle x\rangle^{1+\delta}\langle\nabla \rangle ^{\delta} V(x)\|_{L^2}\lesssim 1$ for $n=4m-1$ with a sufficiently small $\delta>0$. Hence, we shall focus on the case $1\le n\le 2m-1$ here. Indeed, applying the second resolvent identity,
\[R^+(\lambda^{2m})=R_0^+(\lambda^{2m})-R_0^+(\lambda^{2m})VR^+(\lambda^{2m}),\]
$W^H$ can be divided as follows:
\begin{equation*}
	W^H=W^H_1+W^H_2,
\end{equation*}
where
\begin{equation*}
	\begin{split}
		W_1^H=&\frac{m}{\pi i}\int_0^\infty\lambda^{2m-1}(1-\chi(\lambda))R_0^+(\lambda^{2m})V(R_0^+(\lambda^{2m})-R_0^-(\lambda^{2m}))d\lambda,\\
		W_2^H=&\frac{m}{\pi i}\int_0^\infty\lambda^{2m-1}(1-\chi(\lambda))R_0^+(\lambda^{2m})VR^+(\lambda^{2m})V(R_0^+(\lambda^{2m})-R_0^-(\lambda^{2m}))d\lambda.
	\end{split}
\end{equation*}
Using the expression \eqref{eq-free kernel-F} of $R^\pm_0(\lambda^{2m})$ , the integral kernel of $W_1^H$ can be written as a difference of the following terms:
\begin{equation*}
	\begin{split}
		W_{1}^{H,+}(x,y)=\frac{m}{\pi i}\int_0^\infty\lambda^{n-2m}(1-\chi(\lambda))&\Big(\int_{\mathbb{R}^n}e^{i\lambda|x-z|+i\lambda|z-y|}
		\frac{F_n^+(\lambda|x-z|)V(z)F_n^+(\lambda|z-y|)}{|x-z|^{\frac{n-1}{2}}|z-y|^{\frac{n-1}{2}}}dz\Big)d\lambda,
	\end{split}
\end{equation*}
\begin{equation*}
	\begin{split}
		W_{1}^{H,-}(x,y)=\frac{m}{\pi i}\int_0^\infty\lambda^{n-2m}(1-\chi(\lambda))&\Big(\int_{\mathbb{R}^n}e^{i\lambda|x-z|-i\lambda|z-y|}
		\frac{F_n^+(\lambda|x-z|)V(z)F_n^-(\lambda|z-y|)}{|x-z|^{\frac{n-1}{2}}|z-y|^{\frac{n-1}{2}}}dz\Big)d\lambda.
	\end{split}
\end{equation*}
Note that $F_n^+$ shares the same property as $F_n^-$(see \eqref{eq-F decay-odd} ).   Applying an argument for oscillatory integral, we shall obtain
\begin{equation*}
	\left| W_{1}^{H,\pm}(x,y) \right|\lesssim \int_{\mathbb{R}^n}\frac{\min\Big\{\big| \log||x-z|\pm|z-y||\big| ,\,\big||x-z|\pm|z-y|\big|^{-n-1}\Big\}|V(z)|}{|x-z|^\frac{n-1}{2}|y-z|^\frac{n-1}{2}}dz. 
\end{equation*}
This result leads to the conclusion that 
$$	\sup_{y}\int_{\mathbb{R}^n}\left| W_{1}^{H,\pm}(x,y) \right| dx+\sup_{x}\int_{\mathbb{R}^n}\left| W_{1}^{H,\pm}(x,y) \right| dy \lesssim 1,$$
which implies that $W_{1}^{H,\pm}\in\mathbb{B}(L^{p}(\mathbb{R}^n))$ for all $1\le p\le \infty$.

For the remainder term $W_2^H$,  by using the high energy estimates for the resolvent  $R^\pm(\lambda^{2m})$ (see \cite[Lemma 3.9]{FSWY20})  and  the integration by parts, we can prove that the integral kernel $W_2^H(x,y)$ is an admissible kernel, which implies that $W_2^H\in \mathbb{B}(L^p(\R^n))$ for all $1\le p\le\infty.$

\subsection{ Notations}
 We first list some common notations and conventions as follows:
\begin{itemize}
\item  $A\lesssim B$ means $A \leq  CB$, where $C>0$ is an absolute constant. $A\lesssim_l B$ means $A \leq  C_lB$ with a constant $C_l$ dependent on $l$.
 $\mathbb{N}_0=\{0,1,\cdots\}$,  $\mathbb{N}^+=\{1,2,\cdots\}$,  $\mathbb{Z}=\{0,\pm1, \pm2,\cdots\}$.
\vspace{0.1cm}
\item $L^p=L^p(\R^n;\mathbb{C})$ for $1\le p\le \infty$, and $L^{2}_\sigma(\mathbb{R}^n)=\{f; \langle x\rangle^\sigma f(x)\in L^2 \}$.  $L^{p,q}(\R^n)$  denotes the Lorentz space for $1\le p\le \infty$ and  $1\le q\le \infty$ (cf. \cite{gra}).
\vspace{0.1cm}
\item $\lfloor l\rfloor$ denotes the greatest integer less than and equal to  $l$. 
\vspace{0.1cm}
\item $\mathbb{B}(X,Y)$ denotes the space of bounded operators from Banach space $X$ to Banach space $Y$, and $\mathbb{B}(X)=\mathbb{B}(X,X)$.
\vspace{0.1cm}
\item $\mathbf{1}_{\Omega}$ denotes the characteristic function of a set $\Omega\subset \mathbb{R}^n$.
\vspace{0.1cm}
\item $T(x,y)$  denotes  the integral kernel of an operator $T$, namely, 
$$(Tf)(x)=\int T(x,y) f(y) dy.$$

\end{itemize}
\begin{itemize}
\item $m_n$  in \eqref{equ0.1},   denotes {\bf  the total number of zero resonance types} of $H=(-\Delta)^m+V$ on $\mathbb{R}^n$ when $n$ is odd and $1\le n\le 4m-1.$
\vspace{0.1cm}
\item $k_c$  in \eqref{defkc},   denotes {\bf the critical resonance type} of $H$ in the sense that the range of the value $p$ of $L^p$-bound of wave operators will shrink into  a finite interval when  $\k>k_c$.
\vspace{0.1cm}
\item $J_\k$ is the index set with $0\le \k\le m_n$, which is defined in \eqref{eq-Jk-odd} and \eqref{eq-Jk-odd-2}.

\end{itemize}

	
	
	
	
\subsection{The organization of paper}
In Section \ref{sec:3}, we present the properties of free resolvent, the asymptotic behavior of $\big(M^\pm(\lambda)\big)^{-1}$ as $\lambda \to 0^+$ and some integral kernel estimates, which will be used to prove the $L^p$-boundedness of $W_\pm$. 

Section \ref{sec:low energy} is devoted to proving the low energy part of Theorem \ref{thm-main result}.
The high energy part of the proof for Theorem \ref{thm-main result} is presented in section \ref{sec:high energy}. In Section \ref{sec-unb}, we analyze the unboundedness of $W_\pm(H;(-\Delta)^m)$ when $\k>k_c$. We give the proof of two key lemmas in Section \ref{sec:proof}. 

In Appendix \ref{sec:Minverse}, we recall the process of obtaining the asymptotic expansion of $\big(M^\pm(\lambda)\big)^{-1}$.

\section{Preliminaries}\label{sec:3}
This section is dedicated to study the properties of the free resolvent, the asymptotic behavior of $\big(M^\pm(\lambda)\big)^{-1}$ as $\lambda \to 0^+$ and some integral kernel estimates, which will be used to prove the $L^p$-boundedness of $W$.

\subsection{Properties of the free resolvent of $(-\Delta)^m$}
 In this subsection, we recall some properties of the free resolvent of $(-\Delta)^m$. 
For $z\in \mathbb{C}\setminus[0,+\infty),$  set
\begin{equation*}
R_0(z)=((-\Delta)^m-z)^{-1},\ \ \ R(-\Delta;z)=(-\Delta-z)^{-1}.
\end{equation*}
It is known that $R_0(z)$ has the following expression (see \cite{FSWY20}):
\begin{equation}\label{eq-R0z}
R_0(z)=\frac{1}{mz^{1-\frac{1}{m}}}\sum_{k=0}^{m-1}e^{i\frac{\pi k}{m}}R(-\Delta;e^{i\frac{\pi k}{m}}z^{\frac{1}{m}}),
\end{equation}
where  $\arg(z^{\frac{1}{m}})\in (0, \frac{2\pi}{m})$.
Then by the formula \eqref{eq-R0z} and the limiting absorption principle for the free resolvent $R(-\Delta;\,z)$ (see e.g. Agmon \cite{Agmon}), one obtains that 
\begin{equation}\label{defR0}
R_0^\pm(\lambda^{2m})=\big((-\Delta)^m-(\lambda^{2m}\pm i0)\big)^{-1}=\frac{1}{m\lambda^{2m}}\sum_{k\in I^\pm}\lambda_k^2R^\pm(-\Delta;\lambda_k^2),\quad \lambda>0,
\end{equation}
where $\lambda_k=\lambda e^{i\frac{k\pi}{m}}$, $R^\pm(-\Delta;\lambda_k^2)=(-\Delta-(\lambda_k^2\pm i0))^{-1}$ and
\begin{equation}\label{eq-I}
	I^+=\{0,1,\cdots,m-1\},\ \ \ \  I^-=\{1,\cdots,m\}.
\end{equation}
In particular, note that for each $k\in\{1,\cdots,m-1\},$  $\lambda_k^2\in  \mathbb{C}\setminus[0,+\infty)$ and 
$$R^+(-\Delta;\lambda_k^2)=R^-(-\Delta;\lambda_k^2)=(-\Delta-\lambda_k^2)^{-1}.$$
Hence we have that 
\begin{equation}\label{R_+-R_}
R_0^+(\lambda^{2m})-R_0^-(\lambda^{2m})=\frac{1}{m\lambda^{2m-2}}\Big(R^+(-\Delta;\lambda^2)-R^-(-\Delta;\lambda^2)\Big).
\end{equation}
Moreover, the integral kernel of $R_0^\pm(\lambda^{2m})$ can be written as 
\begin{equation}\label{eq-free resovelt-LAP}
	R_0^\pm(\lambda^{2m})(x,y)=\frac{1}{m\lambda^{2m}}\sum_{k\in I^\pm}\lambda_k^2R^\pm(-\Delta;\lambda_k^2)(x,y).
\end{equation}
where
\begin{equation*}
	R(-\Delta;z)(x,y)=\frac{i}{4}\Big(\frac{z^{\frac{1}{2}}}{2\pi|x-y|}\Big)^{\frac{n}{2}-1}H^{(1)}_{\frac{n}{2}-1}(z^{\frac{1}{2}}|x-y|).
\end{equation*}
Here $\Im z^{\frac{1}{2}}> 0$ and $H^{(1)}_{\frac{n}{2}-1}$ is the first Hankel function (cf. \cite{GR02}).

When $n\ge 1$ is odd, by the explicit expressions of the Hankel function (see \cite{GR02}) and the expression \eqref{eq-free resovelt-LAP}, we have
\begin{equation}\label{eq2.8}
	R_0^{\pm}(\lambda^{2m})(x,y)=\frac{1}{(4\pi)^{\frac{n-1}{2}}m\lambda^{2m}}\sum_{k\in I^{\pm}}\frac{\lambda_k^2e^{ i\lambda_k|x-y|}}{|x-y|^{n-2}}\sum^{\frac{n-3}{2}}_{j=\min\{0,\frac{n-3}{2}\}}{c_j(i\lambda_k|x-y|)^j},
\end{equation}
where
\begin{equation}\label{equ1.4}
	c_j=\begin{cases}
		-\frac12,\quad&\mbox{if}\,\, j=-1,\\
		\frac{(-2)^j(n-3-j)!}{j!(\frac{n-3}{2}-j)!},&\mbox{if}\,\, 0\le j\le \frac{n-3}{2}.
	\end{cases}
\end{equation}

Next, we establish some expansions of the integral kernel $R_0^\pm(\lambda^{2m})(x,y)$.
\begin{lemma}\label{lemma-free resolvent expansion}
Let $1\le n\le4m-1$ be odd and $\theta\in\mathbb{N}_0$. Then for any $\lambda>0$, we have the following expansions:

(i)~ If $0\le \theta\le 2m-n$ when $1\le n\le 2m-1$ or $\theta=2m-n$ when $2m<n<4m$, one has
\begin{equation}\label{eq-free expansion-odd-1}
R_0^\pm(\lambda^{2m})(x,y)
=\sum_{0\le j\le\lfloor \frac{\theta-1}{2}\rfloor}a_j^\pm\lambda^{n-2m+2j}|x-y|^{2j}+\lambda^{n-2m}r_{\theta}^\pm(\lambda|x-y|);
\end{equation}

 (ii)~    If $2m-n+1\le \theta\le 4m-n$ when $1\le n\le 2m-1$ or $\theta=0$ when $2m<n<4m$, one has
\begin{equation}\label{eq-free expansion-odd-2}
R_0^\pm(\lambda^{2m})(x,y)
=\sum_{0\le j\le\lfloor \frac{\theta-1}{2}\rfloor}a_j^\pm\lambda^{n-2m+2j}|x-y|^{2j}+b_0|x-y|^{2m-n}+\lambda^{n-2m}r_{\theta}^\pm(\lambda|x-y|);
\end{equation}

(iii)~  If $\theta=4m-n+1$, one has
\begin{equation}\label{eq-free expansion-odd-5}
	\begin{split}
		R_0^\pm(\lambda^{2m})(x,y)
		=b_0|x-y|^{2m-n}+\sum_{0\le j\le\lfloor \frac{\theta-1}{2}\rfloor}a_j^\pm &\lambda^{n-2m+2j}|x-y|^{2j}+\\&b_1\lambda^{2m}|x-y|^{4m-n} +\lambda^{n-2m}r_{\theta}^\pm(\lambda|x-y|).
	\end{split}
\end{equation}
Here in the  \eqref{eq-free expansion-odd-1}---\eqref{eq-free expansion-odd-5}, $a_j^\pm\in\mathbb{C}\setminus\mathbb{R}$ ,  $b_0,~b_1\in\mathbb{R}\setminus\{0\}$ and all the remainder terms $r_{\theta}^\pm(z)\in C^\infty(\mathbb{R})$ satisfy that 
\begin{equation}\label{eq-r estimate-odd}
\big|\frac{d^l}{dz^l} r_{\theta}^\pm(z)\big|\lesssim_l 
\begin{cases}
   |z|^{\theta-l},\ \ \ &l=0,\cdots,\frac{n-1}{2}+\theta,\\
   |z|^{-\frac{n-1}{2}},\ \ \ &l=\frac{n-1}{2}+\theta+1,\cdots.	    
\end{cases}
\end{equation}
Moreover, for $1\le \theta\le 4m-n+1$,  one has that
\begin{equation}\label{eq-free expansion-odd-4}
	{\small\begin{split}
		\big(R_0^+- R_0^-\big)(\lambda^{2m})(x,y)
		=&\sum_{0\le j\le \lfloor\frac{\theta-1}{2}\rfloor}(a_j^+-a_j^-)\lambda^{n-2m+2j}|x-y|^{2j}+\lambda^{n-2m}\big(r_{\theta}^+-r_{\theta}^-\big)(\lambda|x-y|)).
	\end{split}}
\end{equation}

\end{lemma}

\begin{proof}
To prove \eqref{eq-free expansion-odd-1}--\eqref{eq-free expansion-odd-5}, for any $\theta\in \mathbb{N}$, we apply the Taylor formula for $e^{ i\lambda_k|x-y|}$ in \eqref{eq2.8}to get
\begin{equation}\label{eq-free-uniform}
	R^{\pm}_0(\lambda^{2m})(x,y)=\frac{1}{(4\pi)^{\frac{n-1}{2}}m} \sum_{k\in I^{\pm}}\sum_{l=0}^{n+\theta-3}d_l\lambda_k^{l+2-2m}|x-y|^{l-n+2}+\lambda^{n-2m}r_{\theta}^{\pm}(\lambda|x-y|),
\end{equation}
where $d_l:=\sum^{\min\{l,\, \frac{n-3}{2}\}}_{j=\min\{0,\frac{n-3}{2}\}}{\frac{c_j}{(l-j)!}}$,
$c_j$ is given in \eqref{equ1.4}, and
\begin{align}\label{eq2.10}
	r_{\theta}^{\pm}(z)=\sum_{k\in I^{\pm}} \sum_{j=\min\{0,\frac{n-3}{2}\}}^{\frac{n-3}{2}} C_{j,\theta}~ {e^{\frac{i(n-2 m+\theta) k\pi}{m}}}z^\theta \int_{0}^{1} e^{i se\frac{i k \pi}{m} z}(1-s)^{n-j+\theta-3} d s, \quad z\in \mathbb{R}.
\end{align}
Denoted by
\begin{equation}\label{eq2.10.11}
	a_j^{\pm}=\frac{1}{(4\pi)^{\frac{n-1}{2}}m}\sum_{k\in I^{\pm}}d_{2j+n-2}\,e^{\frac{i k\pi}{m}(2j+n-2m)},\quad \,\, b_j=(4\pi)^{-\frac{n-1}{2}} d_{2mj+2m-2},\quad j\in \mathbb{N}_0.
\end{equation}
Then identities \eqref{eq-free expansion-odd-1}--\eqref{eq-free expansion-odd-5} with the respective range of $\theta$ follow from  \eqref{eq-free-uniform} by using the property
\begin{equation*}\label{eq2.8.1}
	\sum_{k=0}^{m-1}{\lambda_k^{2j}}= \sum_{k=1}^{m}{\lambda_k^{2j}}=0,\qquad \text{when}\,\,\, j\in \mathbb{Z}\setminus m\mathbb{Z},
\end{equation*}
and  the fact that $d_l=0$ (see \cite[Lemma 3.3]{J80})  when $l$ is odd and $1\le l\le n-4$. Furthermore, \eqref{eq-free expansion-odd-4} follows immediately from the identities \eqref{eq-free expansion-odd-1}--\eqref{eq-free expansion-odd-5}.

By a direct computation, one has for $z\in \mathbb{R}^+$ and $l=0,\,1,\,\cdots$ that
\begin{equation*}
\left|\frac{d^l}{dz^l}\left( {z}^{\theta } \int_{0}^{1} e^{i s e^\frac{i k\pi}{m}z}(1-s)^{n-j+\theta-3} d s\right)  \right| \lesssim
|z|^{\max\{\theta-l,~0\}}, \quad |z|\leq 1
\end{equation*}	
Meanwhile, note that 
$${z}^{\theta } \int_{0}^{1} e^{i s e^\frac{i k\pi}{m}z}(1-s)^{n-j+\theta-3} d s= z^{n-2-j}\int_{0}^{z} e^{i s e^\frac{i k\pi}{m}}(z-s)^{n-j+\theta-3} d s.$$
Then, we have 
\begin{equation*}
\left|\frac{d^l}{dz^l}\left( z^{n-2-j}\int_{0}^{z} e^{i s e^\frac{i k\pi}{m}}(z-s)^{n-j+\theta-3} d s.\right)  \right| \lesssim|z|^{\max\{\theta-l,-(n-j-2)\}}, \quad |z|\geq 1.
\end{equation*}	
These two estimates, together with $\min\{0,\frac{n-3}{2}\}\le j\le \frac{n-3}{2}$,  yield that for $l=0,\,1\, \cdots$, 
\begin{equation}\label{estforrpm}
	\left|\frac{d^l}{dz^l}\Big( r_{\theta}^{\pm}(z)\Big)  \right| \lesssim_l
	\begin{cases}
		|z|^{\max\{\theta-l,~0\}}, &\quad |z|\leq 1,\\
		|z|^{\max\{\theta-l,\,-\frac{n-1}{2}\}}, &\quad |z|\geq 1,
	\end{cases}
\end{equation}	
which implies \eqref{eq-r estimate-odd}.
\end{proof}

Furthermore,  we also have the following representation of the integral kernels $R_0^\pm(\lambda^{2m})(x,\, y).$
\begin{lemma}\label{lemma-free kernel}
Let $1\le n\le 4m-1$ be odd, then we have
\begin{equation}\label{eq-free kernel-F}
R_0^\pm(\lambda^{2m})(x,y)=\frac{\lambda^{\frac{n+1}{2}-2m}}{|x-y|^{\frac{n-1}{2}}}e^{\pm i\lambda|x-y|}F_n^\pm(\lambda|x-y|),
\end{equation}
where for any $l\in\mathbb{N}_0,$ $F_n^\pm(z)\in C^\infty([0,\,\infty))$ satisfies that
\begin{equation}\label{eq-F decay-odd}
  \big|\frac{d^l}{dz^l} F^\pm_n(z)\big|\lesssim_l \langle z\rangle ^{-l}.
\end{equation}
\end{lemma}
\begin{proof}
Firstly,  define
\begin{equation}\label{F_n1}
F_n^\pm(\lambda|x-y|):=\frac{|x-y|^{\frac{n-1}{2}}}{\lambda^{\frac{n+1}{2}-2m}} e^{\mp i\lambda|x-y|}R_0^\pm(\lambda^{2m})(x,y),
\end{equation}
since the right hand side can be considered as a function of $\lambda|x-y|$ by using the expression \eqref{eq2.8}.
By applying  \eqref{eq-free expansion-odd-1} for $1\le n\le 2m-1$ and  \eqref{eq-free expansion-odd-2}for $2m+1\le n\le 4m-1$ with $\theta =0$, we also  have
\begin{equation}\label{F_n2}
{\small F_n^\pm(\lambda|x-y|)=
\begin{cases}
e^{\mp i\lambda|x-y|}\cdot \left( \lambda|x-y|\right)^{\frac{n-1}{2}} r_{0}^{\pm}(\lambda|x-y|),&\quad 1\le n\le 2m-1,\\	

e^{\mp i\lambda|x-y|}\left(  b_0\left( \lambda|x-y|\right)^{2m-\frac{n+1}{2}}+\left( \lambda|x-y|\right)^{\frac{n-1}{2}} r_{0}^{\pm}(\lambda|x-y|)\right) ,&\quad 2m+1\le n\le 4m-1.
\end{cases}}
\end{equation}
Hence, for $|z|\le 1$, \eqref{eq-F decay-odd} is derived from \eqref{F_n1} and \eqref{estforrpm} noting that $2m-\frac{n+1}{2}\ge 0$ is an integer. On the other hand, for  $|z|\ge 1$,  we can obtain \eqref{eq-F decay-odd} simply by using \eqref{F_n2} and \eqref{eq2.8}. 
\end{proof}

\subsection{The asymptotic expansion for $\big(M^\pm(\lambda)\big)^{-1}$with zero  resonances.}\label{expansion} As seen in \eqref{eq-wave operator-low-1}, to obtain the $L^p$-boundedness of wave operators $W_\pm$ in the low energy part, one of the key steps is to establish the asymptotic expansion of $\big(M^\pm(\lambda)\big)^{-1}$ near the zero energy threshold. In this subsection, we shall accomplish this task.
\subsubsection{Resonance orthogonal projections  $\{Q_j\}_{j\in J_{\k}}$}
Firstly, we define some notations used throughout the section.  Let $m\ge 2$ and $n$ be an odd integer. Recall that 
\begin{equation*}\
	m_n=\min\big\{m, \ 2m-\frac{n-1}{2} \big\}=
	\begin{cases}
		m,&\ \ \mbox{if}\ \  1\le n\le 2m-1,\\[0.3cm]
		\mbox{$\frac{4m-n+1}{2}$},&\ \ \mbox{if}\ \ 2m+1\le n\le 4m-1,
	\end{cases}
\end{equation*}
 which is the total number of zero resonance types of $H=(-\Delta)^{m}+V$ on $\R^n$.  
 
Let $\k$ be an integer and $0\le k\le m_n+1$.  If $1\le n\leq 2m-1$, then we define the index set $J_{\k}$ of real number:
\begin{equation}\label{eq-Jk-odd}
J_{\k}:=\begin{cases}
	\{m-\frac{n}{2}\}\cup\{j\in\mathbb{N}_0;\,0\le j< m-\frac{n-1}{2}+\k\},\quad&\mbox{if}\,\, 0\le \k\le m_n,\\[0.3cm]
	\{2m-\frac{n}{2}\}\cup J_{m_n},&\mbox{if}\,\, \k=m_n+1;
\end{cases}
\end{equation}
For $2m+1\leq n\leq 4m-1$, we denote
\begin{equation}\label{eq-Jk-odd-2}
	J_{\k}:=\begin{cases}
		\{m-\frac{n}{2}\}\cup\{j\in\mathbb{N}_0;\,0\leq j< \k \},\quad&\mbox{if}\, \, 0\le \k\le m_n,\\[0.3cm]
		\{2m-\frac{n}{2}\}\cup J_{m_n},&\mbox{if}\,\, \k=m_n+1.
	\end{cases}
\end{equation}

 More specifically, the elements in $J_{\k}$ for each $0\le \k\le m_n+1$ are listed as follows: \\ 
If $1\le n\leq 2m-1$,  then 
\begin{equation}\label{eq-Jk-odd''}
J_{\k}=\begin{cases}\{0,\ 1,\ \cdots,\  m-\frac{n+1}{2},\ {\bf m-\frac{n}{2}}\},\quad&\mbox{if}\,\, \k=0,\\[0.2cm]
	\{0,\ 1,\ \cdots,\  m-\frac{n+1}{2},\  {\bf m-\frac{n}{2}},\  m-\frac{n-1}{2}, \ \cdots,\    m-\frac{n-1}{2}+\k-1\},\quad&\mbox{if}\,\, 1\le \k\le m_n=m,\\[0.2cm]
	 \{0,\ 1,\ \cdots,\  m-\frac{n+1}{2},\  {\bf m-\frac{n}{2}},\   m-\frac{n-1}{2}, \ \cdots,\    2m-\frac{n+1}{2},\ {\bf 2m-\frac{n}{2}}\},&\mbox{if}\,\, \k=m_n+1;
\end{cases}
\end{equation}
If $2m+1\leq n\leq 4m-1$,  then  
\begin{equation}\label{eq-Jk-odd-2''}
	J_{\k}=\begin{cases}
 \{{\bf m-\frac{n}{2}}\},\quad&\mbox{if}\, \, \k=0,\\[0.2cm]
		\{{\bf m-\frac{n}{2}},\ 0, \ 1, \  \cdots, \ \k-1 \},\quad&\mbox{if}\, \, 1\le \k\le m_n=2m-\frac{n-1}{2},\\[0.2cm]
		\{{\bf m-\frac{n}{2}},\ 0, \ 1, \  \cdots, \ 2m-\frac{n+1}{2},\ {\bf 2m-\frac{n}{2}} \},&\mbox{if}\,\, \k=m_n+1.
	\end{cases}
\end{equation}
Clearly, $J_{\k}\subset J_{\k'}$ for $\k<\k'$ in each dimension class. Since $n$ is an odd integer, the elements $m-{n\over 2}$ and $2m-{n\over 2}$ are the only two possible {\bf half-integers} appearing in the set $J_{\k}$, and each set $J_{\k}$ contains at least one or both of the two half-integers.

Note that the elements in $J_{\k}$ are related to the expansions of $R_0^{\pm}(\lambda^{2m})(x,y)$ in Lemma \ref{lemma-free resolvent expansion}. Hence for each $j\in J_{\k}$, we define the operators
\begin{equation}\label{eq-G-odd}
	G_{2j}f=\int_{\R^n}|x-y|^{2j}f(y)\, d y,\quad \text{for}\, \, f\in \mathcal{S}(\R^n),
\end{equation}
and set 
\begin{equation}\label{eq-T_0-odd}
T_0=U+v(-\Delta)^{-m}v=U+b_{0}vG_{2m-n}v,
\end{equation}
where $v(x)=\sqrt{|V(x)|}, U(x)=\text{sgn} V(x)$  and $b_0G_{2m-n}$ is the fundamental solution operator of $(-\Delta)^m$  with  a constant  $b_0$  given in Lemma \ref{lemma-free resolvent expansion}.   These operators $G_{2j}$ and $T_0$ will play fundamental roles in the inverse expansion of $M^{\pm}(\lambda)=U+vR^{\pm}_0(\lambda^{2m}) v$ at zero threshold of $H=(-\Delta)^m+V$.

Now we introduce some subspaces of $L^2(\R^n)$ and their associated orthogonal projections $S_j$. In the sequel,  $X^\perp$ denotes the orthogonal complement to a subspace $X$ of $L^2(\R^n)$,  and ${\rm{Ker}}(T)$ denotes the kernel space of some operator $T$ on $L^2(\R^n)$.  

Denoted by $S_{-l}=I$ when $l=1,2,\cdots.$ For $j\in J_{m_n+1}$ and a fixed integer   $0\le \k\le m_n+1$, if $n$ is odd and $1\le n \le2m-1$, we define
{\small
\begin{equation}\label{eq-S_j-odd}
 S_j L^2=\begin{cases}
 	\{x^\alpha v;\,\,|\alpha|\le j\}^\perp,&\mbox{if}\,\ 0\le j< m-\frac{n}{2},\\
 	\mbox{Ker}\,(S_{m-\frac{n+1}{2}} T_0 S_{m-\frac{n+1}{2}})\bigcap S_{m-\frac{n+1}{2}} L^2,\,&\mbox{if}\,\ j=m-\frac{n}{2},\\
 	\{x^\alpha v;\,\,\,|\alpha|\le j\}^\perp\bigcap\mbox{ker}(S_{2m-n-j-1}T_0S_{m-\frac{n}{2}})\bigcap S_{m-\frac{n}{2}} L^2,
 	 &\mbox{if}\,\ m-\frac{n}{2}<j< m-\frac{n-1}{2}+\k,\\
 	\{0\},&\mbox{if}\,\ j\ge m-\frac{n-1}{2}+\k;
 \end{cases}
\end{equation}
}
if $n$ is odd and $2m+1\leq n\leq 4m-1$, we define
\begin{equation}\label{eq-S_j-odd-2}
S_j L^2=\begin{cases}
\mbox{Ker}\,(T_0) ,\,&\mbox{if}\,\ j=m-\frac{n}{2},\\
\{x^\alpha v;\,\,\,|\alpha|\le j\}^\perp
\bigcap S_{m-\frac{n}{2}} L^2,
&\mbox{if}\,\ 0\le j< \k,\\
\{0\},&\mbox{if}\,\ j\le \k.
		\end{cases}
\end{equation}
\begin{remark}
\mbox{}
\begin{itemize}

\item   In the  case of  a fixed resonance type  $\mathbf{k}$ of  $H=(-\Delta)^m+V$,   we emphasize that these projections $S_j$ are well-defined by the decay condition on $V$ in Assumption \ref{Assumption},  and  one has $S_j=0$ if  $j\in J_{m_n+1}\setminus J_{\k}$. 
\item  From the definition of $S_jL^2$ above,  
 the following  inclusion relations hold for  $1\le n\le 2m-1$,
\begin{equation*}
	S_0L^2\supset S_1L^2\supset\cdots\supset S_{m-\frac{n+1}{2}}L^2\supset S_{m-\frac{n}{2}}L^2\supset S_{m-\frac{n-1}{2}}L^2\supset\cdots\supset S_{2m-\frac{n+1}{2}}L^2=\{0\},
\end{equation*}
and  for $2m+1\le n\le 4m-1$, 
\begin{equation*}
	S_{m-\frac{n}{2}}L^2\supset S_0L^2\supset S_1L^2\supset\cdots\supset S_{2m-\frac{n+1}{2}}L^2=\{0\}.
\end{equation*}
 In particular,   zero is a regular point of $H$ ( i.e. $\k=0$ )  is equivalent that $S_{m-\frac{n}{2}}L^2=\{0\}$.
\item Furthermore, for each $j\in J_{m_n+1}$, one always has the cancellation relations:
\begin{equation}\label{CforQ_j}
	S_j(x^\alpha v)=0, \quad \text{for}~\alpha\in {\mathbb{N}^n_0},~|\alpha|\leq\max\{0,\lfloor j\rfloor\},
\end{equation} where $\lfloor j\rfloor$ denotes the greatest integer at most $j$. Also, for each $j\ge 2m-n$, we have 
\begin{equation}\label{S_jT_0}
	T_0S_j=(U+b_0vG_{2m-n}v)S_j=0.
\end{equation}
\end{itemize}
\end {remark}

 In addition, we define another group of orthogonal projections $Q_j$ for each $j\in J_{\k}$ by
\begin{equation}\label{eq-Q_j-odd}
Q_j:=
\begin{cases}   
I-S_0,  &\quad \text{if $1\le n\le 2m-1$ and $j=0$}, \\
I-S_{m-\frac n2},  &\quad \text{if $2m< n <4m$ and $j=m-\frac n2$}, \\
S_{j'}-S_j,   &\quad \text{else},     
\end{cases}
\end{equation}
where $$ j'=\max\Big\{l\in J_{m_n+1};\,\, \ \ l<j \Big\}.$$  Then , it follows that $Q_j\le S_{\lfloor j+\frac12\rfloor-1}$ when $j>0$ and  $Q_j$ satisfies the following  cancellation relations:
\begin{equation}\label{eq-cancellation-Q_j}
	Q_j(x^\alpha v)=0, \quad \text{for ~$ \alpha\in {\mathbb{N}_0^n}, ~~|\alpha|\leq\max\{0,\lfloor j+1/2\rfloor\}-1$}.	
\end{equation}

For convenience,  let 
\begin{equation}\label{thetaj}
\vartheta(j)=\max\{0,\lfloor j+1/2 \rfloor\}.
\end{equation}

By the definition of these orthogonal projections $\{Q_j\}_{j\in J_{\k}}$,  we can give the following equivalent characterizations of zero-energy resonances of $H$, which can be seen in \cite{CHHZ-2024}.

\begin{proposition}\label{prop-resonance-odd}
Let $1\le n\le 4m-1$ be odd, $0\le \k\le m_n+1$,  and assume that $V(x)$ satisfies the Assumption \ref{Assumption}. Then the following statements are valid:
\vskip0.1cm
\indent\emph{(\romannumeral1)} Zero energy is a regular point of $H$ (  i.e. $\k=0$ )  if and only if  $\sum_{j\in J_{0}}Q_j=I$;
\vskip0.1cm
\indent\emph{(\romannumeral2)} Zero energy is the $\mathbf{k}$-th kind resonance with $0< \mathbf{k}\le m_n+1$ if and only if  $\sum_{j\in J_{\mathbf{k-1}}}Q_j\neq I$ and $\sum_{j\in J_{\mathbf{k}}}Q_j=I$.

\end{proposition}

\subsubsection{Asymptotic expansions of $(M^\pm(\lambda))^{-1}$}
Now we turn to study the asymptotic expansion of $(M^\pm(\lambda))^{-1}$ as $\lambda\to 0$, where $v(x) = \sqrt{|V(x)|}, U (x) = \sgn V(x)$ and $$M^\pm(\lambda)=U+vR_0^\pm(\lambda^{2m})v,\,\,\, \lambda>0.$$

Let zero energy be the $\mathbf{k}$-th kind resonance of $H=(-\Delta)^m+V$. Then applying the expansions of $R_0^{\pm}(x,y)$  in Lemma \ref{lemma-free resolvent expansion}, $M^\pm(\lambda)$ can be expanded as follows:
\begin{itemize}
\item If $ 0\le \mathbf{k}\le m_n$, then we have
\begin{equation}\label{equ3.52}
\begin{aligned}
M^\pm(\lambda)&=\sum\limits_{0\le j\le \lfloor\frac{\theta-1}{2}\rfloor}a_j^\pm\lambda^{n-2m+2j}vG_{2j}v+T_0+\lambda^{n-2m}vr_{\theta}^\pm(\lambda)v,
\end{aligned}
\end{equation}
where
\begin{equation*}
\theta=
\begin{cases}
\max\{0,2m-n\}+1,&\quad\text{for}~ \mathbf{k}=0,\\
\max\{0,2m-n\}+2\mathbf{k},&\quad \text{for}~  1\le \mathbf{k}\le m_n;
\end{cases}
\end{equation*}
\item If $\mathbf{k}= m_n+1$, taking $\theta=4m-n+1$, then we have
\begin{equation}\label{equ3.52.1}
\begin{aligned}
M^\pm(\lambda)&=\sum\limits_{0\le j\le \lfloor\frac{\theta-1}{2}\rfloor}a_j^\pm\lambda^{n-2m+2j}vG_{2j}v+T_0+b_1\lambda^{2m}vG_{4m-n}v+\lambda^{n-2m}vr_{4m-n+1}^\pm(\lambda)v.
\end{aligned}
\end{equation}
\end{itemize}
In the above expansions,  $r_{\theta}^\pm(\lambda)$ are the operators with integral kernels $r_{\theta}^\pm(\lambda|x-y|)$.

Proposition \ref{prop-resonance-odd} shows that  when zero is the $\mathbf{k}$-th kind resonance,   $\{Q_j\}_{j\in J_{\mathbf{k}}}$  gives a direct sum decomposition of $L^2$ ( i.e. $\sum_{j\in J_{\k}} Q_jL^2=L^2$ ). Then for a fixed $\lambda>0$, we define an isomorphism $B_\lambda$ on $L^2$ 
$$B:=B_\lambda=\big(\lambda^{-j} Q_j\big)_{j\in J_{\k}}:~ \big(f_j\big)_{j\in J_{\k}}\in\bigoplus_{j\in J_{\k}}Q_jL^2\longmapsto \sum_{j\in J_{\k}}\lambda^{-j}Q_jf_j\in L^2,$$
where  $\bigoplus_{j\in J_{\k}}Q_jL^2$ denotes a vector-valued space.
Thus, by the  transform $B_\lambda$, to study the inverse of $M^\pm(\lambda)$ on $L^2$, it suffices to analyze the inverse of the following  operator matrices $B^*M^\pm(\lambda)B$ on  $\bigoplus_{j\in J_{\k}}Q_jL^2$,  which was recently derived in \cite{CHHZ-2024} based on the aforementioned ideas.

Now we present the results on the asymptotic expansion of $(M^\pm(\lambda))^{-1}$ near zero energy threshold.  In Appendix A, we also give a concise proof for reader's convenience. 
\begin{theorem} \label{thm-M inverse-odd}
 Let $1\le n\le 4m-1$ be odd, $0\le \mathbf{k}\le m_n+1,$ and assume that $V(x)$ satisfies  Assumption \ref{Assumption}. Then there exists some $\lambda_0\in(0,1)$ such that $M^\pm(\lambda)$ are invertible on $L^2(\mathbb{R}^n)$ for all $0<\lambda<\lambda_0$, and we have

\begin{equation}\label{eq-M expansion-odd}
\big(M^\pm(\lambda)\big)^{-1}=\sum_{i,j\in J_\mathbf{k}}\lambda^{2m-n-i-j}Q_i(M_{i,j}^\pm+\Gamma_{i,j}^\pm(\lambda))Q_j,
\end{equation}
where all $M^\pm_{i,j}$ are $\lambda-$independent operators,  all $\Gamma_{i,j}^\pm(\lambda)$ are $\lambda-$dependent and satisfy
\begin{equation}\label{equ3.47}
\big\|\partial_\lambda^l\Gamma_{i,j}^\pm(\lambda)\big\|_{\mathbb{B}(L^2)}\lesssim \lambda^{\frac{1}{2}-l}, \qquad \text{for}~l=0,1,\cdots,\frac{n+3}{2}.
\end{equation}
 Additionally, for any $l=0,1,\cdots,\frac{n+3}{2}$, we have
\begin{equation}\label{equ2.100}
	\begin{cases}
		\big\|\partial_\lambda^l\Gamma_{ m-\frac{n}{2},  m-\frac{n}{2}}^\pm(\lambda)\big\|_{\mathbb{B}(L^2)}\lesssim \lambda^{\max\{1,\, n-2m\}-l}\quad&\mbox{when}\,\,\, \mathbf{k}=0,
		\\
		\big\|\partial_\lambda^l\Gamma_{ 2m-\frac{n}{2},  2m-\frac{n}{2}}^\pm(\lambda)\big\|_{\mathbb{B}(L^2)}\lesssim \lambda^{1-l}\quad&\mbox{when}\,\,\, \mathbf{k}=m_n+1.
	\end{cases}
\end{equation}
 Furthermore, we have that
 $$M^\pm_{m-\frac{n}{2},m-\frac{n}{2}}=(Q_{m-\frac{n}{2}}T_0Q_{m-\frac{n}{2}})^{-1}, \,\,\, M^\pm_{2m-\frac{n}{2},2m-\frac{n}{2}}=(b_1Q_{2m-\frac{n}{2}}vG_{4m-n}vQ_{2m-\frac{n}{2}})^{-1},$$ 
where $b_1$ is given in Lemma \ref{lemma-free resolvent expansion},
 and 
 \begin{equation}\label{eq-mij=0}
 M^\pm_{i,j}=0, \quad i\ne j, ~i\in\{m-\frac{n}{2},\, 2m-\frac{n}{2}\} ~\text{or} ~j\in\{m-\frac{n}{2},\, 2m-\frac{n}{2}\}.
 \end{equation}
 \end{theorem}

\begin{remark}
$\Gamma_{i,j}^\pm(\lambda)$ in the above theorem are $\mathbb{B}(L^2)$-valued $C^{\frac{n+3}{2}}$ functions of $\lambda$ instead of $C^{\frac{n+1}{2}}$ functions of $\lambda$ presented in Theorem 2.7 of \cite{CHHZ-2024}. This slight modification arises because that for a fixed $0\le \mathbf{k}\le m_n+1$, $vr_{\theta}^\pm(\lambda)v$ in \eqref{equ3.52} and \eqref{equ3.52.1} are indeed  $\mathbb{B}(L^2(\mathbb{R}^n))$-valued $C^{\frac{n+3}{2}}$ functions of $\lambda$ by the Hilbert-Schmidt property, our decay assumption on $V$ and \eqref{eq-r estimate-odd}. Meanwhile they satisfy
\begin{equation}
	\big\|\partial_\lambda^l\Big(vr_{\theta}^\pm(\lambda)v\Big)\big\|_{\mathbb{B}(L^2)}\lesssim  \lambda^{\theta-l},\quad\text{for}~~0<\lambda<1, ~~l=0,\,1\,\cdots,\,\frac{n+3}{2}.
\end{equation}
Thus, by repeating the proof for Theorem 2.7 of \cite{CHHZ-2024}, we establish a slightly modified version, which we present as Theorem \ref{thm-M inverse-odd}.
\end{remark}

\subsection{Two key lemmas.}
 The following two lemmas take advantage of the cancellation property of  projections $Q_j$ to study the structure of  the kernels  $ \Big(Q_jvR_0^\pm(\lambda^{2m})\Big)(x, y)$  of  $Q_jvR^\pm(\lambda^{2m})$ given by the following integrals:
$$\Big(Q_jvR^\pm(\lambda^{2m})f\Big)(x)=\int_{\R^n}\Big(Q_jvR_0^\pm(\lambda^{2m})\Big)(x, y)\, f(y)dy,$$
which will be ultimately utilized to analyze the integral kernel of $W^L$.   Let  $R_0^\pm(\lambda^{2m})(x, y)$ be the kernel of  $R_0^\pm(\lambda^{2m})$.  Then  one has
$$\Big(Q_jvR_0^\pm(\lambda^{2m})\Big)(x, y)=Q_jv\Big(R_0^\pm(\lambda^{2m})(\cdot, y)\Big)(x).$$
In particular,  without confusion,  $\Big(Q_jvR_0^\pm(\lambda^{2m})\Big)(x, y)$ can be regarded as $Q_jv$ acting on the first variable $z$  of the kernel $R_0^\pm(\lambda^{2m})(z, y)$ for each fixed $y$.
\vskip0.2cm
To understand more on the role of projection $Q_j$, we first note that by  Lemma \ref{lemma-free kernel},  the kernels of $vR_0^\pm(\lambda^{2m})$ without the action of $Q_j$ can be written as follows:
\begin{equation}\label{eq-free kernel-F'}
	vR_0^\pm(\lambda^{2m})(x,y)=e^{\pm i\lambda|y|} \,\omega^{\pm}(\lambda, x,y),
\end{equation}
where 
$$\omega^{\pm}(\lambda, x,y)=\frac{\lambda^{{n+1\over 2}-2m}\,v(x)}{ |x-y|^{\frac{n-1}{2}}}\,\,e^{\pm i\lambda\big(|x-y|-|y|\big)}F_n^\pm(\lambda|x-y|),$$
satisfy that for each $l=0,1,\cdots,\frac{n+3}{2}$ (depending on the decay condition of potential $V$), 
$$
\big\|\partial_\lambda^l \omega(\lambda,x,y)\big\|_{L^2_x}\lesssim\lambda^{\frac{n+1}{2}-2m-l}\langle y\rangle^{-\frac{n-1}{2}},\,\, \lambda>0.
$$
Comparing with \eqref{eq-k1 estimate-odd-3} or \eqref{eq-k2-odd-2} in the following  Lemma \ref{lemma-QjvR-odd},  the kernel  $\big(Q_jvR_0^\pm(\lambda^{2m})\big)(x, y)$, benefiting from the cancellation of $Q_j$, has a better property than $vR_0^\pm(\lambda^{2m})(x,y)$ as $\lambda \to 0$ due to the gain of an extra factor $\lambda^{\vartheta(j)}$.    This gain, in some sense, serves as a compensation for the singularity factor  $\lambda^{-j}$  in the expansion \eqref{eq-M expansion-odd} of $(M^\pm(\lambda))^{-1}$ if zero is $\k$-th kind resonance  of $H$.

Now we list the following lemmas, whose proofs are complicated and given in Section \ref{sec:proof}.  

\begin{lemma}\label{lemma-QjvR-odd}
	Let $1\le n\le 4m-1$ be odd, and $0\le \mathbf{k}\le m_n+1$. Assume that $V(x)$ satisfies Assumption \ref{Assumption}.    Then for any  $j\in J_\mathbf{k}$ and $\lambda\in(0,1)$, we have
	\begin{equation}\label{eq-QjR kernel-odd}
		Q_jv\Big(R_0^\pm(\lambda^{2m})(\cdot, y)\Big)(x)=e^{\pm i\lambda |y|}\omega_{j}^{\pm}(\lambda,x,y),
	\end{equation}
	where for any $l=0,1,\cdots,\frac{n+3}{2}$, the functions $\omega_{j}^{\pm}(\lambda,x,y)$ satisfy
	\begin{align}\label{eq-k1 estimate-odd-1}
		\big\|\partial_\lambda^l \omega_{j}^{\pm}(\lambda,x,y)\big\|_{L_x^2}&\lesssim \lambda^{\min\{n-2m+\vartheta(j),\, 0\}-l}, \\ \label{eq-k1 estimate-odd-3}
		\big\|\partial_\lambda^l \omega_{j}^{\pm}(\lambda,x,y)\big\|_{L_x^2}&\lesssim \lambda^{\min\{\frac{n+1}{2}-2m+\vartheta(j),\, 0\}-l}\langle y\rangle^{-\frac{n-1}{2}}.
	\end{align}
	Furthermore, when $j\in J_\mathbf{k}$ and $j\ge  \max\{2m-n,\, m-\frac n2\}$, then for any $l=0,1,\cdots,\frac{n+3}{2}$,  we have
	\begin{equation}\label{eq-k1 estimate-odd-2}
		\big\|\partial_\lambda^l \omega_{j}^{\pm}(\lambda,x,y)\big\|_{L_x^2}\lesssim \lambda^{-l}\langle y\rangle^{-\min\{n-2m+\vartheta(j), \,\frac{n-1}{2}\}}.\\
	\end{equation}
	For any $0\le \mathbf{k}\le m_n+1$ and $j\in J_\mathbf{k}$, we have
	\begin{equation}\label{eq3.900}
		Q_jv\big(R_0^+(\lambda^{2m})(\cdot,\,y)-R_0^-(\lambda^{2m})(\cdot,\,y)\big)(x)=e^{ i\lambda |y|}\tilde{\omega}_{j}^{+}(\lambda,x,y)+e^{-i\lambda |y|}\tilde{\omega}_{j}^{-}(\lambda,x,y),
	\end{equation}
	where the functions $\tilde{\omega}_{j}^{\pm}(\lambda,x,y)$ satisfy
	\begin{align}\label{eq-k2-odd-1}
		\big\|\partial_\lambda^l \tilde{\omega}_{j}^{\pm}(\lambda,x,y)\big\|_{L_x^2}&\lesssim\lambda^{n-2m+\vartheta(j)-l},\\ \label{eq-k2-odd-2}
		\big\|\partial_\lambda^l \tilde{\omega}_{j}^{\pm}(\lambda,x,y)\big\|_{L_x^2}&\lesssim\lambda^{\frac{n+1}{2}-2m+\vartheta(j)-l}\langle y\rangle^{-\frac{n-1}{2}}.
	\end{align}
	In the inequalities above,  $\vartheta(j)=\max\{0,\lfloor j+1/2 \rfloor\}$  as given in \eqref{thetaj}.
\end{lemma}

\vskip0.2cm
\begin{remark}\label{vartheta}
	For $j\in J_{m_n+1}$, if $n$ is odd and $1\le n \le2m-1$, we have
	{\small
		\begin{equation}\label{eq-vartheta-odd-1}
			\vartheta(j)=\begin{cases}
				j&\mbox{if}\,\ 0\le j\le m-\frac{n+1}{2},\\
				m-\frac{n-1}{2},\ &\mbox{if}\,\ j=m-\frac{n}{2},\\
				j,\ &\mbox{if}\,\  m-\frac{n-1}{2}\le j\le 2m-\frac{n+1}{2},\\
				2m-\frac{n-1}{2},&\mbox{if}\,\ j=2m-\frac{n}{2};
			\end{cases}
		\end{equation}
	}
	and if $n$ is odd and $2m+1\leq n\leq 4m-1$, we have
	\begin{equation}\label{eq-vartheta-odd-2}
		\vartheta(j)=\begin{cases}
			0 ,\,&\mbox{if}\,\ j=m-\frac{n}{2},\\
			j,
			&\mbox{if}\,\ 0\le j\le 2m-\frac{n+1}{2} ,\\
			2m-\frac{n-1}{2},&\mbox{if}\,\ j=2m-\frac{n}{2}.
		\end{cases}
	\end{equation}   
	As a consequence, for all $j<2m-\frac{n}{2}$, we have $\frac{n+1}{2}-2m+\vartheta(j)\le 0$ and  $n-2m+\vartheta(j)\le \frac{n-1}{2}$. Only when $j=2m-\frac{n}{2}$, we have $\frac{n+1}{2}-2m+\vartheta(j)\ge 0$ and $n-2m+\vartheta(j) >\frac{n-1}{2}$. 
\end{remark}

\begin{remark}
	The exponential minimum restriction of $\lambda$ in the estimates \eqref{eq-k1 estimate-odd-1} and \eqref{eq-k1 estimate-odd-3} cannot be removed. This can be seen from the asymptotic  expression  \eqref{eq-free expansion-odd-5} with $x=0$, that is 
	$$Q_jv\Big(R_0^\pm(\lambda^{2m})(\cdot,0)\Big)(x)=\omega_{j}^{\pm}(\lambda,x,0)=Q_j\big(v(x)|x|^{2m-n}\big)+O(\lambda^1),$$	
	for $j>2m-n$. The term $Q_j\big(v(x)|x|^{2m-n}\big)$ will not vanish since $2m-n$ is odd. Hence we only have that
	$$\big\|\omega_{j}^{\pm}(\lambda,x,0)\big\|_{L_x^2}=O(1), \ \ 0<\lambda\ll1,$$
	which means that the estimates \eqref{eq-k1 estimate-odd-1} and \eqref{eq-k1 estimate-odd-3} cannot be improved generally.
\end{remark}

We also need the following result on a more refined analysis of $Q_jv\Big(R_0^\pm(\lambda^{2m})(\cdot, y)\Big)(x)$.

\begin{lemma}\label{lem3.10}
	Let $1\le n\le 4m-1$ be odd, and $0\le \mathbf{k}\le m_n+1$. Assume that $V(x)$ satisfies Assumption \ref{Assumption}.    Then for any  $j\in J_\mathbf{k}$ and $\lambda\in(0,1)$, we have
	\begin{equation}\label{eq3.1000}
		\Big(Q_jvR_0^\pm(\lambda^{2m}(\cdot,y)\Big)(x)=e^{\pm i\lambda|y|}\Big(\omega_{j,1}^\pm(\lambda,x,y)+\omega_{j,2}^\pm(\lambda,x,y)+\omega_{j,3}^\pm(\lambda,x,y)\Big).
	\end{equation}
	Here, 
	\begin{equation}\label{eq3.100001}
		\omega_{j,1}^\pm(\lambda,x,y)=\lambda^{\frac{n+1}{2}-2m+\vartheta(j)}\,\mathbf{1}_{\{|y|>1\}}\sum_{|\alpha|=|\beta|=\vartheta(j)}A_{\alpha,\beta}^\pm \frac{y^\alpha \Big(Q_j(v z^\beta)\Big)(x)}{|y|^{\frac{n-1}{2}+\vartheta(j)}},
	\end{equation}
	and for any $l=0,\cdots,\frac{n+3}{2}$, $\omega_{j,2}^\pm(\lambda,x,y)$ and $ \omega_{j,3}^\pm(\lambda,x,y)$ satisfy
	\begin{equation}\label{eq3.1001}
		\Big\|\partial^l_\lambda \omega_{j,2}^\pm(\lambda,x,y)\Big\|_{L_x^2}\lesssim 
		\begin{cases}
			\lambda^{\frac{n-1}{2}-2m+\vartheta(j)+\varepsilon-l}\langle y\rangle^{-\frac{n+1}{2}+\varepsilon} \quad & j<2m-\frac n2~~ \text{and} ~~\varepsilon\in [0,1] ,\\	
			\lambda^{-l}\langle y\rangle^{-\frac{n+1}{2}}, \quad & j=2m-\frac n2,
		\end{cases}
	\end{equation}
	and 
	\begin{equation}\label{eq3.10011}
		\Big\|\partial^l_\lambda \omega_{j,3}^\pm(\lambda,x,y)\Big\|_{L_x^2}\lesssim \lambda^{\frac{n+3}{2}-2m+\vartheta(j)}\langle y\rangle^{-\frac{n-1}{2}}. 
	\end{equation}
	In the inequalities above,  $\vartheta(j)=\max\{0,\lfloor j+1/2 \rfloor\}$ as given in \eqref{thetaj}.
\end{lemma}

\subsection {Some criteria to $L^p$-boundedness and oscillatory integrals} \label{sec-pre-for-low}
 In what follows, we prepare several basic criteria to establish $L^p$-boundedness of  operators with specific integral kernels.

Denote by $T(x,y)$  the integral kernel of the operator $T$. We say that the integral kernel $T(x,y)$ is admissible if and only if $T(x,y)$ satisfies
\begin{equation}\label{eq-admissible condition}
\sup_{x\in\mathbb{R}^n}\int |T(x,y)|dy+\sup_{y\in\mathbb{R}^n}\int|T(x,y)|dx<\infty.
\end{equation}
Now, let us recall the classical Schur's lemma.
\begin{lemma}\label{Schur Lemma}
If $T(x,y)$ is admissible, then $T$ is bounded on $L^p(\mathbb{R}^n)$ for all $1\le p\le\infty$.
\end{lemma}

Moreover, we also need the following three lemmas.

\begin{lemma}\label{lemma-bounedness-1}
Assume that the integral kernel $T(x,y)$ obeys the following pointwise bounds:
\[|T(x,y)|\lesssim \langle x\rangle^{-n}{\bf1}_{\{|x|>2|y|\}}.\] 
Then the operator $T$ is bounded on $L^p(\mathbb{R}^n)$ for all $1< p\le \infty.$
\end{lemma}
\begin{proof}
Firstly, it is obvious that the operator with the kernel $ \langle x \rangle^{-n} $ is bounded from $ L^1(\mathbb{R}^n) $ to $ L^{1,\infty}(\mathbb{R}^n) $. Furthermore, one has that for any $ x \in \mathbb{R}^n $, there exists
\[
\sup_{x\in\mathbb{R}^n}\int_{\mathbb{R}^n} |T(x,y)|dy \lesssim \sup_{x\in\mathbb{R}^n}\Big(\langle x\rangle^{-n} \int_{\mathbb{R}^n} {\bf1}_{\{|x|>2|y|\}}dy \Big)\lesssim 1.
\]
Therefore,  $T$ is  also bounded on $ L^{\infty }(\mathbb{R}^n) $. Then, by the Marcinkiewicz interpolation theorem (see Section 1.3 in \cite{gra}), the operator $ T $ is bounded on $ L^p(\mathbb{R}^n) $ for any $ 1 < p \le  \infty$. 
\end{proof}

\begin{lemma}\label{lemma-bounedness-11}
Assume that the integral kernel $T(x,y)$ obeys the following pointwise bounds:
\[|T(x,y)|\lesssim \langle x\rangle^{-\alpha}\langle y\rangle^{\alpha-n}{\bf1}_{\{|x|<\frac 12|y|\}}\] 
for some $\alpha\in[0,n)$. Then the operator $T$ is bounded on $L^p(\mathbb{R}^n)$ for all $1\le  p< \frac{n}{\alpha}$.
\end{lemma}
\begin{proof}
One can easily find that 
$$\sup_{y\in \R^n}\int_{\mathbb{R}^n} |T(x,y)| dx\lesssim 1.$$
Thus, $T$ is bounded  on $L^1(\mathbb{R}^n)$. 
Furthermore, 
by the H\"older-type inequality in Lorentz spaces (see Section 1.4 in \cite{gra}), one has
\begin{equation}\label{eq.estforkernel}
\Big\| Tf(x) \Big\|_{L^{\frac{n}{\alpha},\infty}}\lesssim \Big\|\langle x \rangle^{-\alpha}\int_{\mathbb{R}^n}\langle y \rangle^{-n+\alpha}f(y) \, dy\Big\|_{L^{\frac{n}{\alpha},\infty}} \lesssim \|\langle x \rangle^{-n}\|_{L^{1,\infty}}  \|f\|_{L^{\frac{n}{\alpha},1}}.	
\end{equation}
Therefore,  $T$ is also bounded from $ L^{\frac{n}{\alpha},1}(\mathbb{R}^n) $ to $ L^{\frac{n}{\alpha},\infty}(\mathbb{R}^n) $.
Thus, the conclusion follows from the off-diagonal Marcinkiewicz interpolation theorem (see Section 1.4 in \cite{gra}).
\end{proof}

\begin{lemma}\label{lemma-bounedness-2}
Assume that the integral kernel $T(x,y)$ obeys the following pointwise bounds:
\[|T(x,y)|\le C_\varepsilon \langle x\rangle^{-(n-\varepsilon)}\langle |x|-|y|\rangle^{-\varepsilon}{\bf1}_{\{\frac12|y|\le|x|\le 2|y|\}}\] 
for any $\varepsilon\in(0,1).$ Then the operator $T$ is bounded on $L^p(\mathbb{R}^n)$ for all $1< p< \infty.$
\end{lemma}
\begin{proof}
Firstly,  when $\frac{|x|}{2}\le |y|\le2|x|$, for any $\alpha \in(0,1)$, one has
\begin{equation}\label{eq-chi_3}
\langle x\rangle^{-(n-\varepsilon)}\langle |x|-|y|\rangle^{-\varepsilon}\lesssim\begin{cases}
\langle x\rangle^{-\alpha}\langle |x|-|y|\rangle^{-\frac{n-\alpha}{n}}
\langle y\rangle^{-\frac{(n-1)(n-\alpha)}{n}}, \quad \text{when  $\varepsilon=\frac{n-\alpha}{n}$}, & \\[3mm]
\langle x\rangle^{-(n-\alpha)}\langle |x|-|y|\rangle^{-\frac{\alpha}{n}}
\langle y\rangle^{-\frac{(n-1)\alpha}{n}},\quad \text{when  $\varepsilon=\frac{\alpha}{n}$}.& 
\end{cases}
\end{equation}
In addition,  note that
\[\sup_{x\in\mathbb{R}^n}\Big\|\langle |x|-|y|\rangle^{-\frac{n-\alpha}{n}}\langle y\rangle^{-\frac{(n-1)(n-\alpha)}{n}}\Big\|_{L^{\frac{n}{n-\alpha},+\infty}_y}\lesssim 1,\]
and
\[\sup_{x\in\mathbb{R}^n}\Big\|\langle |x|-|y|\rangle^{-\frac{\alpha}{n}}\langle y\rangle^{-\frac{(n-1)\alpha}{n}}\Big\|_{L^{\frac{n}{\alpha},+\infty}_y}\lesssim 1.\]
Then, by the H\"{o}lder-type inequality in Lorentz space and the off-diagonal Marcinkiewicz interpolation theorem, one has that the operator $T$ is bounded in $L^p(\mathbb{R}^n)$ for any $\frac{n}{n-\alpha}<p<\frac{n}{\alpha}$. Since the choice of $ \alpha $ can be  arbitrary in $(0,\,1)$, it follows that $ T $ is bounded in $ L^p(\mathbb {R}^n) $ for any $ 1 < p < \infty $.
\end{proof}

\section{Low energy estimates }\label{sec:low energy}

In this section, we are devoted to studying the $L^p$-boundedness of $W^L$, which is defined in \eqref{eq-wave operator-low-0}. Namely, we will prove the following theorem.
\begin{theorem}\label{thm-low energy-regular-odd}
Let $1\le n\le 4m-1$ be odd. Assume that Assumption \ref{Assumption}, which depends on the resonance type $\mathbf{k}$, holds. Then the following statements hold:

\vspace{5pt}
\indent\emph{(\romannumeral1)} If $\mathbf{k}=0$ (zero is the regular point of $H$), then $W^L\in \mathbb{B}(L^p(\mathbb{R}^n))$ for all $1<p<\infty$;

\vspace{5pt}
\indent\emph{(\romannumeral2)} If  $1\le \mathbf{k}\le k_c$, then $W^L\in \mathbb{B}(L^p(\mathbb{R}^n))$ for all $1<p<\infty;$

\vspace{5pt}
\indent\emph{(\romannumeral3)} If  $k_c+1\le \mathbf{k}\le m_n$, then $W^L\in \mathbb{B}(L^p(\mathbb{R}^n))$ for all $1<p<\frac{n}{n-2m+\mathbf{k}+k_c-1}$;

\vspace{5pt}
\indent\emph{(\romannumeral4)} If $\mathbf{k}=m_n+1$ (zero is an eigenvalue  of $H$), then $W^L\in \mathbb{B}(L^p(\mathbb{R}^n))$ for all $1<p<\frac{2n}{n-1}.$
\end{theorem}
\underline{{\bf The outline of  proof of Theorem \ref{thm-low energy-regular-odd}:}}
To study the $L^p$-boundedness of $W^L$, recall that
\[W^L=\frac{m}{\pi i}\int_0^\infty\lambda^{2m-1}\chi(\lambda)R^+(\lambda^{2m})V\big(R_0^+(\lambda^{2m})-R_0^-(\lambda^{2m})\big)d\lambda.\] 
By the following symmetric resolvent identity:
\[R^+(\lambda^{2m})V=R_0(\lambda^{2m})v(M^+(\lambda))^{-1}v,\]
we can rewrite $W^L$ as 
\begin{equation*}
\begin{split}
W^L=&\frac{m}{\pi i}\int_0^\infty \lambda^{2m-1}\chi(\lambda)R_0^+(\lambda^{2m})v(M^+(\lambda))^{-1}(\lambda)v(R_0^+(\lambda^{2m})-R_0^+(\lambda^{2m}))d\lambda\\
:=&\frac{m}{\pi i}\int_0^\infty G(\lambda)\chi(\lambda)d\lambda.
\end{split}
\end{equation*}
Note that $\supp \chi(\lambda)\subset[-\lambda_0,\lambda_0]$. So if zero energy is the $\mathbf{k}$-th kind resonance of $H=(-\Delta)^m+V$, then by the expansion \eqref{eq-M expansion-odd} of $(M^\pm(\lambda))^{-1}$ in Theorem \ref{thm-M inverse-odd}, i.e. 
\[\big(M^\pm(\lambda)\big)^{-1}=\sum_{i,j\in J_\mathbf{k}}\lambda^{2m-n-i-j}Q_i(M_{i,j}^\pm+\Gamma_{i,j}^\pm(\lambda))Q_j, \]
we can obtain that 
\begin{equation*}
	\begin{split}
		G(\lambda)=&\sum_{i,j\in J_\k}\lambda^{4m-n-1-i-j} R_0^{+}(\lambda^{2m})vQ_i\big(M_{i,j}^++\Gamma_{i,j}^+(\lambda)\big)Q_jv\big(R_0^{+}-R_0^{-}\big)(\lambda^{2m}), \\
		:=&\sum_{i,j\in J_\k} G_{i,j}(\lambda),
	\end{split}
\end{equation*}
The kernels $G_{i,j}(\lambda)(x,y)$  can be represented as
\begin{equation}\label{eqforGij}
	\begin{split}
		G_{i,j}(\lambda)(x,y)=&\int_{\R^n} \lambda^{4m-n-1-i-j} R_0^{+}(\lambda^{2m})(x,z)\Big(vQ_i(M_{i,j}^++\Gamma_{i,j}^+(\lambda))Q_jv\big(R_0^{+}-R_0^{-}\big)(\lambda^{2m})\Big)(z, y)dz\\ 
  =&\lambda^{4m-n-1-i-j}\Big\langle \big(M_{i,j}^++\Gamma_{i,j}^+(\lambda)\big)Q_jv\big(R_0^{+}-R_0^{-}\big)(\lambda^{2m})(\cdot, y),\, Q_ivR_0^{-}(\lambda^{2m})(\cdot,x)\Big\rangle_{L^2},
	\end{split}
\end{equation}
where  $R^{\pm }_0(\lambda^{2m})(x,y)$ are the kernels of resolvent $R^{\pm}_0(\lambda^{2m})$, $\langle f, g\rangle_{L^2}=\int_{\R^n} f \bar{g}$ denotes the inner product on $L^2(\R^n)$ and we used the symmetry fact $R_0^{-}(\lambda^{2m})(\cdot,x)=R_0^{-}(\lambda^{2m})(x,\cdot)$ in the above formula \eqref{eqforGij}.
Then one has that
\begin{equation}\label{eq.rewrifrowl}
	W^L=\frac{m}{\pi i}\sum_{i,j\in J_\k}\int_0^\infty G_{i,j}(\lambda)\chi(\lambda)d\lambda:=\frac{m}{\pi i}\sum_{i,j\in J_\k}	W_{i,j}^L.
\end{equation}
In the sequel, it suffices to study the $L^p$ boundedness of  $W_{i,j}^L$ for each $i,j\in J_\mathbf{k}$. The  kernel $W_{i,j}^L (x,y)$ will further be split into the following three parts:
\begin{equation}\label{decom-W}
\begin{cases}
W_{i,j,1}^L(x,y)=&{\bf1}_{\{|x|>2|y|\}}W_{i,j}^L(x,y),\\
W_{i,j,2}^L(x,y)=&{\bf1}_{\{|x|<\frac 12|y|\}}W_{i,j}^L(x,y),\\
W_{i,j,3}^L(x,y)=&{\bf1}_{\{\frac12|y|\le|x|\le 2|y|\}}W_{i,j}^L(x,y).
\end{cases}	
\end{equation}
We will prove the $L^p$-boundedness of $W_{i,j,1}^L$, $W_{i,j,2}^L$ in Subsection \ref{subsec:3.1} and $W_{i,j,3}^L$ in Subsection \ref{subsec:3.2}. And Theorem \ref{thm-low energy-regular-odd} can be immediately obtained by Proposition\ref{propo-low energy-odd-1}, Proposition \ref{propo-low energy-odd-2} and Proposition \ref{estforw_3} below.

\subsection{The $L^p$-boundedness of $W_{i,j,1}^L$ and $W_{i,j,2}^L$.}\label{subsec:3.1} 
Recall that from \eqref{eqforGij} and \eqref{eq.rewrifrowl}, for each $i,j\in J_{\k}$, one has that 
$$W_{i,j}^L(x,y)=\int_0^\infty G_{i,j}(\lambda)(x,y)\chi(\lambda) d\lambda,$$
and \begin{equation*}
	\begin{split}
		G_{i,j}(\lambda)(x,y)=\lambda^{4m-n-1-i-j}\Big\langle \big(M_{i,j}^++\Gamma_{i,j}^+(\lambda)\big)Q_jv\big(R_0^{+}-R_0^{-}\big)(\lambda^{2m})(\cdot, y),\, Q_ivR_0^{-}(\lambda^{2m})(\cdot,x)\Big\rangle_{L^2}.
	\end{split}
\end{equation*}
By Lemma \ref{lemma-QjvR-odd}, we know that
\[Q_jvR_0^\pm(\lambda^{2m})(\cdot,\,x)=e^{\pm i\lambda |x|}\omega_{j}^{\pm}(\lambda,\cdot,x),\]
and
\[Q_jv\big(R_0^+(\lambda^{2m})-R_0^-(\lambda^{2m})\big)(\cdot,\,y)=e^{ i\lambda |y|}\tilde{\omega}_{j}^{+}(\lambda,\cdot,y)+e^{-i\lambda |x|}\tilde{\omega}_{j}^{-}(\lambda,\cdot,y).\]
Hence, we can represent $G_{i,j}(\lambda)(x,y)$ as 
\begin{equation}\label{eqexpfroGij}
	G_{i,j}(\lambda)(x,y)= e^{i\lambda(|x|+|y|)}T_{i,j}^+(\lambda,x,y)+e^{i\lambda(|x|-|y|)}T_{i,j}^-(\lambda,x,y),
\end{equation}
where
\begin{equation}\label{eq-def-Tij}
	T_{i,j}^{\pm}(\lambda,x,y)=\lambda^{4m-n-1-i-j}\Big\langle \big(M_{i,j}^++\Gamma_{i,j}^+(\lambda)\big)\widetilde{\omega}_{j}^{\pm}(\lambda,\cdot,y),\, \omega_{i}^{-}(\lambda,\cdot,x)\Big\rangle_{L^2}.
\end{equation}
Now, the integral kernels $W_{i,j,1}^L(x,y)$ and $W_{i,j,2}^L(x,y)$ can be written as 
\begin{equation}\label{eq.w-odd-1}
W_{i,j,1}^L(x,y)
=\mathbf{1}_{\{|x|>2|y|\}}\int_0^\infty \Big(e^{i\lambda(|x|+|y|)}T_{i,j}^+(\lambda,x,y)+e^{i\lambda(|x|-|y|)}T_{i,j}^-(\lambda,x,y)\Big)\chi(\lambda)d\lambda,
\end{equation}
\begin{equation}\label{eq.w-odd-2}
	W_{i,j,2}^L(x,y)
	=\mathbf{1}_{\{|x|<\frac 12|y|\}}\int_0^\infty \Big(e^{i\lambda(|x|+|y|)}T_{i,j}^+(\lambda,x,y)+e^{i\lambda(|x|-|y|)}T_{i,j}^-(\lambda,x,y)\Big)\chi(\lambda)d\lambda.
\end{equation}
\subsubsection{The $L^p$-boundedness of $W_{i,j,1}^L$}
\begin{proposition}\label{propo-low energy-odd-1}
Let $1\le n\le 4m-1$ be odd, $0\le \mathbf{k}\le m_n+1$, and assume that $V$ satisfies Assumption \ref{Assumption}. Then, for any $i,j\in J_\mathbf{k}$, we have $W_{i,j,1}^{L}$ is bounded on $L^p(\mathbb{R}^n)$ for all $1<p\le \infty$.
\end{proposition}

We first study the behavior of $T_{i,j}^\pm(\lambda;x,y)$ to prove Proposition \ref{propo-low energy-odd-1}.

\begin{lemma}\label{lemma-T-odd1}
Let $1\le n\le 4m-1$ be odd, $0\le \mathbf{k}\le m_n+1$, and assume that $V$ satisfies Assumption \ref{Assumption}. Then
for any $i,j\in J_\mathbf{k}$, \,$\lambda\in(0,\lambda_0)$ and $l=0,\cdots, \frac{n+3}{2}$, we have 
\begin{equation}\label{eq-kernel estimates-low energy-odd-1}
\big|\partial^l_\lambda T_{i,j}^\pm(\lambda;x,y)\big|\lesssim \lambda^{\frac{n-1}{2}-l}\langle x\rangle^{-\frac{n-1}{2}}.
\end{equation}
\end{lemma}
  
\begin{proof}
Recall that
\[T_{i,j}^{\pm}(\lambda,x,y)=\lambda^{4m-n-1-i-j}\Big\langle \big(M_{i,j}^++\Gamma_{i,j}^+(\lambda)\big)\widetilde{\omega}_{j}^{\pm}(\lambda,\cdot,y),\, \omega_{i}^{-}(\lambda,\cdot,x)\Big\rangle_{L^2},\] 
then by Leibniz's rule and Schwartz inequality,  for any $l=0,1,\cdots, \frac{n+3}{2}$, we have
\begin{equation}\label{eq4.0001}
	\begin{aligned}
		\big|\partial^l_\lambda T_{i,j}^\pm(\lambda;x,y)\big|
		\lesssim \sum_{l_1+l_2+l_3+l_4=l}&\lambda^{4m-n-i-j-1-l_1}
		\left\|\frac{d^{l_2}}{d\lambda^{l_2}}(M_{i,j}^++\Gamma_{i,j}^+(\lambda))\right\|_{\mathbb{B}(L^2)} \\&\times \left\|\partial_\lambda^{l_3}\widetilde{\omega}_{j}^{\pm}(\lambda,\cdot,y) \right\|_{L^2} \left\|\partial_\lambda^{l_4}\omega_{i}^{-}(\lambda,\cdot,x) \right\|_{L^2}.
	\end{aligned}
\end{equation}	

By using \eqref{equ3.47} in Theorem \ref{thm-M inverse-odd}, as well as \eqref{eq-k1 estimate-odd-3} and \eqref{eq-k2-odd-1} in Lemma \ref{lemma-QjvR-odd}, it follows that 
\begin{equation}\left\|\frac{d^{l_2}}{d\lambda^{l_2}}(M_{i,j}^++\Gamma_{i,j}^+(\lambda))\right\|_{\mathbb{B}(L^2)}\lesssim \lambda^{-l_2},\, \, \lambda<\lambda_0,
\end{equation} 
\begin{equation}
	\left\|\partial_\lambda^{l_3}\widetilde{\omega}_{j}^{\pm}(\lambda,\cdot,y)\right\|_{L^2}\lesssim \lambda^{n-2m+\vartheta(j)-l_3}, \ \lambda<\lambda_0.
\end{equation}
and 
\begin{equation}\label{omega_i}
	\left\|\partial_\lambda^{l_4}\omega_{i}^{-}(\lambda,\cdot,x) \right\|_{L^2}\lesssim \lambda^{\min\{\frac{n+1}{2}-2m+\vartheta(i),\ 0\}-l_4}\langle x\rangle^{-\frac{n-1}{2}}, \ \lambda<\lambda_0.
\end{equation}

When $i< 2m-\frac n2$ and $j\le 2m-\frac n2$, we have that
$\frac{n+1}{2}-2m+\vartheta(i)\le 0$ by Remark \ref{vartheta}. Hence it follows from \eqref{eq4.0001}-\eqref{omega_i}  that for each $l=0,1,\cdots, \frac{n+3}{2}$,
\begin{equation*}
	\big|\partial^l_\lambda T_{i,j}^\pm(\lambda;x,y)\big|\lesssim \lambda^{\frac{n-1}{2}+\vartheta(i)+\vartheta(j)-i-j-l}\langle x\rangle^{-\frac{n-1}{2}}	\lesssim \lambda^{\frac{n-1}{2}-l}\langle x\rangle^{-\frac{n-1}{2}}, \ \lambda<\lambda_0,
\end{equation*} 
where in the last inequality we use the facts $\vartheta(i)=\max\{0,\lfloor i+1/2 \rfloor\}\geq i$ and $\vartheta(j)\geq j$.

When  $i=2m-\frac n2$ and $j=2m-\frac n2$, it is easy to obtain that $\vartheta(i)=\vartheta(j)=2m-\frac{n-1}{2}$, and then  $$\frac{n+1}{2}-2m+\vartheta(i)=1, \ \ \  n-2m+\vartheta(j)=\frac{n+1}{2}.$$ Hence from \eqref{eq4.0001}-\eqref{omega_i} again, we have that  
\begin{equation*}
	\big|\partial^l_\lambda T_{i,j}^\pm(\lambda;x,y)\big|\lesssim \lambda^{ 4m-n-i-j-1+\frac{n+1}{2}-l}\langle x\rangle^{-\frac{n-1}{2}}= \lambda^{\frac{n-1}{2}-l}\langle x\rangle^{-\frac{n-1}{2}}, \ \lambda<\lambda_0.
\end{equation*}  

When $i=2m-\frac n2$ and $j<2m-\frac n2$,  since  $M_{i,j}^+=0$ by the \eqref{eq-mij=0} in Theorem \ref{thm-M inverse-odd},  one has that 
\begin{equation}\label{M_ij+gamma_ij}
	\left\|\frac{d^{l_2}}{d\lambda^{l_2}}(M_{i,j}^++\Gamma_{i,j}^+(\lambda))\right\|_{\mathbb{B}(L^2)}=\left\|\frac{d^{l_2}}{d\lambda^{l_2}}\Gamma_{i,j}^+(\lambda)\right\|_{\mathbb{B}(L^2)}\lesssim \lambda^{\frac{1}{2}-l_2}, \, \, \lambda<\lambda_0.
\end{equation}
Then, based on the refined estimate \eqref{M_ij+gamma_ij}, the desired \eqref{eq-kernel estimates-low energy-odd-1} follows from the same reasons as above for this case. 
\end{proof}

To prove Proposition \ref{propo-low energy-odd-1},  
we   also need a result of oscillatory integral estimate.
\begin{lemma}\label{oscillatory estimates}
Let $b>-1$ and $k>b+1$, and assume that $f(\lambda)\in C^k(\mathbb{R})$ such that
\begin{equation}\label{symbol-est}
\Big|\frac{d^l}{d\lambda^l} f(\lambda)\Big|\lesssim \lambda^{b-l}, \quad l=0,1,\cdots,k.
\end{equation}
Then, we have  
\begin{equation}\label{eq.oscillatory est}
\left| \int_0^\infty e^{i\lambda x}f(\lambda)\chi(\lambda)d\lambda\right| \lesssim \langle x\rangle^{-(b+1)},\quad x\in\mathbb{R},
\end{equation}
where $\chi(\lambda)$ is a smooth function such that $\chi\equiv1$ on $(-\lambda_0/2,\,\lambda_0/2)$ and $\supp\chi\subset[-\lambda_0,\,\lambda_0]$ for some fixed $\lambda_0>0$.
\end{lemma}
\begin{proof}
We denote the integral in \eqref{eq.oscillatory est} by $I(x)$. 
If $|x|\le 1$, \eqref{eq.oscillatory est} can be obtained by the simple fact $|I(x)|\lesssim 1$. If $|x|>1$, we split $I(x)=I_0(x)+I_1(x)$, where
\begin{equation*}
	I_i(x)=\int_0^\infty e^{i\lambda x}f(\lambda)\chi(\lambda)\phi_i(\lambda|x|)d\lambda, \quad \text{for } i=0,1,
\end{equation*}
where $\phi_0 \in C^{\infty}(\R)$ such that $\phi_0(\lambda)=1$ when $|\lambda|\leq 1/2$ and $\phi_0(\lambda)=0$ when $|\lambda|\geq 1$, and $\phi_1(\lambda)=1-\phi_0$.
Thus, a direct computation shows that 
$|I_0(x)|\lesssim |x|^{-(b+1)}$. Next, by applying integration by parts $k$ times to $I_1(x)$, and then we have
\begin{equation*}
	|I_1(x)|\lesssim |x|^{-k}\int_{\frac{1}{2}|x|^{-1}}^\infty \lambda^{b-k}d\lambda\lesssim |x|^{-(b+1)}.
\end{equation*}
This completes the proof.
\end{proof}

\begin{proof}[\bf{Proof of Proposition \ref{propo-low energy-odd-1}.}]
Let  $f(\lambda)=\langle x\rangle^{\frac{n-1}{2}}T_{i,j}^\pm(\lambda;x,y)\tilde{\chi}(\lambda)$, where $\tilde{\chi}$ is a smoothing function  such that  $\supp\tilde{\chi}\subset[-\lambda_0,\lambda_0]$ and  $\tilde{\chi}(\lambda) \chi(\lambda)=\chi(\lambda)$.
Then by the estimate \eqref{eq-kernel estimates-low energy-odd-1} in Lemma \ref{lemma-T-odd1},  we can conclude that  $f(\lambda)\in C^{\frac{n+3}{2}}(\R)$ and satisfies the estimate  \eqref {symbol-est} with $b={\frac{n-1}{2}}$  for  each $0\le  l\le k=\frac{n+3}{2}$.  Hence it  easily  follows from  Lemma \ref{oscillatory estimates} that 
$$
\Big|\int_0^\infty e^{i\lambda(|x|\pm|y|)}T_{i,j}^{\pm}(\lambda,x,y)\chi(\lambda)d\lambda\Big|\lesssim\langle x \rangle^{-\frac{n-1}{2}}\langle|x|\pm|y|\rangle^{-\frac{n+1}{2}}. $$
Note that the support of $W_{i,j,1}^{L}(x,y)$ is in $\{(x,y)\in\mathbb{R}^{n}\times\mathbb{R}^n,~|x|>2|y|\}$.
Then the expression \eqref{eq.w-odd-1} immediately indicates that
 $$ \big|W_{i,j,1}^{L}(x,y)\big| \lesssim \langle x \rangle^{-\frac{n-1}{2}}\langle|x|-|y|\rangle^{-\frac{n+1}{2}}\mathbf{1}_{\{|x|>2|y|\}}\lesssim \langle x \rangle^{-n}\mathbf{1}_{\{|x|>2|y|\}}. $$
Thus, according to Lemma \ref{lemma-bounedness-1}, $W_{i,j,1}^{L}$ is bounded on $L^p(\mathbb{R}^n)$ for all $ 1 < p \le  \infty $.
\end{proof}

\subsubsection{The $L^p$-boundedness of $W_{i,j,2}^L$}

\begin{proposition}\label{propo-low energy-odd-2}
Let $1\le n\le 4m-1$ be odd, $0\le \mathbf{k}\le m_n+1$, and assume that $V$ satisfies Assumption \ref{Assumption}. Then, for any $i,j\in J_\mathbf{k}$, we have the following statements:
	\vskip 0.15cm
\indent\emph{(\romannumeral1)} If ~~$i\le \max\{2m-n,\, m-\frac n2\}$, then $W_{i,j,2}^{L}\in \mathbb{B}(L^p(\mathbb{R}^n))$ for all  $1\le p<\infty$;
		\vskip 0.15cm
\indent\emph{(\romannumeral2)} If~~$\max\{2m-n,\, m-\frac n2\}<i<2m-\frac n2$, then $W_{i,j,2}^{L}\in \mathbb{B}(L^p(\mathbb{R}^n))$ for all $1\le p<\frac{n}{n-2m+\vartheta(i)}$;
		\vskip 0.15cm
\indent\emph{(\romannumeral3)} If~~$i=2m-\frac n2$, then $W_{i,j,2}^{L}\in \mathbb{B}(L^p(\mathbb{R}^n))$ for all $1\le p<\frac{2n}{n-1}$.
	\vskip 0.15cm
\noindent Here $\vartheta(i)=\max\{0,\lfloor i+1/2 \rfloor\}$ given in \eqref{thetaj}.
\end{proposition}
\vskip0.3cm

\begin{remark}\label{rmk-on-p}
We point out that the restriction on the range of the exponent $p$ in Theorem \ref{thm-main result} and Theorem \ref{thm-low energy-regular-odd} is purely due to  Proposition \ref{propo-low energy-odd-2}.
\end{remark}

 Let $m_n=\min\{m, \ 2m-\frac{n-1}{2}\}$ and $k_c=\max\{m-\frac {n-1}{2}, 0\}$ defined in \eqref{equ0.1} and \eqref{defkc}, respectively. Suppose that zero is a $\mathbf{k}$-th kind resonance with $0\le \k\le m_n+1$. By the definition of $J_{\k}$ in \eqref{eq-Jk-odd} and \eqref{eq-Jk-odd-2}, if $1\le n\leq 2m-1$,  then $m_n=m$ and $k_c=m-\frac{n-1}{2}$
\begin{equation}\label{eq-Jk-odd'}
J_{\k}:=\begin{cases}\{0,\ 1,\ \cdots,\  m-\frac{n+1}{2},\ {\bf m-\frac{n}{2}}\},\quad&\mbox{if}\,\, \k=0,\\[0.2cm]
	\{0,\ 1,\ \cdots,\  m-\frac{n+1}{2},\  {\bf m-\frac{n}{2}},\  k_c, \ \cdots,\   k_c+\k-1\},\quad&\mbox{if}\,\, 1\le \k\le m,\\[0.2cm]
	 \{0,\ 1,\ \cdots,\  m-\frac{n+1}{2},\  {\bf m-\frac{n}{2}},\  k_c, \ \cdots,\   k_c+m-1,\ {\bf 2m-\frac{n}{2}}\},&\mbox{if}\,\, \k=m+1;
\end{cases}
\end{equation}
If $2m+1\leq n\leq 4m-1$,  then $m_n=2m-\frac{n-1}{2}$ and $k_c=0$, 
\begin{equation}\label{eq-Jk-odd-2'}
	J_{\k}:=\begin{cases}
 \{{\bf m-\frac{n}{2}}\},\quad&\mbox{if}\, \, \k=0,\\[0.2cm]
		\{{\bf m-\frac{n}{2}},\ 0, \ 1, \  \cdots, \ \k-1 \},\quad&\mbox{if}\, \, 1\le \k\le 2m-\frac{n-1}{2},\\[0.2cm]
		\{{\bf m-\frac{n}{2}},\ 0, \ 1, \  \cdots, \ 2m-\frac{n+1}{2},\ {\bf 2m-\frac{n}{2}} \},&\mbox{if}\,\, \k=2m-\frac{n+1}{2}.
	\end{cases}
\end{equation}
 Let $j_{\k}=\max J_{\k}$ denote the maximum of number elements in $J_{\k}$. Then for all odd $n$ and $1\le n\le 4m-1$, 
\begin{equation}\label{MaxJ_k}
\max J_\mathbf{k}=
\begin{cases}
m-\frac{n}{2}, &\quad \text{when}~\mathbf{k}=0,\\
k_c+\mathbf{k}-1,	&\quad \text{when}~0<\mathbf{k}\le k_c=\max\{m-\frac {n-1}{2}, 0\},\\
k_c+\mathbf{k}-1,&\quad \text{when}~k_c<\mathbf{k}\le m_n,\\
2m-\frac{n}{2}, &\quad \text{when}~\mathbf{k}= m_n+1.
\end{cases}
\end{equation}
In particular, we notice that $\max J_{\k}$ increases as $\k$ increases, and
\begin{equation}\label{maxJ_K_c}
\max J_{k_c}=\max\big\{2m-n,\, m-\frac n2\big\}=
\begin{cases}
2m-n, &\quad \text{if }~1\le n\le 2m-1,\\
m-\frac{n}{2},	&\quad \text{if}~\ 2m+1\le n\le 4m-1.
\end{cases}
\end{equation} 
Thus, if $0\le\k\le k_c$, then each $i\in J_\mathbf{\k}$ implies that $i\le\max J_{k_c}=\max\big\{2m-n,\, m-\frac n2\big\}.$ Hence Proposition \ref{propo-low energy-odd-2} (i) shows that  for any $i, j\in J_\mathbf{k}$, $W_{i,j,2}^{L}\in \mathbb{B}(L^p(\mathbb{R}^n))$ for all $1\le p<\infty$. 

If $k_c<\k\le m_n$, then $i\in J_\mathbf{\k}$ means that $\max J_{k_c}<i<2m-\frac n2$ or $i\le \max J_{k_c}$. So by Proposition \ref{propo-low energy-odd-2} (i) and (ii), we always can get that  for each fixed  $i, j\in J_\mathbf{k}$, $W_{i,j,2}^{L}\in \mathbb{B}(L^p(\mathbb{R}^n))$ for all for all $1\le p<n/(n-2m+\vartheta(i))$.  Note that $\vartheta(i)\le \vartheta(k_c+\k-1)= k_c+\k-1$ for each $i\in J_\mathbf{\k}$, Hence we have concluded that if $k_c<\k\le m_n$,  $W_{i,j,2}^{L}\in \mathbb{B}(L^p(\mathbb{R}^n))$ for all $1\le p<n/(n-2m+k_c+\k-1)$ and $i, j\in J_\mathbf{k}$.  

Finally, if $\k=m_n+1$, then Proposition \ref{propo-low energy-odd-2} (i)-(iii) can easily give that $W_{i,j,2}^{L}\in \mathbb{B}(L^p(\mathbb{R}^n))$ for all $1\le p<2n/(n-1)$ and $i, j\in J_{m_n+1}$.

Hence summing up discussions above, we can obtain the following proposition:

\begin{proposition}\label{propo-low energy-odd-2'}
Let $1\le n\le 4m-1$ be odd, $0\le \mathbf{k}\le m_n+1$, and assume that $V$ satisfies Assumption \ref{Assumption}. Then we have the following statements:
	\vskip 0.2cm
\indent\emph{(\romannumeral1)} If $0\le \mathbf{k}\le k_c$, then  for each $i,j\in J_\mathbf{k}$,  $W_{i,j,2}^{L}\in \mathbb{B}(L^p(\mathbb{R}^n))$ for all  $1\le p<\infty$;
		\vskip 0.2cm
\indent\emph{(\romannumeral2)} If~~$k_c+1\le \mathbf{k}\le m_n$, then  for each  $i,j\in J_\mathbf{k}$, $W_{i,j,2}^{L}\in \mathbb{B}(L^p(\mathbb{R}^n))$ for all $1\le p<\frac{n}{n-2m+k_c+\mathbf{k}-1}$;
		\vskip 0.2cm
\indent\emph{(\romannumeral3)} If~~$\mathbf{k}=m_n+1$, then for each $i,j\in J_\mathbf{k}$,  $W_{i,j,2}^{L}\in \mathbb{B}(L^p(\mathbb{R}^n))$ for all $1\le p<\frac{2n}{n-1}$;
\end{proposition}

To proceed, we need another estimate of the kernel $T_{i,j}^\pm(\lambda,x,y)$ to prove Proposition \ref{propo-low energy-odd-2}.
\begin{lemma}\label{lemma-T-odd2} 
Let $1\le n\le 4m-1$ be odd and $0\le \mathbf{k}\le m_n+1$. $T^{\pm}_{i,j}(\lambda,x,y)$ are defined in \eqref{eq-def-Tij}. Assume that $V$ satisfies Assumption \ref{Assumption}. Then
for any $i,j\in J_\mathbf{k}$, \,$\lambda\in(0,\lambda_0)$ and $l=0,\cdots, \frac{n+3}{2}$, the following holds.\\
\indent\emph{(\romannumeral1)} If $i\le \max\{2m-n,\, m-\frac n2\}$, one has 
\begin{equation} \label{eq-kernel estimates-low energy-odd-2-1}
  \big|\partial^l_\lambda T_{i,j}^\pm(\lambda,x,y)\big|\lesssim\lambda^{\frac{n-1}{2}-l}\langle y\rangle^{-\frac{n-1}{2}};  
\end{equation}
\indent\emph{(\romannumeral2)} If $\max\{2m-n,\, m-\frac n2\}<i<2m-\frac n2$, one has 
\begin{equation} \label{eq-kernel estimates-low energy-odd-2-2}
  \big|\partial^l_\lambda T_{i,j}^\pm(\lambda,x,y)\big|\lesssim\lambda^{2m-\frac{n+
					1}{2}-i-l}\langle x \rangle^{2m-n-\vartheta(i)}\langle y\rangle^{-\frac{n-1}{2}};  
\end{equation}
\indent\emph{(\romannumeral3)} If $i=2m-\frac{n}{2}$, one has 
\begin{equation} \label{eq-kernel estimates-low energy-odd-2-3}
  \big|\partial^l_\lambda T_{i,j}^\pm(\lambda,x,y)\big|\lesssim\lambda^{-l}\langle x\rangle^{-\frac{n-1}{2}}\langle y\rangle^{-\frac{n-1}{2}}.  
\end{equation}
\end{lemma}
\begin{proof}	
The proof follows from the similar processes for the proof of Lemma \ref{lemma-T-odd1}. Recall from the \eqref{eq-def-Tij} that \begin{equation*}
	T_{i,j}^{\pm}(\lambda,x,y)=\lambda^{4m-n-1-i-j}\Big\langle \big(M_{i,j}^++\Gamma_{i,j}^+(\lambda)\big)\widetilde{\omega}_{j}^{\pm}(\lambda,\cdot,y),\, \omega_{i}^{-}(\lambda,\cdot,x)\Big\rangle_{L^2}.
\end{equation*}

When $ i\le \max\{2m-n,\, m-\frac n2\}$,   
	\eqref{eq-kernel estimates-low energy-odd-2-1} is obtained by using the estimates
	\eqref{eq-k1 estimate-odd-1} and \eqref{eq-k2-odd-2};	
 
 When $\max\{2m-n,\, m-\frac n2\}<i<2m-\frac n2$, \eqref{eq-kernel estimates-low energy-odd-2-2} is obtained by  using the estimates \eqref{eq-k1 estimate-odd-2} and \eqref{eq-k2-odd-2}. 
 Meanwhile, note that $M_{i,j}^+=0$ when $i=2m-\frac n2$ and $j\neq 2m-\frac n2$. Hence,  we obtain \eqref{eq-kernel estimates-low energy-odd-2-3} by using the estimates \eqref{equ3.47},  \eqref{eq-k1 estimate-odd-2} and \eqref{eq-k2-odd-2} for $i= 2m- \frac n2$.
\end{proof}

Now, we turn to prove Proposition \ref{propo-low energy-odd-2}.
\begin{proof}[\bf{Proof of Proposition \ref{propo-low energy-odd-2}.}]
Recall  that
\[W_{i,j,2}^L(x,y)
	=\mathbf{1}_{\{|x|<\frac 12|y|\}}\int_0^\infty \Big(e^{i\lambda(|x|+|y|)}T_{i,j}^+(\lambda,x,y)+e^{i\lambda(|x|-|y|)}T_{i,j}^-(\lambda,x,y)\Big)\chi(\lambda)d\lambda.\]
Combining Lemma \ref{oscillatory estimates} with Lemma \ref{lemma-T-odd2} , we have
 $$
 |W_{i,j,2}^{L}(x,y)|  \lesssim
 \begin{cases} 
 	\langle y \rangle^{-n}\mathbf{1}_{\{|x|<\frac 12|y|\}} , & \text{if }\ i \leq \max\{2m-n,\, m-\frac n2\}, \\[3mm] 
 	\langle x \rangle^{-(n-2m+\vartheta(i))} \langle y \rangle^{-(2m-i)}\mathbf{1}_{\{|x|<\frac 12|y|\}} , & \text{if }\ \max\{2m-n,\, m-\frac n2\} < i < 2m-\frac{n}{2}, \\[3mm] 
 	\langle x \rangle^{-\frac{n-1}{2}} \langle y \rangle^{-\frac{n+1}{2}}\mathbf{1}_{\{|x|<\frac 12|y|\}} , & \text{if }\ i = 2m-\frac{n}{2}.
 \end{cases}
 $$
Then, by Lemma \ref{lemma-bounedness-11}, the $L^p$-boundedness of $W_{i,j,2}^{L}$ in this proposition can be established.
\end{proof}


\subsection{The $L^p$-boundedness of $W_{i,j,3}^L$}\label{subsec:3.2} 
\begin{proposition}\label{estforw_3}
Let $1\le n\le 4m-1$ be odd, $0\le \mathbf{k}\le m_n+1$, and assume that $V$ satisfies Assumption \ref{Assumption}. Then
for any $i,j\in J_\mathbf{k}$, $W_{i,j,3}^L$ is a bounded operator on $L^p(\mathbb{R}^n)$ for all $1<p<\infty$.	
\end{proposition}
\underline{Reduction of the proof of Proposition \ref{estforw_3}.} 

Recall that from \eqref{eqforGij} and \eqref{eq.rewrifrowl}, for each $i,j\in J_{\k}$, one has that 
$$W_{i,j,3}^L(x,y)=\mathbf{1}_{\{\frac12|y|\le|x|\le 2|y|\}}\int_0^\infty G_{i,j}(\lambda)(x,y)\chi(\lambda) d\lambda,$$
and \begin{equation*}
	\begin{split}
		G_{i,j}(\lambda)(x,y)=\lambda^{4m-n-1-i-j}\Big\langle \big(M_{i,j}^++\Gamma_{i,j}^+(\lambda)\big)Q_jv\big(R_0^{+}-R_0^{-}\big)(\lambda^{2m})(\cdot, y),\, Q_ivR_0^{-}(\lambda^{2m})(\cdot,x)\Big\rangle_{L^2}.
	\end{split}
\end{equation*}
Firstly,   by using the refined formula \eqref{eq3.1000} in Lemma \ref{lem3.10}, 
\begin{equation*}
	\Big(Q_jvR_0^\pm(\lambda^{2m}(\cdot,x)\Big)(z)=e^{\pm i\lambda|x|}\Big(\omega_{j,1}^\pm(\lambda,z,x)+\omega_{j,2}^\pm(\lambda,z,x)+\omega_{j,3}^\pm(\lambda,z,x)\Big).
\end{equation*}
we can rewrite that
\begin{equation*}
	\begin{aligned}
		G_{i,j}(\lambda)(x,y)=&\sum_{s=1}^3
		e^{i\lambda|x|}\lambda^{4m-n-1-i-j}\langle \big(M_{i,j}^++\Gamma_{i,j}^+(\lambda)\big)Q_jv(R_0^+-R_0^-)(\lambda^{2m})(\cdot,y),\, \omega_{i,s}^{-}(\lambda,\cdot,x)\Big\rangle_{L^2},\\
		:=&\sum_{s=1}^3G_{i,j}^{s}(\lambda)(x,y).
	\end{aligned}
\end{equation*}

{\bf For the term $G_{i,j}^{1}(\lambda)(x,y)$.}  Applying \eqref{eq3.1000} again, we then have 
\begin{equation}\label{eq4.10101}
	G_{i,j}^{1}(\lambda)(x,y)=\sum_{t=1}^3\Big(e^{i\lambda(|x|+|y|)}T_{i,j,1,t}^{+}(\lambda,x,y)-e^{i\lambda(|x|-|y|)}T_{i,j,1,t}^{-}(\lambda,x,y)\Big),
\end{equation}
where
\begin{equation}\label{eq4.101012}
	T_{i,j,1,t}^{\pm}(\lambda,x,y)=\lambda^{4m-n-1-i-j}\Big\langle \big(M_{i,j}^++\Gamma_{i,j}^+(\lambda)\big)\omega_{j,t}^{\pm}(\lambda,\cdot,y),\, \omega_{i,1}^{-}(\lambda, \cdot,x)\Big\rangle, \ \ \ ~t=1,2,3.	
\end{equation}
\vskip0.2cm
{\bf For the terms $G_{i,j}^{2}(\lambda)(x,y)$ and $G_{i,j}^{3}(\lambda)(x,y)$}. 

Using \eqref{eq3.900} in Lemma \ref{lemma-QjvR-odd}, we have that
\begin{equation}
	{G}_{i,j}^{s}(\lambda)(x,y)= e^{i\lambda(|x|+|y|)}\tilde{T}_{i,j,s}^{+}(\lambda,x,y)+e^{i\lambda(|x|-|y|)}\tilde{T}_{i,j,s}^{-}(\lambda,x,y),\quad \ \  ~s=2,\,3,
\end{equation}
where
\begin{equation}\label{eq.4.100123}
	\tilde{T}_{i,j,s}^{+}(\lambda,x,y)=\lambda^{4m-n-1-i-j}\Big\langle \big(M_{i,j}^++\Gamma_{i,j}^+(\lambda)\big)\tilde{\omega}_{j}^{\pm}(\lambda,\cdot,y),\, \omega_{i,s}^{-}(\lambda,\cdot,x)\Big\rangle_{L^2}.	
\end{equation}
Thus, the integral kernel  $W_{i,j,3}^{L}(x,y)$  of  $W_{i,j,3}^L$  in \eqref{decom-W} can be written as a linear combination of the following terms:
\begin{equation}\label{eq-def-Iijb}
	I_{i,j}^{\pm,bad}(x,y)=\mathbf{1}_{\{\frac12|y|\le|x|\le 2|y|\}}\int_0^\infty e^{i\lambda(|x|\pm|y|)}T_{i,j,1,1}^{\pm}(\lambda,x,y)\chi(\lambda)d\lambda,
\end{equation}
\begin{equation*}
	I_{i,j,1}^{\pm,good}(x,y)=\mathbf{1}_{\big\{\frac12|y|\le|x|\le 2|y|\big\}}\int_0^\infty e^{i\lambda(|x|\pm|y|)}\big(T_{i,j,1,2}^\pm(\lambda,x,y)+\tilde{T}_{i,j,2}^\pm(\lambda,x,y)\big)\chi(\lambda)d\lambda,
\end{equation*}
and
\begin{equation*}
	I_{i,j,2}^{\pm,good}(x,y)=\mathbf{1}_{\big\{\frac12|y|\le|x|\le 2|y|\big\}}\int_0^\infty e^{i\lambda(|x|\pm|y|)}\big(T_{i,j,1,3}^\pm(\lambda,x,y)+\tilde{T}_{i,j,3}^\pm(\lambda,x,y)\big)\chi(\lambda)d\lambda,
\end{equation*}
where the corresponding operators are denoted by  $I_{i,j}^{\pm,b}$, $I_{i,j,1}^{\pm,g}$ and $I_{i,j,2}^{\pm,g}$.

Consequently, to prove the results in Proposition \ref{estforw_3}, it suffices to establish the  $L^p$-boundedness of the {\bf  good parts} $I_{i,j,1}^{\pm,g}$, $I_{i,j,2}^{\pm,g}$, and the {\bf  bad parts} $I_{i,j}^{\pm,b}$ for all $1<p<\infty$, which are presented in the following two subsections.

\subsubsection{The  good parts $I_{i,j,1}^{\pm,g}$ and $I_{i,j,2}^{\pm,g}$ }\mbox\\
\vskip0.2cm

\begin{proposition}\label{lem.estforwij3g} Under Assumption \ref{Assumption}, for any $i,j\in J_\mathbf{k}$, 
$I_{i,j,1}^{\pm,g}$ and $I_{i,j,2}^{\pm,g}$ are  bounded on $L^p(\mathbb{R}^n)$ for all $1<p<\infty$.
\end{proposition}
We first establish the following results similar to those in Lemma \ref{lemma-T-odd1}. 
\begin{lemma}\label{est for T-3}
Let $1\le n\le 4m-1$ be odd, $0\le \mathbf{k}\le m_n+1$, and assume that $V$ satisfies Assumption \ref{Assumption}. Then,
for any $i,j\in J_\mathbf{k}$, $\lambda\in(0,\lambda_0)$ and $l=0,\cdots, \frac{n+3}{2}$, we have the following estimates.
\vskip0.2cm
\indent\emph{(\romannumeral1)}~ For any $\varepsilon \in [0,1)$, $\tilde{T}_{i,j,2}^\pm(\lambda,x,y)$ and  $T_{i,j,1,2}^\pm(\lambda,x,y)$ given in \eqref{eq4.101012} and \eqref{eq.4.100123} satisfy
\begin{equation} \label{eq.4.10001}
	\big|\partial^l_\lambda T_{i,j,1,2}^\pm(\lambda,x,y)\big| +\big|\partial^l_\lambda \tilde{T}_{i,j,2}^\pm(\lambda,x,y)\big|\le C
	\begin{cases}
		\lambda^{-\varepsilon-l}\langle x\rangle^{-(n-1+\varepsilon)}, \quad &\text{if $i< 2m- \frac n2$}, \\[3mm]
		\lambda^{-l}\langle x\rangle^{-n}, &\text{if $i= 2m- \frac n2$}
	\end{cases}
\end{equation}
where the absolute constant $C$ in the above inequality is independent of $\varepsilon$;
\vskip0.2cm
\indent\emph{(\romannumeral2)}~ In addition, $\tilde{T}_{i,j,3}^\pm(\lambda,x,y)$ and  $T_{i,j,1,3}^\pm(\lambda,x,y)$ given in \eqref{eq4.101012} and \eqref{eq.4.100123} satisfy
\begin{equation} \label{eq.4.10002}
	\big|\partial^l_\lambda T_{i,j,1,3}^\pm(\lambda,x,y)\big| +\big|\partial^l_\lambda 
	\tilde{T}_{i,j,3}^\pm(\lambda,x,y)\big|\lesssim 
	\lambda^{1-l}\langle x\rangle^{-(n-1)}.
\end{equation}
\end{lemma}
\vskip0.3cm
\begin{proof}
Note that by the definition of $\omega_{i,1}^{-}(\lambda,\cdot,x)$ in \eqref{eq3.100001}, for any $l=0,\cdots, \frac{n+3}{2}$,  there is 
\begin{equation*}
\begin{aligned}
  &\left\|\partial^l_\lambda \omega_{i,1}^{-}(\lambda,\cdot,x) \right\|_{L^2}\\
  \lesssim &
  \lambda^{\frac{n+1}{2}-2m+\vartheta(i)-l} \langle x\rangle^{-\frac{n-1}{2}} \left\|\sum_{|\alpha|=|\beta|=\vartheta(j)}A_{\alpha,\beta}^- Q_j(vz^\beta)(y)	\right\|_{L^2_y}\lesssim \lambda^{\frac{n+1}{2}-2m+\vartheta(i)-l} \langle x\rangle^{-\frac{n-1}{2}},  
\end{aligned}
\end{equation*}
which is consistent with the form of the estimate in equation \eqref{eq-k2-odd-1}. Using the  same analysis as that for Lemma \ref{lemma-T-odd1},  the above estimate, together with \eqref{eq-k2-odd-2} and \eqref{eq3.1001},  derives \eqref{eq.4.10001}. Similarly, \eqref{eq-k2-odd-2} combined with \eqref{eq3.10011} derives \eqref{eq.4.10002}.
\end{proof}

\begin{proof}[\bf{Proof of Proposition \ref{lem.estforwij3g}.}]
We initially consider the operator  $I_{i,j,1}^{\pm,g}$.
It follows from Lemma \ref{oscillatory estimates} and the estimate \eqref{eq.4.10001} that for $\frac{1}{2}|y|\le |x|\le2|y|$ and  each $\varepsilon\in [0,1)$,  we have 
\begin{equation*}
|I_{i,j,1}^{\pm,good}(x,y)|\lesssim
\begin{cases}
\langle|x|-|y|\rangle^{-1+\varepsilon}\langle x\rangle^{-(n-1+\varepsilon)}, &\quad i< 2m- \frac n2,\\[2mm]
\langle|x|-|y|\rangle^{-1}\langle x\rangle^{-n}, &\quad i= 2m- \frac n2.
\end{cases}
\end{equation*}
By using Lemma \ref{lemma-bounedness-2}, we can obtain that when $i<2m- \frac n2$, the operator $I_{i,j,1}^{\pm,g}$ are bounded on $L^p(\mathbb{R}^n)$ for $1<p<\infty$. Meanwhile, one can check that $\langle|x|-|y|\rangle^{-1}\langle x\rangle^{-n}$  is admissible in the region $\frac{1}{2}|y|\le |x|\le2|y|$. Thus, $I_{i,j,1}^{\pm,g}$ is bounded on $L^p(\mathbb{R}^n)$ for $1\le p\le\infty$  when $i=2m- \frac n2$. Similarly, according to Lemma \ref{oscillatory estimates} and \eqref{eq.4.10002}, we conclude that $I_{i,j,2}^{\pm,good}(x,y)$ are supported in the region  $\frac{1}{2}|y|\le |x|\le2|y|$ and satisfy 
\begin{equation*}
|I_{i,j,2}^{\pm,good}(x,y)|\lesssim
   \langle|x|-|y|\rangle^{-2}\langle x\rangle^{-(n-1)}, 
\end{equation*}
which is admissible. Therefore, the $I_{i,j,2}^{\pm,g}$ are bounded on $L^p(\mathbb{R}^n)$ for all $1 \le p \le \infty$.
\end{proof}

\subsubsection{The  bad parts $I_{i,j}^{\pm,b}$}

\begin{proposition}\label{lem.estforwij3b1} Under Assumption \ref{Assumption}, for any $i,j\in J_\mathbf{k}$, 
$I_{i,j}^{\pm,b}$ are  bounded on $L^p(\mathbb{R}^n)$ for all $1<p<\infty$.
\end{proposition}

\begin{proof}
Recall that 
\begin{align}\label{bad parts}
I_{i,j}^{\pm,bad}(x,y)=\mathbf{1}_{\big\{\frac12|y|\le|x|\le 2|y|\big\}}\int_0^\infty e^{i\lambda(|x|\pm|y|)}T_{i,j,1,1}^{\pm}(\lambda,x,y)\chi(\lambda)d\lambda,
\end{align}
For simplicity, also denote $I_{i,j}^{\pm,b}(x,y)$ as $I_{i,j}^{\pm,bad}(x,y)$ below. 

Let $\chi_0={\bf{1}}_{\{x\in\mathbb{R}^n;\,|x|\le 1\}}$ and $\chi_s={\bf{1}}_{\{x\in\mathbb{R}^n;\,2^{s-1}<|x|\le 2^s\}}$ for $s\in \N^+$. Define  
\begin{equation*}
	\psi(x,y)=\mathbf{1}_{\big\{\frac12|y|\le|x|\le 2|y|\big\}}-\sum_{|s-h|<2}\chi_s(x)\chi_h(y).
\end{equation*}
Thus,  we have the following  decomposition:
\begin{equation}\label{sumofw3b}
I_{i,j}^{\pm,b}(x,y)=\sum_{|s-h|<2}\chi_s(x)I_{i,j}^{\pm,b}(x,y)\chi_h(y)+ \psi(x,y)I_{i,j}^{\pm,b}(x,y).
\end{equation}
Since both $\frac14|y|\le |x|\le 4|y|$ and $|x|,|y|\lesssim||x|-|y||$ hold in the support of $\psi(x,y)$, applying the same analysis as that used for the good parts $I_{i,j,1}^{\pm,g}(x,y)$ and $I_{i,j,2}^{\pm,g}(x,y)$, it is not difficult to conclude that $\psi(x,y)I_{i,j}^{\pm,b}(x,y)$  is bounded by $O\big(\min\{|x|^{-n}, |y|^{-n}\}\big)$ in the region $\frac14|y|\le |x|\le 4|y|$, which is an admissible part, and thus the corresponding  operator is bounded on $L^p(\mathbb{R}^n)$ for all $1 \le p\le \infty$. 

Hence, to prove this proposition, it suffices to show that the remaining part $\sum_{|s-h|<2}\chi_sI_{i,j}^{\pm,b}\chi_h$ is bounded on $L^p(\mathbb{R}^n)$ for all $1<p<\infty$, which is presented in Proposition \ref{pro: bound for series} below.
\end{proof}


\begin{proposition}\label{pro: bound for series}
  The operator $\sum_{|s-h|<2}\chi_sI_{i,j}^{\pm,b}\chi_h$ is bounded on $L^p(\mathbb{R}^n)$ for all $1<p<\infty$.  
\end{proposition}
\begin{proof}
We first claim that for any $i,j\in J_\mathbf{k}$, there exists a constant $C_p$ independent of all $s,h\in \mathbb{N}_0$ with $|s-h|<2$ such that the operators \begin{equation}\label{claim}
\|\chi_sI_{i,j}^{\pm,b}\chi_h\|_{\mathbb{B}(L^p(\mathbb{R}^n))}\le C_p
\end{equation}
 for all $1<p<\infty$. In fact, once this claim \eqref{claim} is established,  then for all $1<p<\infty$,  
\begin{equation}
\begin{aligned}
 \Big\|\sum_{\substack{s,h\in \mathbb{N}_0\\|s-h|<2}} \chi_sI_{i,j}^{\pm,b}\chi_h f\Big\|^p_{L^p}&=\int_{\mathbb{R}^n} \Big|\sum_{s\in \mathbb{N}_0}  \sum_{|s-h|<2}\chi_sI_{i,j}^{\pm,b}\chi_h f(x) \Big|^p dx\\  
 &=\sum_{s\in \mathbb{N}_0} \int_{\mathbb{R}^n} \Big| \sum_{|s-h|<2}\chi_sI_{i,j}^{\pm,b}\chi_h f(x) \Big|^p dx\\  
 &\lesssim \sum_{s\in \mathbb{N}_0} \sum_{|s-h|<2} \int_{\mathbb{R}^n} \Big| \chi_h f(x) \Big|^p dx \lesssim  \big\| f\big\|^p_{L^p}. 
\end{aligned}
\end{equation}

To prove the estimate \eqref{claim} uniformly for all $s,h\in \mathbb{N}_0$ with $|s-h|<2$, recall that by \eqref{bad parts} and \eqref{eq4.101012}, 
\begin{equation}\label{bad parts1}I_{i,j}^{\pm,b}(x,y)=\int_0^\infty e^{i\lambda(|x|\pm|y|)}T_{i,j,1,1}^{\pm}(\lambda,x,y)\chi(\lambda)d\lambda,
\end{equation}
and
\begin{equation*}
T_{i,j,1,1}^{\pm}(\lambda,x,y)=\lambda^{4m-n-1-i-j}\Big\langle \big(M_{i,j}^++\Gamma_{i,j}^+(\lambda)\big)\omega_{j,1}^{\pm}(\lambda,\cdot,y),\, \omega_{i,1}^{-}(\lambda, \cdot,x)\Big\rangle_{L^2}.
\end{equation*}
 By the expression of $\omega_{i,1}^{\pm}$ in \eqref{eq3.100001}, we can write $T_{i,j,1,1}^{\pm}(\lambda,x,y)$ as 
\begin{equation*}
\sum_{|\alpha|=|\beta|=\vartheta(j)}\sum_{|\tilde{\alpha}|=|\tilde{\beta}|=\vartheta(i)}A_{\alpha,\beta}^\pm A_{\tilde{\alpha},\tilde{\beta}}^-\left( {\Omega}_{\alpha,\beta,\tilde{\beta},\tilde{\alpha},0}(\lambda,x,y)+\Omega_{\alpha,\beta,\tilde{\beta},\tilde{\alpha},1}(\lambda,x,y)\right) ,
\end{equation*}
where
\begin{equation}\label{Tab0}
{\Omega}_{\alpha,\beta,\tilde{\beta},\tilde{\alpha},0}(\lambda,x,y)=\lambda^{\vartheta(i)+\vartheta(j)-i-j} \Big\langle M_{i,j}^+Q_jv(z^\beta),\, Q_iv(z^{\tilde{\beta}}) \Big\rangle \cdot \frac{\mathbf{1}_{\{|y|>1\}}y^\alpha}{|y|^{\frac{n-1}{2}+\vartheta(j)}}	\ \frac{\mathbf{1}_{\{|x|>1\}}x^{\tilde{\alpha}}}{|x|^{\frac{n-1}{2}+\vartheta(i)}},
\end{equation}
 and 
\begin{equation}\label{Tab1}
\Omega_{\alpha,\beta,\tilde{\beta},\tilde{\alpha},1}(\lambda,x,y)=\lambda^{\vartheta(i)+\vartheta(j)-i-j}
\Big\langle\Gamma_{i,j}^+(\lambda)Q_jv(z^\beta),\, Q_iv(z^{\tilde{\beta}}) \Big\rangle \cdot \frac{\mathbf{1}_{\{|y|>1\}}y^\alpha}{|y|^{\frac{n-1}{2}+\vartheta(j)}}	\ \frac{\mathbf{1}_{\{|x|>1\}}x^{\tilde{\alpha}}}{|x|^{\frac{n-1}{2}+\vartheta(i)}}.
\end{equation}

\noindent Thus, $\chi_s(x)I_{i,j}^{\pm,b}(x,y)\chi_h(y)$ can be written as a linear combination of the following terms: 
\begin{equation}\label{omegaq}
\Omega_{q}^\pm(x,\, y)=\chi_s(x)\chi_h(y)\int_0^\infty e^{i\lambda(|x|\pm|y|)}\Omega_{\alpha,\beta,\tilde{\beta},\tilde{\alpha},q}(\lambda,x,y)\chi(\lambda)d\lambda,\quad q=0,\, 1, 	
\end{equation}
where $s,h\in \N_0$  such that $|s-h|\le 2$,    and  $|\alpha|=|\beta|=\vartheta(j)$, $|\tilde{\alpha}|=|\tilde{\beta}|=\vartheta(i)$.

\underline{{\bf For the case of   $q=1$}}.  

From \eqref{equ3.47} and our  decay assumption \eqref{decayonV} on $V$, one  can obtain that for any $l=0,\cdots, \frac{n+3}{2}$,
$$
\left|\partial_\lambda^l{\Omega}_{\alpha,\beta,\tilde{\beta},\tilde{\alpha},1}(\lambda,x,y) \right| \lesssim \lambda^{\frac12-l}\langle x\rangle^{-\frac{n-1}{2}}\langle y\rangle^{-\frac{n-1}{2}},
$$
then by Lemma \ref{oscillatory estimates},  we have  
$$\big|\Omega_{1}^\pm(x,\, y)\big|\lesssim \langle|x|-|y|\rangle^{-\frac 32}\langle x\rangle^{-\frac{n-1}{2}}\langle y\rangle^{-\frac{n-1}{2}}\chi_s(x)\chi_h(y). $$
And using the fact $\frac 14 |y|\le |x| \le 4|y|$ in the support of $\chi_s(x)\chi_h(y)$ for all $|s-h|<2$, one has
$$\big|\Omega_{1}^\pm(x,\, y)\big|\lesssim  \langle|x|-|y|\rangle^{-\frac 32}\min\{\langle x\rangle^{-n+1},\, \langle y\rangle^{-n+1}\},$$
which  are admissible.  Hence  the corresponding operators $\Omega_{1}^\pm$ are uniformly bounded on $L^p(\mathbb{R}^n)$ for all $1\le p\le \infty$  over all $s,h\in \mathbb{N}_0$ with $|s-h|<2$.

\underline{{\bf For the case of   $q=0$}}.

Note that by the definition of $J_\mathbf{k}$ and $\vartheta(i)$ given in \eqref{thetaj}, when $i,j \in J_\mathbf{k}$, one of the conditions $\vartheta(i)+\vartheta(j)-i-j=0$ or $\vartheta(i)+\vartheta(j)-i-j\ge \frac 12$  must hold ( see Remark \ref{thetaj} ). If $\vartheta(i)+\vartheta(j)-i-j\ge \frac 12$, we have 
that for $l=0,\cdots, \frac{n+3}{2}$, then   
$$
\left|\partial_\lambda^l{\Omega}_{\alpha,\beta,\tilde{\beta},\tilde{\alpha},0}(\lambda,x,y) \right| \lesssim \lambda^{\frac12-l}\langle x\rangle^{-\frac{n-1}{2}}\langle y\rangle^{-\frac{n-1}{2}}.
$$
Thus, applying the same analysis as in the proof for the case $q=1$ gives that 
$\|\Omega_{0}^\pm \big\|_{\mathbb{B}(L^p)} \le C$
hold uniformly for all $1\le p\le \infty$ and $|s-h|<2$ when $\vartheta(i)+\vartheta(j)-i-j\neq 0$..
\vskip0.2cm
We are left to consider the case  $\vartheta(i)+\vartheta(j)-i-j=0$.
In the following, we further split the situation into $||x|\pm|y||\le 1$ and $||x|\pm|y||> 1$ to estimate $\Omega_{0}^\pm(x,\, y)$.

Firstly, let $||x|\pm|y||\le 1$.  Note that  when $\vartheta(i)+\vartheta(j)-i-j=0$, by  the definition \eqref{Tab0}, one has
$$
\left|{\Omega}_{\alpha,\beta,\tilde{\beta},\tilde{\alpha},0}(\lambda,x,y) \right| \lesssim \langle x\rangle^{-\frac{n-1}{2}}\langle y\rangle^{-\frac{n-1}{2}}.
$$
Thus,  a direct computation shows that when $|s-h|< 2$, 
\begin{equation}\label{zetasmall}
\begin{aligned}
\Big|\Omega_{0}^\pm(x,\, y)\Big|&\lesssim \langle x\rangle^{-\frac{n-1}{2}}\langle y\rangle^{-\frac{n-1}{2}}\chi_s(x)\chi_h(y),\\
&\lesssim \min\big\{\langle x\rangle^{-n+1},\, \langle y\rangle^{-n+1}\big\},
\end{aligned}	
\end{equation}
which implies that $\Omega_{0}^\pm(x,\, y){\bf{1}}_{\{||x|\pm|y||\le 1\}}$ are  admissible. Therefore, we have that the corresponding operators  are bounded on $L^p(\mathbb{R}^n)$ for $1\le p\le \infty$ with the bound independent of $|s-h|<2$. 

If  $||x|\pm|y||> 1$,  using integration by parts once for $\Omega_{0}^\pm(x,\, y)$,  we get
\begin{align}
\Omega_{0}^\pm(x,\, y)&= C \mathbf{1}_{\{|x|>1\}}\mathbf{1}_{\{|y|>1\}}\frac{\chi_h(y)y^\alpha}{|y|^{\frac{n-1}{2}+\vartheta(j)}}	\frac{\chi_s(x)x^{\tilde{\alpha}}}{|x|^{\frac{n-1}{2}+\vartheta(i)}}\frac{1}{|x|\pm|y|}\label{Omega_{0}}\\
&\ \ \ \ -C\frac{1}{|x|\pm|y|}\mathbf{1}_{\{|x|>1\}}\mathbf{1}_{\{|y|>1\}}\frac{\chi_h(y)y^\alpha}{|y|^{\frac{n-1}{2}+\vartheta(j)}}	\frac{\chi_s(x)x^{\tilde{\alpha}}}{|x|^{\frac{n-1}{2}+\vartheta(i)}}\int_0^\infty e^{i\lambda(|x|\pm|y|)}\chi'(\lambda)d\lambda,\label{Omega_{0}'}
\end{align}
where the constant $C=-i \big\langle M_{i,j}^+Q_jv(z^\beta),\, Q_iv(z^{\tilde{\beta}}) \big\rangle_{L^2}$.

Since $\supp \chi'$ is away from the origin,  by the integrations by parts again, we can show that the second part \eqref{Omega_{0}'} of $\Omega_{0}^\pm(x,\, y)$ is uniformly controlled  by $O\big(\langle|x|\pm |y|\rangle^{-2}\min\{\langle x\rangle^{-n+1},\,\langle y\rangle^{-n+1}\}\big)$ when $||x|\pm|y||> 1$ for all $|s-h|< 2$, and therefore are admissible. 

Finally, it remains to deal with the first part \eqref{Omega_{0}} of $ \Omega_{0}^\pm(x,\, y)$ when $||x|\pm|y||> 1$.  Note that when $s=0$ or $h=0$,
$$\chi_s(x)\chi_h(y)\mathbf{1}_{\{|x|>1\}}\mathbf{1}_{\{|y|>1\}}=0.$$ 
 Hence, we only need to focus on the case when $s>1$, $h>1$ and $|s-h|<2$. Moreover, since $|\alpha|=\vartheta(j)$ and $|\tilde{\alpha}|=\vartheta(i)$,  both functions: $$\mathbf{1}_{\{|y|>1\}}\frac{y^\alpha}{|y|^{\vartheta(j)}} \ \ {\rm and } \ \ \mathbf{1}_{\{|x|>1\}}\frac{x^{\tilde{\alpha}}}{|x|^{\vartheta(i)}}$$ 
 are bounded on $\R^n$. So we can reduce the first part \eqref{Omega_{0}} into $L^p$ boundedness of  the following truncated operator:
$$(T^\pm f)(x)=\int_{||x|\pm|y||> 1}\frac{\chi_h(|y|)}{|y|^{\frac{n-1}{2}}}\ \frac{\chi_s(|x|)}{|x|^{\frac{n-1}{2}}}\ \frac{1}{|x|\pm|y|}\ f(y)dy,$$
where we use the same notation $\chi_h(|y|)=\chi_h(y)$ and $\chi_h(|x|)=\chi_h(x)$. 
Define the sphere mean of $f$ by 
$$(M f)(r)=\int_{S^{n-1}} f(r\omega)d\omega, \ \  dy=r^{n-1}drd\omega.$$
Then, for $1<p<\infty$,  by the polar coordinates transform, one has
\begin{equation}\label{eq-trans-to-hilbert}
\begin{split}
\left\|T^\pm f\right\|_{L^p}^p
=\int_0^\infty&\Big|\int_{|\rho\pm r|>1}\frac{ \chi_h(r)\ M f(r) r^{\frac{n-1}{2}}}{\rho\pm r}dr\Big|^p
\chi_s(\rho)\rho^{n-1-\frac{(n-1)p}{2}}d\rho\\
\lesssim& \int_0^\infty\Big|\int_{|\rho\pm r|>1}\frac{\chi_h(r)\ M f(r)r^{\frac{n-1}{2}}}{\rho\pm r}dr\Big|^p
2^{h(n-1)(1-\frac{p}{2})}\chi_s(\rho)d\rho\\
\lesssim& 2^{h(n-1)(1-\frac{p}{2})}\int_0^\infty\Big|\int_{|\rho\pm r|>1}\frac{\chi_h(r)\ M f(r)r^{\frac{n-1}{2}}}{\rho\pm r}dr\Big|^p d\rho,\\
\end{split}
\end{equation}
 where  we have used the fact that $|s-h|<2$, $s,h\ge 1$ and $r\sim\rho\sim 2^s\sim 2^h$ in the support of $\chi_s(\rho)$.   

Since the truncated Hilbert transforms 
$$\mathbb{H}: g\mapsto \int_{|\rho\pm r|>1}\frac{g(r)}{\rho\pm r}dr,$$
are $L^p$-bounded for $1<p<\infty$, by the left-side of  \eqref{eq-trans-to-hilbert}, it follows that
\begin{equation}\label{eqhilbert}
\begin{split}
\left\|T^\pm f\right\|_{L^p}^p\lesssim &2^{h(n-1)(1-\frac{p}{2})}
\int_{\mathbb{R}} |\chi_h(r)M f(r) r^{\frac{n-1}{2}}|^p dr\\
\lesssim &\int_{\mathbb{R}} |\chi_h(r)M f(r)|^p r^{\frac{n-1}{2}}dr
= \|\chi_h f\|^p_{L^p(\mathbb{R}^n)}\le \|f\|^p_{L^p(\mathbb{R}^n)},
\end{split}
\end{equation}
Therefore, the inequality ~\eqref{eqhilbert},  together with \eqref{zetasmall},  implies that for all $1< p<\infty$,
$$\big\|\Omega_{0}^\pm \big\|_{\mathbb{B}(L^p)} \le C_p$$
holds uniformly for $|s-h|<2$ when $\vartheta(i)+\vartheta(j)-i-j=0$.

Summing up the arguments above, we have concluded that there exists a constant $C_p$ independent of all $s,h\in \mathbb{N}_0$ with $|s-h|<2$ such that the operators $\|\chi_sI_{i,j}^{\pm,b}\chi_h\|_{\mathbb{B}(L^p(\mathbb{R}^n))}\le C_p$ for all $1<p<\infty$.  Hence, the whole proof of Proposition \ref{lem.estforwij3b1} is finished.
\end{proof}

Finally, we return to prove Theorem \ref{thm-low energy-regular-odd}.

\begin{proof}[{\bf Proof of Theorem \ref{thm-low energy-regular-odd}.}]
Recall that by \eqref{eq.rewrifrowl} and \eqref{decom-W}, we can reduce the $L^p$-boundedness of the low energy part $W^{L}$ to the $L^p$-boundedness of $W_{i,j,s}^L$ for $i,\, j\in J_{\k}$ and $s=1,\,2,\,3$.  It is derived from Proposition \ref{propo-low energy-odd-1} and Proposition \ref{estforw_3} that $W_{i,j,1}^L$ and $W_{i,j,3}^L$ are bounded on $L^p(\mathbb{R}^n)$ for $1<p<\infty$, for all $i,\, j\in J_{\k}$ and $s=1,\,3$. Therefore, it is left to consider $W_{i,j,2}^L$. Indeed, from Proposition \ref{propo-low energy-odd-2'}, we immediately conclude that when zero energy is a $\mathbf{k}$-th kind resonance, all of $W_{i,j,2}^L$ belong to  $\mathbb{B}(L^p)$ for the range of $p$ as stated in Theorem \ref{thm-low energy-regular-odd}. 
Thus, we finish the proof of Theorem \ref{thm-low energy-regular-odd}.
\end{proof} 

\section{High energy estimates}\label{sec:high energy}
In this section, we will prove the $L^p$-boundedness of $W^H$. The main result is stated below.
\begin{theorem}\label{thm-high energy-n<2m}
Let $1\le n\le 4m-1$ be odd, and Assumption \ref{Assumption} holds. Then the high energy part $W^H$ is bounded on $L^p(\mathbb{R}^n)$ for all $1\le p\le\infty$.
\end{theorem}

Since Erdo\u{g}an-Green \cite{EG22} has proved that $W^H$ is bounded on $L^p(\R^n)$ for all $1\le p\le\infty$ when $n>2m,$ we only need to prove the case for odd $n$ when $1\le n\le 2m-1$.
Recall that $W^H$ can be represented as follows,
\begin{equation*}
W^H=\frac{m}{\pi i}\int_0^\infty\lambda^{2m-1}(1-\chi(\lambda))R^+(\lambda^{2m})V\big(R_0^+(\lambda^{2m})-R_0^-(\lambda^{2m})\big)d\lambda,
\end{equation*}
where the cut-off function $\chi(\lambda)$ is defined such that $\chi \in C_0^\infty(\mathbb{R}^n)$ and $\chi \equiv 1$ on $(-\lambda_0/2, \lambda_0/2)$, with $\text{supp} \chi \subset [-\lambda_0, \lambda_0]$.  Here, we fix $\lambda_0$, as given in Theorem \ref{thm-M inverse-odd}.
 By the following resolvent identity 
\begin{equation*}
R^+(\lambda^{2m})=R_0^+(\lambda^{2m})-R_0^+(\lambda^{2m})VR^+(\lambda^{2m}),
\end{equation*}
we can divide the high energy portion $W^H$ into two parts as
\begin{equation*}
\begin{split}
W_1^H=&\int_0^\infty\lambda^{2m-1}(1-\chi(\lambda))R_0^+(\lambda^{2m})V(R_0^+(\lambda^{2m})-R_0^-(\lambda^{2m}))d\lambda,\\
W_2^H=&\int_0^\infty\lambda^{2m-1}(1-\chi(\lambda))R_0^+(\lambda^{2m})VR^+(\lambda^{2m})V(R_0^+(\lambda^{2m})-R_0^-(\lambda^{2m}))d\lambda.
\end{split}
\end{equation*}
Therefore, to obtain Theorem \ref{thm-high energy-n<2m}, it suffices to prove the $L^p$-boundedness of $W_1^H$ and $W_2^H$. The related results are presented as follows.

\begin{theorem}\label{propo-high energy-n<2m-1-1}
	Let $1\le n\le 2m-1$ be odd. Assume that $H$ has no embedded positive eigenvalues  and $|V(x)|\lesssim\langle x\rangle^{-\beta}$ with some $\beta>n$. Then the operator $W_1^H$ is bounded on $L^p(\mathbb{R}^n)$ for all $1\le p\le\infty.$
\end{theorem}

\begin{theorem}\label{thm-high energy-2}
	Let $1\le n\le 2m-1$ be odd. Assume that $H$ has no embedded positive eigenvalues  and $|V(x)|\lesssim\langle x\rangle^{-\beta}$ with some $\beta>n+5.$ Then the operator $W_2^H$ is bounded on $L^p(\mathbb{R}^n)$ for all $1\le p\le\infty.$
\end{theorem}

We first prove the result on $W_1^H$. To achieve it, we need the following oscillatory integral estimate.
\begin{lemma}\label{osci-high}
Consider the following oscillatory integral
\begin{equation}\label{eq-osci-high}
I(x)=\int_0^\infty e^{i\lambda x}\lambda^{b}(1-\chi(\lambda))d\lambda, \quad x,~b \in \mathbb{R},~ |x|\neq 0. 	
\end{equation}
For each $b\le -1$ and each integer $N>0$, one has 
$$\left|I(x)\right|\lesssim_N \min\{|\log|x||,\,|x|^{-N} \}.$$
\end{lemma}
\begin{proof}
If $|x|\le 1$, we divide $I(x)$ as 
\begin{equation*}
	I_i(x)=\int_0^\infty  e^{i\lambda x}\lambda^{b}(1-\chi(\lambda))\phi_i(\lambda|x|)d\lambda, \quad \text{for } i=0,1,
\end{equation*}
where $\phi_0 \in C^{\infty}(\R)$ such that $\phi_0(\lambda)=1$ when $|\lambda|\leq 1/2$ and $\phi_0(\lambda)=0$ when $|\lambda|\geq 1$, and $\phi_1(\lambda)=1-\phi_0$.
Then a direct computation shows that 
$$\left| I_0(x)\right| \lesssim_{\lambda_0} \int_{\lambda_0/2}^{|x|^{-1}}\lambda^{-1}d\lambda\lesssim_{\lambda_0}\left|\log|x| \right|.  $$
For $I_1(x)$, applying the integration by parts once, we have 
$$\left| I_1(x)\right| \lesssim_{\lambda_0}  |x|^{-1}\int_{|x|^{-1}}^{\infty}\lambda^{-2}d\lambda\lesssim 1.$$
If $|x|>1$, by using the integration by parts $N$ times to $I(x)$, then we have
\begin{equation*}
	I(x)=|x|^{-N}\int_{\lambda_0/2}^\infty \lambda^{b-1}d\lambda\lesssim_{\lambda_0} |x|^{-N}.
\end{equation*}
Combining these estimates, we complete the proof. 
\end{proof}
\begin{remark}
The above oscillatory integral \eqref{eq-osci-high} should be considered in the distributional sense, that is,  we should introduce a smooth cut-off function $\phi_0$ as before and estimate the integral 
\begin{equation*}
	I_L(x)=\int_0^\infty e^{i\lambda x}\lambda^{b}(1-\chi(\lambda))\phi_0(\lambda/L)d\lambda, \quad x,~b \in \mathbb{R},~ |x|\neq 0, 	
\end{equation*}
uniformly for $L>0$. We omit the discussion on it for convenience.
\end{remark}
At this point, we can prove the $L^p$-boundedness of $W_1^H$ in Theorem \ref{propo-high energy-n<2m-1-1}.

\begin{proof}[\bf{Proof of Theorem \ref{propo-high energy-n<2m-1-1}.}]
Using \eqref{eq-free kernel-F},  we can write the integral kernel $W_1^H(x,y)$ as follows:
$$W_1^H(x,y)=W_1^{H,+}(x,y)-W_1^{H,-}(x,y),$$
where
\begin{equation*}
\begin{split}
W_1^{H,\pm}(x,y)=\int_0^\infty\lambda^{n-2m}(1-\chi(\lambda))\Big(\int_{\mathbb{R}^n}e^{i\lambda|x-z|\pm i\lambda|z-y|}
\frac{F_n^+(\lambda|x-z|)V(z)F_n^\pm(\lambda|z-y|)}{|x-z|^{\frac{n-1}{2}}|z-y|^{\frac{n-1}{2}}}dz\Big)d\lambda.
\end{split}
\end{equation*}
By \eqref{eq-F decay-odd} and Lemma \ref{osci-high},  the oscillatory integral in $\lambda$ above satisfies
\begin{equation*}
\begin{aligned}
&\left| \int_0^\infty e^{i\lambda(|x-z|\pm |z-y|)}\lambda^{n-2m}(1-\chi(\lambda))F_n^+(\lambda|x-z|)F_n^\pm(\lambda|z-y|)d\lambda\right| \\
\lesssim& \min\Big\{\big|\log\big||x-z|\pm |z-y| \big|\big|,\,\big||x-z|\pm |z-y| \big|^{-n-1} \Big\}.	
\end{aligned}
\end{equation*}
Thus, for any $x\in \mathbb{R}^n$, we have  
\begin{equation*}
	\begin{aligned}
\int_{\mathbb{R}^n}|W_{1}^{H,\pm}(x,y)|dx&\lesssim \int_{\mathbb{R}^{2n}}\frac{\min\{\big| \log||x-z|\pm|z-y||\big| ,\,\big||x-z|\pm|z-y|\big|^{-n-1}\}|V(z)|}{|x-z|^\frac{n-1}{2}|y-z|^\frac{n-1}{2}}dxdz \\
&= \int_{\mathbb{R}^{n}} \int_{0}^\infty \frac{\min\{\big| \log\big|\rho\pm|z-y|\big|\big| ,\,\big|\rho\pm|z-y|\big|^{-n-1}\}|V(z)|}{|y-z|^\frac{n-1}{2}} \rho^\frac{n-1}{2} d\rho dz \\
&\lesssim\int_{\mathbb{R}^{n}} \langle z-y\rangle^\frac{n-1}{2} \frac{|V(z)|}{|y-z|^\frac{n-1}{2}}dz<\infty. 
	\end{aligned}
\end{equation*}
This implies that $W_{1}^{H,\pm}$ are bounded on $L^1(\mathbb{R}^n)$ by Shur's test lemma.
On the other hand, $W_{1}^{H,\pm}$ are also bounded on $L^{\infty}(\mathbb{R}^n)$ by duality arguments. Therefore, by the interpolation theory, $W_1^H$ is bounded on $L^p(\mathbb{R}^n)$ for all $1\le p\le\infty$. 
\end{proof}

Next, we show the $L^p$-boundedness of $W_2^H$.  To prove it, we need the following estimate for the resolvent in high energy part (see \cite[Lemma 3.9]{FSWY20}).
\begin{lemma}\label{lemma-high energy-second part-n<2m}
	Assume that $|V(x)|\lesssim \langle x\rangle^{-2(k+1)}$, then for any $\sigma>k+\frac{1}{2},$ $\partial^k_\lambda (R^\pm(\lambda^{2m}))$ is bounded from $L^{2}_\sigma(\mathbb{R}^n)$ to $L^{2}_{-\sigma}(\mathbb{R}^n).$ Furthermore, one has
	\begin{equation}\label{eq-resolvent high energy estimates}
		\Big\|\partial^k_\lambda (R^\pm(\lambda^{2m}))\Big\|_{L^{2}_\sigma(\mathbb{R}^n)\rightarrow L^{2}_{-\sigma}(\mathbb{R}^n)}\lesssim \lambda^{-(2m-1)}.
	\end{equation}
\end{lemma}

\begin{proof}[\bf{Proof of Theorem \ref{thm-high energy-2}.}]
Recall that the operator $W_2^H$ defined by
\[W_2^H=\int_0^\infty\lambda^{2m-1}(1-\chi(\lambda))R_0^+(\lambda^{2m})VR^+(\lambda^{2m})V(R_0^+(\lambda^{2m})-R_0^-(\lambda^{2m})d\lambda.\]
Applying \eqref{eq-free kernel-F}, we write its integral kernel as
\begin{equation*}
\begin{aligned}
W_2^{H}(x,y)=&\frac{m}{\pi i}\int_0^\infty\lambda^{n-2m}(1-\chi(\lambda))\left(  e^{i\lambda(|x|+ |y|)}T^{H,+}(\lambda,x,y)-e^{i\lambda(|x|- |y|)}T^{H,-}(\lambda,x,y)\right)d\lambda \\
:=&W_2^{H,+}(x,y)-W_2^{H,-}(x,y),
\end{aligned}
\end{equation*}	
where 
$$T^{H,\pm}(\lambda,x,y)=\Big\langle R^+(\lambda^{2m})V\frac{e^{\pm i(|y-\cdot|-|y|)}F_n^\pm(\lambda|y-\cdot|)}{|y-\cdot|^\frac{n-1}{2}},\, V\frac{e^{- i(|x-\cdot|-|x|)}F_n^-(\lambda|x-\cdot|)}{|x-\cdot|^\frac{n-1}{2}} \Big\rangle_{L^2}.$$	
Using Lemma \ref{lemma-high energy-second part-n<2m},  a direct computation shows that for any $l=0,1,2,\cdots, \frac{n+3}{2}$, one has 
$$\Big| \partial_\lambda^lT^{H,\pm}(\lambda,x,y)\Big| \lesssim\lambda^{n-4m+1}\langle x\rangle^{-\frac{n-1}{2}}\langle y\rangle^{-\frac{n-1}{2}}, \quad \lambda\ge\lambda_0/2.$$
The validity of the estimate requires the boundedness of the $\frac{n+3}{2}$-order derivative of $R^+(\lambda^{2m})$, from which the decay condition of $V$ in this theorem is derived.	
Note that $n-4m+1 \le -2$ when $1\le n\le 2m-1$. Thus, when $||x|\pm |y||\le 1$, it follows immediately from the expressions of $W_2^{H,\pm}(x,y)$ that
\begin{equation*}
\big| W_2^{H,+}(x,y)\big|+ \big| W_2^{H,-}(x,y)\big|	\lesssim \langle x\rangle^{-\frac{n-1}{2}}\langle y\rangle^{-\frac{n-1}{2}}.	
\end{equation*}
When $||x|\pm |y||\ge 1$, applying the integration by parts $\frac{n+3}{2}$ times, we have 
\begin{equation*}
	\big| W_2^{H,+}(x,y)\big|+ \big| W_2^{H,-}(x,y)\big|	\lesssim\big||x|-|y| \big|^{-\frac{n+3}{2}}  \langle x\rangle^{-\frac{n-1}{2}}\langle y\rangle^{-\frac{n-1}{2}}.
\end{equation*}
These two estimates show that  $W_2^{H,\pm}(x,y)$ are admissible, thereby ensuring that $W_2^H$ is bounded on $L^p(\mathbb{R}^n)$ for all $1 \leq p \leq \infty$.
\end{proof}

     \section{Proof of Theorem \ref{thm-unbound}}\label{sec-unb}
In this section, we will prove Theorem \ref{thm-unbound}. Specially, we have that $W_\pm(H; (-\Delta)^m)$ are unbounded on $L^p(\mathbb{R}^n)$ for all $\frac{n}{n-2m+\mathbf{k}+k_c-1} < p \leq \infty$ if $k_c<\mathbf{k}\leq m_n$, and for all $\frac{2n}{n - 1}< p \leq \infty$ if zero is an eigenvalue of $H$ with a non-zero distributional solution of $H\phi = 0$ in $\bigcap_{s < -\frac{1}{2}}L^{2}_{s}(\R^n)\setminus L^2(\R^n)$. We derive this result using proof by contradiction. The key idea is to reduce the problem to proving the optimality of the following time-decay estimates for $e^{itH}P_{ac}(H)$ in the weighted $L^2$ space, where $P_{ac}(H)$ denotes the projection onto the absolutely continuous spectral subspace of $H$.  

Recall that for all odd $1\le n\le 4m-1$, 
$$m_n=\min\big\{m, \ 2m-\frac{n-1}{2} \big\}, \ \ \ k_c=\max\{m-\frac {n-1}{2}, 0\}.$$ Let
\begin{equation}\label{p_k}
p_{\k}=\begin{cases}
		\frac{n}{n-2m+\mathbf{k}+k_c-1},\, &\text{if}\ k_c<\k \le m_n,\\
		\frac{2n}{n-1},\, &\text{if}\  \k = m_n+1.
	\end{cases}\end{equation}
\begin{proposition}
Let $\sigma>\frac n2 +1$. Under the assumption of Theorem \ref{thm-unbound} for $k_c < \mathbf{k}\le m_n+1$,  there exists a large number $L>0$, a function $\psi\in L^{2}_{\sigma}(\R^n)$ and a constant $C>0$ such that 
	\begin{equation}\label{eq: time decay lower bound}
	\Big|\big\langle e^{i tH}P_{ac}(H)\psi,\, \psi \big\rangle \Big|  \ge C |t|^{-\frac{n}{m}(\frac 12-\frac{1}{p_{\k}})}, \ \ |t|> L.
\end{equation}
\end{proposition}
\begin{proof}
	From Stone's formula, the propagator $e^{i tH}P_{ac}(H)$ can be represented as
	$$e^{i tH}P_{ac}(H)=\frac{m}{\pi i}\int_0^{\infty} e^{  it\lambda^{2m}}\Big(R^{+}(\lambda^{2m})-R^{-}(\lambda^{2m})\Big) \lambda^{2m-1}d\lambda.$$
	We can decompose $e^{i tH}P_{ac}(H)$ into two parts as
	$$\Omega^{L}=\frac{m}{\pi i}\int_0^{\infty} e^{it\lambda^{2m}}\Big(R^{+}(\lambda^{2m})-R^{-}(\lambda^{2m})\Big) \lambda^{2m-1}\chi(\lambda)d\lambda,$$
	and
	$$\Omega^{H}=\frac{m}{\pi i}\int_0^{\infty} e^{it\lambda^{2m}}\Big(R^{+}(\lambda^{2m})-R^{-}(\lambda^{2m})\Big) \lambda^{2m-1}(1-\chi)(\lambda)d\lambda,$$
	where $\chi(\lambda)$ is the same as that in \eqref{eq-wave operator-low-0}. By using the high-energy estimate for resolvent in \eqref{eq-resolvent high energy estimates},  when we perform integration by parts $\frac{n+1}{2}$ times on $\Omega^{H}$, we can obtain a time-decay estimate of $\Omega^{H}$  as $|t|^{-\frac{n+1}{2}}$ in $\mathbb{B}(L^{2}_{\sigma},\, L^{2}_{-\sigma})$. Clearly,  $\frac{n+1}{2}> \frac{n}{m}(\frac 12-\frac{1}{p_{\k}})$, so the contribution of high-energy part $\Omega^H$ becomes negligible for \eqref{eq: time decay lower bound} in large time $t$. Thus, we will focus on the low-energy part $\Omega^L$ below. By using the symmetric second resolvent identity in \eqref{eq: reso-identity},  one has that
	\begin{align*}
		\Omega&^{L}=e^{it(-\Delta)^{m}}\chi(|D|) \\
		&-\frac{m}{\pi i}\int_0^{\infty} e^{it\lambda^{2m}}\Big[R_0^{+}(\lambda^{2m})v\left(M^{+}(\lambda) \right)^{-1}vR_0^{+}(\lambda^{2m}) -R_0^{-}(\lambda^{2m}) v\left(M^{-}(\lambda) \right)^{-1}vR_0^{-}(\lambda^{2m})\Big] \lambda^{2m-1}\chi(\lambda)d\lambda,
	\end{align*}
	where $D=i\nabla$.
	It is well-known (see e.g. \cite{ZYF}) that 
	\begin{equation*}
		\left\|e^{it(-\Delta)^{m}}\chi(|D|)\right\|_{L^{p'}-L^p}\lesssim |t|^{-\frac{n}{m}(\frac {1}{2}-\frac{1}{p})}, \quad t\neq0, \ \ 2\le p\le \infty,
	\end{equation*} 
which is also negligible to the contribution of \eqref{eq: time decay lower bound} if  $p>p_{\k}$ and $t$ large. 
	Hence, it remains to study the contribution of the part associated with the operator 
	$$R_0^{+}(\lambda^{2m})v\left(M^{+}(\lambda) \right)^{-1}vR_0^{+}(\lambda^{2m}) -R_0^{-}(\lambda^{2m}) v\left(M^{-}(\lambda) \right)^{-1}vR_0^{-}(\lambda^{2m})$$
	in $\mathbb{B}(L^{2}_{\sigma},\, L^{2}_{-\sigma})$.

	\textbf{Case I:} \underline{Zero is the $\k$-th kind resonance of $H$ with $k_c<\k\le m_n$}.  
    
    Taking the expansions \eqref{eq-free expansion-odd-2} of $R_0^{\pm}(\lambda^{2m})$  and \eqref{eq-M expansion-odd} of  $M^{\pm}(\lambda)$, and using the cancellation property of $Q_j$ in \eqref{eq-cancellation-Q_j}, it follows that in $\mathbb{B}(L^{2}_{\sigma},\, L^{2}_{-\sigma})$,
	\begin{equation}\label{eq: expan for spectral measure}
		\begin{aligned}
			R_0^{+}(\lambda^{2m})v\left(M^{+}(\lambda) \right)^{-1}&vR_0^{+}(\lambda^{2m}) -R_0^{-}(\lambda^{2m}) v\left(M^{-}(\lambda) \right)^{-1}vR_0^{-}(\lambda^{2m}) \\
			&=b_0^2\, \lambda^{2m-n-2j_{\k}} \, G_{2m-n}vQ_{j_{\k}}\left( M^{+}_{j_{\k},j_{\k}}-M^{-}_{j_{\k},j_{\k}}\right) Q_{j_{\k}}vG_{2m-n}+\Phi(\lambda),	
		\end{aligned}
	\end{equation}
	where $j_{\k}:=\max J_{\k}=k_c+\k-1$ (see \eqref{MaxJ_k}), $G_{2m-n}$ is the operator with the integral kernel $|x-y|^{2m-n}$, and the operator funtion  $\Phi(\lambda)$ satisfies that 
	\begin{equation}\label{eq: est for Phi}
	\left\|\frac{d^l}{d\lambda^{l}}\Phi(\lambda)\right\|_{\mathbb{B}(L^{2}_{\sigma},\, L^{2}_{-\sigma})}\lesssim \lambda^ {2m-n-2j_{\k}+\frac{1}{2}-l},  
	\end{equation}
	for any $l=0,1,\cdots, \frac{n+3}{2}$ and $\lambda\in (0,\lambda_0)$.

Let's first consider the integral part of $\Omega^{L}$ related to $\Phi(\lambda)$ for all $k_c<\k\le m_n$, which is written as 
	$$ \frac{m}{\pi i}\int_0^{\infty} e^{it\lambda^{2m}}\lambda^{2m-1}\Phi(\lambda)\chi(\lambda)d\lambda. $$
	We decompose it into 
	$$\frac{m}{\pi i}\int_0^{\infty} e^{it\lambda^{2m}}\lambda^{2m-1}\Phi(\lambda)\chi(\lambda)\chi(\lambda/|t|^{\frac {1}{2m}})d\lambda+\frac{m}{\pi i}\int_0^{\infty} e^{it\lambda^{2m}}\lambda^{2m-1}\Phi(\lambda)\chi(\lambda)\big(1-\chi(\lambda/|t|^{\frac {1}{2m}})\big)d\lambda. $$
	By a direct calculation,  the first integral contributes the time decay as $|t|^{-(\frac{4m-n-2j_{\k}}{2m}+\frac{1}{4m})}$. And using integration by parts $\frac{n+3}{2}$ times, the second integral also contributes $|t|^{-(\frac{4m-n-2j_{\k}}{2m}+\frac{1}{4m})}$. Therefore, the part of $\Omega^{L}$ related to $\Phi(\lambda)$ has the time decay as $|t|^{-(\frac{4m-n-2j_{\k}}{2m}+\frac{1}{4m})}$. 

Note that $p_{\k}=\frac{n}{n-2m+\mathbf{k}+k_c-1}$ by  \eqref{p_k} for $k_c<\k\le m_n$, and since $j_{\k}=k_c+\k-1$,  one has 
$$\frac{4m-n-2j_{\k}}{2m}+\frac{1}{4m}=\frac{n}{m}\big(\frac 12-\frac{1}{p_{\k}}\big)+\frac{1}{4m}>\frac{n}{m}\big(\frac 12-\frac{1}{p_{\k}}\big),$$ 
which means that the part of $\Omega^{L}$ related to $\Phi(\lambda)$ is negligible for \eqref{eq: time decay lower bound} as $|t|$ is large.
	
	Next, we consider the principal term in \eqref{eq: expan for spectral measure}. Firstly, we need to show that the principal term $$b_0^2G_{2m-n}vQ_{j_{\k}}\left( M^{+}_{j_{\k},j_{\k}}-M^{-}_{j_{\k},j_{\k}}\right) Q_{j_{\k}}vG_{2m-n}$$ in \eqref{eq: expan for spectral measure} does not vanish in $\mathbb{B}(L^{2}_{\sigma},\, L^{2}_{-\sigma})$.
	Indeed, we first have $Q_{j_{\k}}\neq 0$ by Proposition \ref{prop-resonance-odd} if zero is the $\k$-th kind of resonance of $H$ with $k_c<\k\le m_n$, and  $Q_{j_{\k}}$ is a finite rank projection on $L^2$ by the definition \eqref{eq-Q_j-odd}. Thus, it can be written as 
	$$
	Q_{j_{\k}}f= \sum_{s = 1}^{\tau}\langle f,\, \psi_{j_{\k}}^s\rangle\psi_{j_{\k}}^s. 
	$$ 
	Here, the set $\{\psi_{j_{\k}}^s~|~s = 1,\cdots,\tau\}$ is the orthonormal basis of $Q_{j_{\k}}L^2$. Since $\k >k_c$, one has that $Q_{j_{\k}}T_0 = 0$ by the property in  \eqref{S_jT_0}, which implies  $Q_{j_{\k}}(U+b_0vG_{2m - n}v)U \psi_{j_{\k}}^s=0$. Hence, we have 
	$$
	b_0Q_{j_k}vG_{2m - n}vU \psi_{j_{\k}}^s= -\psi_{j_{\k}}^s.
	$$
	Note that $vU \psi_{j_{\k}}^s \in L^{2}_{\sigma}$ by the decay assumption \eqref{decayonV} on $V$. Thus, to prove the claim, it suffices to prove that $Q_{j_{\k}}\left( M^{+}_{j_{\k},j_{\k}}-M^{-}_{j_{\k},j_{\k}}\right)Q_{j_{\k}}$ does not vanish. From \eqref{eq: inverse of D^pm} in the proof of Theorem \ref{thm-M inverse-odd}, we have 
	$$
	M^{+}_{j_{\k},j_{\k}}-M^{-}_{j_{\k},j_{\k}}= \left( e^{\frac{\pi \mathrm{i}}{2m}(n - 2m + 2{j_{\k}})}-e^{-\frac{\pi \mathrm{i}}{2m}(n - 2m + 2{j_{\k}})}\right) M_{j_{\k},j_{\k}}.
	$$
	Here, $M_{j_{\k},j_{\k}}$ is the operator in the diagonal line of $d^{-1}$ given in \eqref{eq: struce of the inverse of D}. It is clear that 
	$$
	\left( e^{\frac{\pi \mathrm{i}}{2m}(n - 2m + 2{j_{\k}})}-e^{-\frac{\pi \mathrm{i}}{2m}(n - 2m + 2{j_{\k}})}\right) \neq 0.
	$$ 
	Because $d$ is rigidly negative definite on  $\bigoplus_{j\in J_\k''}Q_{j}L^2$, $M_{j_{\k},j_{\k}}$ is also rigidly negative definite on $Q_{j_{\k}}L^2$. As a result, $Q_{j_{\k}}\left(M^{+}_{j_{\k},j_{\k}}-M^{-}_{j_{\k},j_{\k}}\right)Q_{j_{\k}}$ also does not vanish. Moreover, let $\psi=vU\psi_{j_{\k}}^1 \in L^{2}_{\sigma}$. Then one has that
	\begin{equation}\label{eq: no vanish}
	\begin{aligned}
	\big\langle b_0^2 G_{2m-n}vQ_{j_{\k}}\left( M^{+}_{j_{\k},j_{\k}}-M^{-}_{j_{\k},j_{\k}}\right)& Q_{j_{\k}}vG_{2m-n} \psi, \,  \psi\big\rangle \\
	=&\left( e^{\frac{\pi \mathrm{i}}{2m}(n - 2m + 2{j_{\k}})}-e^{-\frac{\pi \mathrm{i}}{2m}(n - 2m + 2{j_{\k}})}\right) \big\langle M_{j_{\k},j_{\k}}\psi_{j_{\k}}^1,\, \psi_{j_{\k}}^1 \big\rangle  \neq 0 .
	\end{aligned}		
	\end{equation}

	Now we come back to consider the contribution of the part of $\Omega^{L}$ related to the principal term 
	$$b_0^2\, \lambda^{2m-n-2j_{\k}} \, G_{2m-n}vQ_{j_{\k}}\left( M^{+}_{j_{\k},j_{\k}}-M^{-}_{j_{\k},j_{\k}}\right) Q_{j_{\k}}vG_{2m-n}.$$
	This part can be written as 
	$$
	I_{M}(t)= \frac{m }{\pi i}\Big(\int_0^{\infty} e^{it\lambda^{2m}} \lambda^{4m-n-2j_{\k}-1} \, \chi(\lambda) d\lambda\Big) \cdot 	\big\langle b_0^2 G_{2m-n}vQ_{j_{\k}}\left( M^{+}_{j_{\k},j_{\k}}-M^{-}_{j_{\k},j_{\k}}\right) Q_{j_{\k}}vG_{2m-n} \psi, \,  \psi\big\rangle.
	$$
	To prove \eqref{eq: time decay lower bound}, it suffices to show that $I_{M}(t)$ satisfies \eqref{eq: time decay lower bound}. By using Lemma \ref{lem: the lower bound of osci int} stated later and the fact \eqref{eq: no vanish},   we have 
	$$\left|I_{M}(t)\right|\sim |t|^{-\frac{4m-n-2j_{\k}}{2m}}=|t|^{-\frac{n}{m}(\frac 12-\frac{1}{p_{\k}})},$$ 
	when $|t|$ is sufficiently large.
	Therefore, we conclude the estimate \eqref{eq: time decay lower bound} for the case of $k_c<\mathbf{k} \le m_n$. 

\textbf{Case II:} \underline{Zero  is the eigenvalue of $H$ (i.e. $\k= m_n+1$)}. 

By utilizing the expansions \eqref{eq-free expansion-odd-2} of $R_0^{\pm}(\lambda^{2m})$, \eqref{eq-M expansion-odd} of  $M^{\pm}(\lambda)$,  the cancellation property of $Q_j$ in \eqref{eq-cancellation-Q_j}, as well as the equalities \eqref{eq-mij=0}, then  
	\begin{equation}\label{eq: expan for spectral measure-1}
		\begin{aligned}
			&R_0^{+}(\lambda^{2m})v\left(M^{+}(\lambda) \right)^{-1}vR_0^{+}(\lambda^{2m}) -R_0^{-}(\lambda^{2m}) v\left(M^{-}(\lambda) \right)^{-1}vR_0^{-}(\lambda^{2m}) \\
			=&b_0^2\, \lambda^{1-2m} \, G_{2m-n}vQ_{2m-\frac{n+1}{2}}\left( M^{+}_{2m-\frac{n+1}{2},2m-\frac{n+1}{2}}-M^{-}_{2m-\frac{n+1}{2},2m-\frac{n+1}{2}}\right) Q_{2m-\frac{n+1}{2}}vG_{2m-n} \\
			+&b_0^2\, \lambda^{\frac12-2m} \, G_{2m-n}vQ_{2m-\frac{n+1}{2}}\left( \Gamma^{+}_{2m-\frac{n+1}{2},2m-\frac{n}{2}}(\lambda)-\Gamma^{-}_{2m-\frac{n+1}{2},2m-\frac{n}{2}}(\lambda)\right) Q_{2m-\frac{n}{2}}vG_{2m-n} \\
			+&b_0^2\, \lambda^{\frac12-2m} \, G_{2m-n}vQ_{2m-\frac{n}{2}}\left( \Gamma^{+}_{2m-\frac{n}{2},2m-\frac{n+1}{2}}(\lambda)-\Gamma^{-}_{2m-\frac{n}{2},2m-\frac{n+1}{2}}(\lambda)\right) Q_{2m-\frac{n+1}{2}}vG_{2m-n} \\
			+&b_0^2\, \lambda^{-2m} \, G_{2m-n}vQ_{2m-\frac{n}{2}}\left( \Gamma^{+}_{2m-\frac{n}{2},2m-\frac{n}{2}}(\lambda)-\Gamma^{-}_{2m-\frac{n}{2},2m-\frac{n}{2}}(\lambda)\right) Q_{2m-\frac{n}{2}}vG_{2m-n} \\
			+&\Phi(\lambda),	
		\end{aligned}
	\end{equation}
where the estimate \eqref{eq: est for Phi} for $\Phi(\lambda)$ still holds.
In the above identity, we have used the fact $$M^+_{2m-\frac{n}{2},2m-\frac{n}{2}}=M^-_{2m-\frac{n}{2},2m-\frac{n}{2}}=(b_1Q_{2m-\frac{n}{2}}vG_{4m-n}vQ_{2m-\frac{n}{2}})^{-1},$$ which is independent of $\pm$ from Theorem \ref{thm-M inverse-odd}.  \textbf{If  $\mathbf{k}=m_n + 1$ and there exists  a non-zero distributional solution of $H\phi = 0$ in $\bigcap_{s < -\frac{1}{2}}L^{2}_{s}(\mathbb{R}^n)\setminus L^2(\mathbb{R}^n)$}, then from Definition \ref{defi-resonance} and Proposition \ref{prop-resonance-odd}, one has $Q_{2m-\frac{n+1}{2}}\neq 0$. Hence, we can take $\psi=vU\psi_{2m-\frac{n+1}{2}}^1\in L^{2}_{\sigma}$, where $\psi_{2m-\frac{n+1}{2}}^1$ is a function in $Q_{2m-\frac{n+1}{2}}L^2$. Moreover, from \eqref{eq: expan for spectral measure-1}, we obtain that
		\begin{align*}
		&\langle R_0^{+}(\lambda^{2m})v\left(M^{+}(\lambda) \right)^{-1}vR_0^{+}(\lambda^{2m}) -R_0^{-}(\lambda^{2m}) v\left(M^{-}(\lambda) \right)^{-1}vR_0^{-}(\lambda^{2m})\psi,\, \psi\rangle\\
		=&\big\langle b_0^2 G_{2m-n}vQ_{j_{2m-\frac{n+1}{2}}}\left( M^{+}_{2m-\frac{n+1}{2},2m-\frac{n+1}{2}}-M^{-}_{2m-\frac{n+1}{2},2m-\frac{n+1}{2}}\right) Q_{2m-\frac{n+1}{2}}vG_{2m-n} \psi, \,  \psi\big\rangle+ \big\langle \Phi(\lambda)\psi,\, \psi \big\rangle,  
	\end{align*}
since  $Q_{2m-\frac{n+1}{2}}$ and $Q_{2m-\frac{n}{2}}$ are orthogonal.	In the same spirit as the case $\mathbf{k}=m_n$, the principal term of the above identity  doesn't vanish.
Hence the  method to obtain the inequality \eqref{eq: time decay lower bound} also works similarly for the case II (i.e. $\mathbf{k}=m_n+1$). Thus, we have completed the whole proof.
\end{proof}

\begin{remark} In fact, we notice that  a non-zero distributional solution $\phi$ to $H\phi = 0$ exists in the set $$\bigcap_{s < -\frac{1}{2}}L^{2}_{s}(\mathbb{R}^n)\setminus L^2(\mathbb{R}^n)$$ if and only if $\psi = Uv\phi\neq 0$ belongs to $Q_{2m-\frac{n + 1}{2}}L^2$. Furthermore, in the specific case $m = 2$ and $n = 3$, such a function $\phi$ is also called by a $p$-wave function (see \cite{MWY23-1}).
\end{remark}

Here, we list an asymptotical estimate of an oscillatory integral used above ( see e.g. Section 5.1 (d) in Chapter 8 of \cite{Stein-1993} ).
\begin{lemma}\label{lem: the lower bound of osci int}
	Suppose that $\psi\in C_0^\infty(\R)$ and $\Re(\mu)>-1$. Then, as $x$ tends to $\infty$, we have 
	$$\int_{0}^{\infty} e^{i\lambda x}\psi (\lambda) \lambda^{\mu}d\lambda\sim \sum_{j=0}^{\infty} a_j x^{-j-1-\mu}, \quad a_j=i^{j+\mu+1} \frac{j!}{\Gamma(j+\mu+1)}\psi^{(j)}(0),$$
	where $\Gamma(x)$ denotes the Gamma function.
\end{lemma}

Finally, we come to prove Theorem \ref{thm-unbound}.
\begin{proof}[{\bf Proof of Theorem \ref{thm-unbound}}]
We prove the results by contradiction.
Assume that zero is the $\mathbf{k}$-th kind of resonance with $k_c<\mathbf{k}\le m_n$, or zero is an eigenvalue of $H$ with a non-zero distributional solution of $H\phi = 0$ in $\bigcap_{s < -\frac{1}{2}}L^{2}_{s}(\R^n)\setminus L^2(\R^n)$. Also, assume that the wave operators $W_{\pm}(H;(-\Delta)^m)$ are bounded on $L^{\bf{q}}$ for some $\bf{q}$ such that $p_{\mathbf{k}} < \bf{q} $. Then, by the intertwining property in \eqref{eq-intertwing identity} and \eqref{eq: for free} in \cite{ZYF}, we have
\begin{equation*}
	\Big\|e^{i tH}P_{ac}(H)\Big\|_{L^{{\bf{q}}^{\prime}}-L^{\bf{q}}}\lesssim |t|^{-\frac{n}{m}(\frac{1}{2}-\frac{1}{\bf{q}})}, \quad t\neq0.
\end{equation*}
By applying the H\"{o}lder inequality, the local decay estimate
\begin{equation}\label{eq: disper est in weighted L2}
	\Big\|e^{i tH}P_{ac}(H)\Big\|_{L^{2}_{\sigma}-L^{2}_{-\sigma}}\lesssim |t|^{-\frac{n}{m}(\frac{1}{2}-\frac{1}{\bf{q}})}, \quad t\neq0,
\end{equation}
holds for any $\sigma>\frac{n}{2}+1$. Combining this estimate with \eqref{eq: time decay lower bound}, there exists a function $\psi \in L^2_\sigma$ such that
$$ |t|^{-\frac{n}{m}(\frac 12-\frac{1}{p_{\k}})} \lesssim \Big|\big\langle e^{i tH}P_{ac}(H)\psi,\, \psi \big\rangle \Big|  \lesssim  |t|^{-\frac{n}{m}(\frac{1}{2}-\frac{1}{\bf{q}})}, $$
for $|t|\ge L$.
However, this estimate contradicts the assumption $\mathbf{q}>p_{\k}$. Therefore, we conclude that $W_{\pm}$ are unbounded on $L^{\bf{q}}$ for each $\bf{q}$ such that $p_{\mathbf{k}} < \bf{q} \le \infty$.  
\end{proof}

\section{Proof of  Lemma \ref{lemma-QjvR-odd} and  \ref{lem3.10}}\label{sec:proof}

In this section, we will give the proof of Lemma \ref{lemma-QjvR-odd} and Lemma \ref{lem3.10}. Before the proof, we first give some lemmas on the integral estimate which will be used in the rest of this section. 

The following lemma can be seen in \cite[Lemma 3.8]{GV}. 
\begin{lemma}\label{lem-est-integral-1}
	Let $n\ge 1$. Then there is some absolute constant $C>0$ such that
	\begin{equation*}\label{eq-integral}
		\int_{\mathbb{R}^n}|x-y|^{-k}\langle y\rangle^{-l}\,\d y\lesssim\langle x\rangle^{-\min\{k,\, k+l-n\}}, \quad x\in\mathbb{R}^n,
	\end{equation*}
	provided  $l\ge 0$, $0\le k<n$ and $k+l>n$.
\end{lemma}

The following lemma (see \cite[Lemma 4.2]{CHHZ-2024}) deals with functions satisfying certain vanishing moments, which will be used to study the property of $Q_jvR_0^\pm(\lambda^{2m})(\cdot, x)$ and $Q_jv\big(R_0^+(\lambda^{2m})(\cdot, x)-R_0^-(\lambda^{2m})(\cdot, x)\big)$.
\begin{lemma}\label{lemma-chhz} 
	Let $l\in\mathbb{Z}$, $l\ge \frac{1-n}{2}$ and $j\in\mathbb{N}_0$. Suppose that  $f\in L_\sigma^2(\mathbb{R}^n)$ with $\sigma >\max\{j,l\}+\frac{n}{2}$, and $\int_{\R^n}x^\alpha f(x)dx=0$ for all~ $|\alpha|\le j-1$, then for any $x\in\mathbb{R}^n$, one has
	\begin{equation}
		\Big|\int_{\R^{n}}|x-y|^lf(y)dy\Big|\lesssim\langle x\rangle^{l-j} \|f\|_{L^2_\sigma}, \quad x\in \mathbb{R}^n.
	\end{equation}
\end{lemma}
Then, we immediately obtain the following proposition. 
\begin{proposition}\label{corcancell}
	Assume that $V$ satisfies \eqref{decayonV} in Assumption \ref{Assumption}, $j\in J_{\mathbf{k}}$, $Q_j$ is the projection given by \eqref{eq-Q_j-odd},  and $l\in\mathbb{Z}$ such that $\frac{1-n}{2}\le l\le j+1$. Then we have 
	\begin{equation}\label{eq-kernel estimate-odd}
		\Big\|Q_j(v(\cdot)|\cdot-y|^l)(x) \Big\|_{L_x^2(\mathbb{R}^n)}\lesssim\langle y\rangle^{l-\vartheta(j)}, \quad x\in \mathbb{R}^n,
	\end{equation}
	where $\vartheta(j)=\max\{0, \lfloor j+\frac{1}{2}\rfloor\}$ given in \eqref{thetaj}.
\end{proposition}
\begin{proof}
	By the definition of $Q_j$, we know that $Q_j(x^{\alpha}v)=0$ for $|\alpha|\le \vartheta(j)-1$. Thus, one has 
	\[\int_{\R^n}x^\alpha v(x)Q_j(v(\cdot)|\cdot-y|^l)(x)dx=0.\]
	Then by Lemma \ref{lemma-chhz}, we have
	\begin{align*}
		\left\|Q_j(v(\cdot)|\cdot-y|^l)(x)\right\|_{L_x^2}^2 
		&=\int_{\R^n} |x-y|^lv(x)Q_j(v(\cdot)|\cdot-y|^l)(x)dx\\
		&\lesssim \left\|v(x)Q_j\big(v(\cdot)|\cdot-y|^l\big)(x)\right\|_{L_{\sigma}^2(\mathbb{R}^n_x)} \left\langle  y\right\rangle ^{l-\vartheta(j)}\\
		&\lesssim \Big\|Q_j(v(\cdot)|\cdot-y|^l)(x)\Big\|_{L^2}\left\langle  y\right\rangle ^{l-\vartheta(j)},
	\end{align*}
	where $\sigma=\max\{j,l\}+\frac{n}{2}+\varepsilon$ for some sufficiently small $\varepsilon>0$. In the last inequality above, we have used the fact that $|v|\lesssim \langle x\rangle^{-\beta/2} \lesssim \langle x\rangle^{-\sigma}$ from \eqref{decayonV} in the assumption \ref{Assumption}. Hence, the estimate \eqref{eq-kernel estimate-odd} holds.  
\end{proof}

\subsection{Proof of Lemma \ref{lemma-QjvR-odd}}\label{secA.1}
\indent We will prove the conclusions on the properties of $Q_jvR_0^\pm(\lambda^{2m})(\cdot,\,y)$ and $\big(Q_jvR_0^+(\lambda^{2m})(\cdot,\,y)- Q_jvR_0^-(\lambda^{2m})(\cdot,\,y)\big)$ in Lemma \ref{lemma-QjvR-odd} by applying the property of free resolvents and Proposition \ref{corcancell}.

We first prove \eqref{eq-QjR kernel-odd}-\eqref{eq-k1 estimate-odd-2}.
Set
\begin{equation*}
	\omega_{j}^{\pm}(\lambda,x,y)=e^{\mp i\lambda|y|}\Big(Q_jvR_0^\pm(\lambda^{2m})(\cdot,\,y)\Big)(x).
\end{equation*}
Then we have \eqref{eq-QjR kernel-odd}. To establish estimates \eqref{eq-k1 estimate-odd-1}-\eqref{eq-k1 estimate-odd-2},
the proof is divided into two cases: $\lambda\langle y\rangle\le 1$ and $\lambda\langle y\rangle\ge 1$. 
\vskip0.2cm
\textbf{\emph{Case 1: $\lambda\langle y\rangle\le 1$.}}
\vskip0.2cm
For each $j\in J_\mathbf{k}$, recall that $Q_j(x^{\alpha}v(x))=0$ for all $|\alpha|\le\vartheta(j)-1$ where $\alpha\in \mathbb{N}_0^n$ . Thus, $Q_j(v(\cdot)|\cdot-y|^{2i})=0$  for $2i\le \vartheta(j)-1$, $i\in\mathbb{N}_0$ and a fixed parameter variable $x$.
By applying $Q_j v$ to the expansion \eqref{eq-free expansion-odd-1} in Lemma \ref{lemma-free resolvent expansion} with $\theta=\min \{\vartheta(j),\, 2m-n\}$, i.e., 
\begin{equation*}
	R_0^\pm(\lambda^{2m})(\cdot,y)
	=\sum_{0\le i\le\lfloor \frac{\theta-1}{2}\rfloor}a_i^\pm\lambda^{n-2m+2i}|\cdot-y|^{2i}+\lambda^{n-2m}r_{\theta}^\pm(\lambda|\cdot-y|);
\end{equation*}
Then
\begin{equation*}
	\Big(Q_jvR_0^\pm(\lambda^{2m})(\cdot,\,y)\Big)(x)=\lambda^{n-2m}\Big(Q_jvr_\theta^\pm(\lambda|\cdot-y|)\Big)(x).
\end{equation*}
We remark that $v(\cdot)r_\theta^\pm(\lambda|\cdot-y|)$ is in $L^2$ by \eqref{eq-r estimate-odd} and the decay assumption on $V$ in \eqref{decayonV}. 
Now, we can rewrite $\omega_{j}^{\pm}(\lambda,x,y)$ as follows
$$\omega_{j}^{\pm}(\lambda,x,y)=e^{\mp i\lambda|y|}\lambda^{n-2m}\Big(Q_jvr_\theta^\pm(\lambda|\cdot-y|)\Big)(x).$$ 
Note that Taylor expansion for $r_\theta^\pm$  yields
\begin{equation*}
	r_\theta^\pm(\lambda|x-y|)=\sum_{s=0}^{\vartheta(j)-1}\big(r_\theta^\pm\big)^{(s)}(\lambda|y|)\ \frac{(\lambda|x-y|-\lambda|y|)^s}{s!}+\tilde{r}_{\vartheta(j)}^\pm(\lambda,|x-y|,|y|),
\end{equation*}
where $\big(r_\theta^\pm\big)^{(s)}(z)=\frac{d^s}{dz^s}r_\theta^\pm(z)$ for $z\in \R$, and
\begin{equation}\label{eq-remainder-r}
	\begin{split}
		&\tilde{r}_{\vartheta(j)}^\pm(\lambda,|x-y|,|y|) \\
		&=\frac{(\lambda|x-y|-\lambda|y|)^{\vartheta(j) }}{(\vartheta(j)-1)!}\int_0^1(1-t)^{\vartheta(j)-1}\big(r_\theta^\pm\big)^{(\vartheta(j))}\big(\lambda|y|+t\lambda(|x-y|-|y|)\big)dt.   
	\end{split}
\end{equation}
Then, we have that $\omega_{j}^{\pm}(\lambda,x,y)$ can be written as the sum of 
\begin{equation}\label{eq-case1-t1}
	e^{\mp i\lambda|y|}\sum_{s=0}^{\vartheta(j)-1}\frac{\lambda^{n-2m+s}}{s!}\big(r_\theta^\pm\big)^{(s)}(\lambda|y|)\, Q_jv\big(|\cdot-y|-|y|\big)^s,
\end{equation}
and
\begin{equation}\label{eq-case1-t2}
	\lambda^{n-2m}e^{\mp i\lambda|y|} Q_j\big(v\tilde{r}_{\vartheta(j) }^\pm(\lambda,|\cdot-y|,|y|)\big).	
\end{equation}
To obtain \eqref{eq-k1 estimate-odd-1}-\eqref{eq-k1 estimate-odd-2} in the region $\lambda\langle y\rangle\le 1$,  it suffices to estimate the two terms, respectively.

We first consider the term \eqref{eq-case1-t2}. Note that the fact $ \theta \leq \vartheta(j) $. Then, according to \eqref{estforrpm}, for $ \lambda \in (0,1) $ and $ l = 0, \cdots, \frac{n+3}{2} $, the following estimates hold uniformly for $ t \in [0,1]$: \vspace{0.8em}\\ 
If $ \big|\lambda|y|+t\lambda(|x-y|-|y|)\big|\le 1$, then
\begin{equation*}
	\Big|\partial_\lambda^l \Big[\big(r_\theta^\pm\big)^{(\vartheta(j))} (\lambda|x|+t\lambda(|x-y|-|y|))\Big]\Big|\lesssim \lambda^{-l};	
\end{equation*} 
If $\big|\lambda|y|+t\lambda(|x-y|-|y|)\big|\ge 1$, then
\begin{equation*}
	\begin{aligned}
		\Big|\partial_\lambda^l \Big[\big(r_\theta^\pm\big)^{(\vartheta(j))} (\lambda|y|+t\lambda(|x-y|-|y|))\Big]\Big|&\lesssim \lambda^{-l} \big|\lambda|y|+t\lambda(|x-y|-|y|)\big|^{\max\{\theta-\vartheta(j),\, l-\frac{n-1}{2}\}}\\
		&\lesssim \lambda^{-l}\langle x\rangle^2,
	\end{aligned}
\end{equation*}
where we utilize the facts $\lambda\in(0,1)$, $t\in[0,1]$, $\big||x-y|-|y|\big|\le |x|$ and $\lambda|y|\le 1$. Thus, a direct computation  shows that for $l=0,\cdots,\frac{n+3}{2}$, the remainders  $\tilde{r}_{\vartheta(j)}^\pm(\lambda,|x-y|,|y|)$ given by \eqref{eq-remainder-r} satisfy
\begin{equation*}
	\Big|\partial^l_\lambda\Big[\tilde{r}_{\vartheta(j)}^\pm(\lambda,|x-y|,|y|)\Big]\Big| \lesssim \lambda^{\vartheta(j)-l}\langle x\rangle^{\vartheta(j)+2}.\ \ \ 
\end{equation*}
In addition, note that $j\in J_{\mathbf{k}}$, so $\vartheta(j)\le \max\{m-\frac n2, 0\}+\mathbf{k}$ by \eqref{eq-Jk-odd''} and  \eqref{eq-Jk-odd-2''}.
It immediately follows that $\frac{\beta}{2}>\vartheta(j)+2+\frac{n}{2}$ where  $\beta$ is given in \eqref{decayonV}. Since $|v(x)|\lesssim \langle x\rangle^{-\frac \beta 2}$, $v(x)\langle x\rangle^{\vartheta(j)+2}\in L^2$. Thus, for $\lambda\langle y\rangle\le 1$,  $ 0<\lambda<1$ and $l=0, \cdots, \frac{n+3}{2}$,  Leibniz's rule yields that
\begin{equation}\label{eq3.901}
	\begin{aligned}
		\Big\|\partial^l_\lambda\Big[\lambda^{n-2m}e^{\mp i\lambda|y|}& Q_j\big(v\tilde{r}_{\vartheta(j) }^\pm(\lambda,|x-y|,|y|)\big)\Big]\Big\|_{L_x^2}\\
		\lesssim&\sum_{s=0}^l\big|\partial^s_\lambda \big(\lambda^{n-2m}e^{\mp i\lambda|y|}\big)\big| \,\Big\|Q_j\big(v \partial^{l-s}_\lambda\tilde{r}_{\vartheta(j) }^\pm(\lambda,|x-y|,|y|)\big) \Big\|_{L_x^2}\\
		\lesssim&\sum_{s=0}^l\big|\partial^s_\lambda \big(\lambda^{n-2m}e^{\mp i\lambda|y|}\big)\big| \,\Big\|v(y) \partial^{l-s}_\lambda\tilde{r}_{\vartheta(j)}^\pm(\lambda,|x-y|,|y|) \Big\|_{L_x^2}\\
		\lesssim &\lambda^{n-2m+\vartheta(j)-l} \ \Big\|v(x)\langle x\rangle^{\vartheta(j)+2}\Big\|_{L_x^2}\\
		\lesssim & \lambda^{n-2m+\vartheta(j)-l}\lesssim\lambda^{n-2m+\theta-l}\langle y\rangle^{\theta-\vartheta(j)}.
	\end{aligned}
\end{equation}

Next, we estimate the term \eqref{eq-case1-t1}.
Applying \eqref{estforrpm} again and $\lambda\langle y\rangle<1$, we have 
\begin{equation*}
	\Big|\partial^l_\lambda\Big[({r}_{\theta}^\pm)^{(s)}(\lambda|y|)\Big]\Big| \lesssim (\lambda|y|)^{\max\{\theta-l-s,\ 0\}}|y|^l\lesssim\lambda^{\theta-s-l}\langle y\rangle^{\theta-s}, \quad l=0,\cdots,\frac{n+3}{2}.
\end{equation*}
This, together with Proposition \ref{corcancell}, shows that  for all  $0\le s\le\vartheta(j)-1$, $\lambda\langle y\rangle\le 1$ and $ 0<\lambda<1$, the following holds:
\begin{equation}\label{eq3.902}
	\begin{aligned}
		\Big\|\partial^l_\lambda\Big[\lambda^{n-2m+s}&e^{\mp i\lambda|y|}\big(r_\theta^\pm\big)^{(s)}(\lambda|y|)\, Q_jv\big(|y-x|-|y|\big)^s\Big]\Big\|_{L_x^2}\\
		\lesssim& \sum_{k=0}^l \left|\partial^k_\lambda \left( \lambda^{n-2m+s}e^{\mp i\lambda|y|}\right)\right|\cdot \left|\partial^{l-k}_\lambda \big(r_\theta^\pm\big)^{(s)}(\lambda|y|) \right|\cdot \Big\|Q_jv(|\cdot-y|-|y|)^s(x)\Big\|_{L_x^2}\\
		\lesssim&  \lambda^{n-2m+\theta-l}\langle y\rangle^{\theta-s} \sum_{\rho=0}^{s}|y|^{s-\rho}\Big\| Q_jv(|\cdot-y|^\rho)(x)\Big\|_{L^2_x}\\
		\lesssim&  \lambda^{n-2m+\theta-l}\langle y\rangle^{\theta-s} \sum_{\rho=0}^{s}\langle y\rangle^{s-\rho}\langle y\rangle^{\rho-\vartheta(j)}\\
		\lesssim&\lambda^{n-2m+\theta-l}\langle y\rangle^{\theta-\vartheta(j)}.	
	\end{aligned}
\end{equation}
In the above inequality, $\rho\le \vartheta(j)-1\le j$ satisfies the condition in Proposition \ref{corcancell}.  

Therefore, combining \eqref{eq3.901} with \eqref{eq3.902},   we have that in the region $\lambda\langle y\rangle\le 1$,  
\begin{equation}\label{eq3.9031}
	\|\partial^l_\lambda \omega_{j}^{\pm}(\lambda,x,y)\|_{L_x^2}\lesssim \lambda^{n-2m+\theta-l}\langle y\rangle^{\theta-\vartheta(j)}, \quad l=0,\cdots,\frac{n+3}{2},
\end{equation}
which  implies that the estimates \eqref{eq-k1 estimate-odd-1}--\eqref{eq-k1 estimate-odd-2} hold in the region $\lambda\langle y\rangle\le 1$.

Indeed, note that $\theta=\min \{\vartheta(j),\, 2m-n\}$, then $n-2m+\theta=\min\{ n-2m+\vartheta(j),\ 0\}$ and $\theta\le \vartheta(j)$, hence it immediately follows from \eqref{eq3.9031} that
\begin{equation*}
	\|\partial^l_\lambda \omega_{j}^{\pm}(\lambda,x,y)\|_{L_x^2}\lesssim \lambda^{\min\{ n-2m+\vartheta(j),\ 0\}-l}, 
\end{equation*}
which is exactly the estimate \eqref{eq-k1 estimate-odd-1} in the region $\lambda\langle y\rangle\le 1$.

To obtain \eqref{eq-k1 estimate-odd-3} from the \eqref{eq3.9031} as $\lambda\langle y\rangle\le 1$,  we divide two cases $\vartheta(j)< 2m-n$ and $\vartheta(j)\ge 2m-n$ to discuss.    For the first case,  one has $\theta=\vartheta(j)$ and $\frac{n+1}{2}-2m+\vartheta(j)<0$ , then 
$$\lambda^{n-2m+\theta}\langle y\rangle^{\theta-\vartheta(j)}
= \lambda^{\frac{n+1}{2}-2m+\vartheta(j)}(\lambda\langle y\rangle)^{\frac{n-1}{2}} \langle y\rangle^{-\frac{n-1}{2}}\le \lambda^{\frac{n+1}{2}-2m+\vartheta(j)} \langle y\rangle^{-\frac{n-1}{2}}, $$
from which  and \eqref{eq3.9031} we obtain that when $\vartheta(j)< 2m-n$ and $\lambda\langle y\rangle\le 1$,
\begin{equation}\label{eq2.43}
	\|\partial^l_\lambda \omega_{j}^{\pm}(\lambda,x,y)\|_{L_x^2}\lesssim \lambda^{\min\{\frac{n+1}{2}-2m+\vartheta(j) ,\ 0\}-l} \langle y\rangle^{-\frac{n-1}{2}}, \ \ l=0,\cdots,\frac{n+3}{2}.
\end{equation}
When $\vartheta(j)\ge 2m-n$, one has $\theta=\min \{\vartheta(j),\, 2m-n\}=2m-n$.  
If $j< 2m-\frac{n}{2}$, then $\frac{n+1}{2}-2m+\vartheta(j)\le 0$, and 
\begin{equation*}
	\lambda^{n-2m+\theta}\langle y\rangle^{\theta-\vartheta(j)}=\langle y\rangle^{2m-n-\vartheta(j)}= \lambda^{\frac{n+1}{2}-2m+\vartheta(j)}( \lambda \langle y\rangle)^{-\frac{n+1}{2}+2m-\vartheta(j)}
	\langle y\rangle^{-\frac{n-1}{2}},
\end{equation*}
which bounded by  $\lambda^{\min\{\frac{n+1}{2}-2m+\vartheta(j) ,\ 0\}} \langle y\rangle^{-\frac{n-1}{2}}$ due to  $\lambda\langle y\rangle\le 1$.
Finally, if $j=2m-\frac{n}{2}$, then $\frac{n+1}{2}-2m+\vartheta(j)=1$ and 
\begin{equation*}
	\lambda^{n-2m+\theta}\langle y\rangle^{\theta-\vartheta(j)}=\lambda^0\langle y\rangle^{-\frac{n+1}{2}}
	\le\lambda^{\min\{\frac{n+1}{2}-2m+\vartheta(j) ,\ 0\}} \langle y\rangle^{-\frac{n-1}{2}}
\end{equation*}
These arguments, together with\eqref{eq3.9031},  deduce the \eqref {eq2.43} again when $\vartheta(j)> 2m-n$ and and $\lambda\langle y\rangle\le 1$.  Hence we have obtain the estimates \eqref{eq-k1 estimate-odd-3}  as $\lambda\langle y\rangle\le 1$. 

Finally, the estimates \eqref{eq-k1 estimate-odd-2} are just from the above discussions for the case $\vartheta(j)\ge 2m-n$ if $j\ge  \max\{2m-n,\, m-\frac n2\}$.

\vskip0.2cm
\textbf{\emph{Case 2 :  $\lambda\langle y\rangle\ge 1$.}} 
\vskip0.2cm
By noting \eqref{eq-free kernel-F}, we have that
\begin{equation*}
	\Big(Q_jvR_0^\pm(\lambda^{2m}(\cdot,y))\Big)(x)=Q_jv\Big(e^{\pm i\lambda|\cdot-y|}\frac{\lambda^{\frac{n+1}{2}-2m}}{|\cdot-y|^{\frac{n-1}{2}}}F_n^\pm(\lambda|\cdot-y|)\Big)(x).
\end{equation*}
Now, $\omega_{j}^{\pm}(\lambda,x,y)$ can be written as follows:
\begin{equation*}
	\omega_{j}^{\pm}(\lambda,x,y)= Q_jv\Big(e^{\pm i\lambda(|\cdot-y|-|y|)}\frac{\lambda^{\frac{n+1}{2}-2m}}{|\cdot-y|^{\frac{n-1}{2}}}F_n^\pm(\lambda|\cdot-y|)\Big)(x).
\end{equation*}
It is clear that 
$$\left|\partial^k_\lambda \Big( e^{\pm i\lambda(|\cdot-y|-|y|)}\Big)\right|\leq \langle\, \cdot\,\rangle^k, \quad \text{for}~ k\in \mathbb{N}_0.$$ 
Then, using \eqref{eq-F decay-odd} and Lemma \ref{lem-est-integral-1}, it is not difficult to find that when $|y|\le 1$,  for any $l=0,\cdots,\frac{n+3}{2}$ and each $N\in \mathbb{N}_0$, we have
\begin{equation}\label{eq3.903}
	\|\partial^l_\lambda \omega_{j}^{\pm}(\lambda,x,y)\|_{L_x^2}\lesssim \lambda^{\frac{n+1}{2}-2m-l}\left\|\frac{\langle x\rangle^lv(x)}{|x-y|^{\frac{n-1}{2}}} \right\|_{L_x^2}  \lesssim_N \lambda^{\frac{n+1}{2}-2m-l}\langle y\rangle^{-N}.
\end{equation}
Hence,  it follows that
estimates \eqref{eq-k1 estimate-odd-1}-\eqref{eq-k1 estimate-odd-2} hold in the case when $\lambda\langle y\rangle\ge 1$ and $|y|\le1.$ 

In what follows, we consider the case when $\lambda\langle y\rangle>1$ and $|y|\ge1$. Using the Taylor expansion for $G^\pm_n(z)=e^{\pm iz}F_n^\pm(z)$ of $\big(\vartheta(j)-1\big)$-th order for $z\in \R$, we have
\begin{equation*}
	G^\pm_n(\lambda|x-y|) =\sum_{s=0}^{\vartheta(j)-1}\big(G^\pm_n\big)^{(s)}(\lambda|y|)\frac{(\lambda|x-y|-\lambda|y|)^s}{s!}+g_{\vartheta(j)}^\pm(\lambda,|x-y|,|y|),
\end{equation*}
where
\begin{align*}
	g_{\vartheta(j)}^\pm(&\lambda,|x-y|,|y|)\\
	&=\frac{(\lambda|x-y|-\lambda|y|)^{\vartheta(j) }}{\big(\vartheta(j)-1\big)!}\int_0^1(1-t)^{\vartheta(j)-1}\big(G^\pm_n\big)^{({\vartheta(j)})}(\lambda|y|+t\lambda(|x-y|-|y|))dt.
\end{align*}
Now, it is obtained that $\omega_{j}^{\pm}(\lambda,x,y)$ is a sum of the following two terms:
\begin{equation}\label{eq-case2-t1}
	e^{\mp i\lambda|y|}\lambda^{\frac{n+1}{2}-2m} \sum_{s=0}^{\vartheta(j)-1}\lambda^s\big(G^\pm_n\big)^{(s)}(\lambda|y|) Q_jv\Big( \frac{(|\cdot-y|-|y|)^s}{s!|\cdot-y|^{\frac{n-1}{2}}}\Big),
\end{equation}
and
\begin{equation}\label{eq-case2-t2}
	e^{\mp i\lambda|y|}\lambda^{\frac{n+1}{2}-2m}Q_jv\Big(\frac{g_{\vartheta(j)}^\pm(\lambda,|\cdot-y|,|y|)}{|\cdot-y|^{\frac{n-1}{2}}}\Big).
\end{equation}
Moreover, from the definition of $G^\pm_n$ and  the Leibniz rule we have 
\begin{equation}\label{eq3.904}
	(G^\pm_n)^{(s)}(z)=\sum_{k=0}^s(\pm i)^k C_s^k e^{\pm iz}(F_n^\pm)^{(s-k)}(z), \quad z\in \R,~ s=0,\cdots, \vartheta(j),
\end{equation}
where the constant $C_s^k=\frac{s!}{k!(s-k!)}$.
To estimate $\omega_{j}^{\pm}(\lambda,x,y)$,  it suffices to estimate \eqref{eq-case2-t1} and \eqref{eq-case2-t2} at this point.

We first estimate the term \eqref{eq-case2-t2} in the region of $\lambda\langle y\rangle>1$ and $|y|\ge1$. Noting the expression \eqref{eq3.904}, one has 
{\small
	$$ e^{\mp i\lambda|y|}(G^\pm_n)^{(\vartheta(j))}\big(\lambda|y|+t\lambda(|x-y|-|y|)\big)=e^{\pm it\lambda(|x-y|-|y|)}\sum_{s=0}^{\vartheta(j)} (\pm i)^s C_{\vartheta(j)}^s (F_n^\pm)^{(\vartheta(j)-s)}\big(\lambda|y|+t\lambda(|x-y|-|y|)\big).$$
}
Hence, by the Leibniz rule, for any $l=0,\cdots,\frac{n+3}{2}$, $0<\lambda<1$ and $|y|\ge 1$, the following estimate
\begin{equation}\label{eq3.9010}
	\left| \partial_\lambda^l\Big[e^{\mp i\lambda|y|}(G^\pm_n)^{(\vartheta(j))}\big(\lambda|y|+t\lambda(|x-y|-|y|)\big)\Big]\right| \lesssim
	\lambda^{-l}\langle x\rangle^{\frac{n+3}{2}},
\end{equation}
holds uniformly for $t\in[0,\,1]$.
This implies 
$$\bigg| \partial_\lambda^l\Big[\lambda^{\frac{n+1}{2}-2m}e^{\mp i\lambda|y|} g_{\vartheta(j)}^\pm(\lambda,|x-y|,|y|)\Big]\bigg|\lesssim
\lambda^{\frac{n+1}{2}-2m+\vartheta(j)-l}\langle x\rangle^{\vartheta(j)+\frac{n+3}{2}}.$$
Therefore,  using Lemma \ref{lem-est-integral-1} and the decay assumption \eqref{decayonV}, for any $l=0,\cdots,\frac{n+3}{2}$ and $\lambda\in(0,1)$,  
\begin{equation}\label{eq3.905}
	\bigg\|\partial^l_\lambda\Big[e^{\mp i\lambda|y|}\lambda^{\frac{n+1}{2}-2m}\big(Q_jv\Big(\frac{g_{\vartheta(j)}^\pm(\lambda,|\cdot-y|,|y|)}{|\cdot-y|^{-\frac{n-1}{2}}}\Big)(x)\Big]\bigg\|_{L_x^2}\lesssim \lambda^{\frac{n+1}{2}-2m+\vartheta(j)-l}\langle y\rangle^{-\frac{n-1}{2}}.
\end{equation}

We next estimate the term \eqref{eq-case2-t1}.
Note that by \eqref{eq3.904}, one has
$$e^{\mp i\lambda|y|}(G^\pm_n)^{(s)}(\lambda|y|)=\sum_{k=0}^s (\pm i)^k C_{s}^k(F_n^\pm)^{(s-k)}(\lambda|y|).$$
And by the estimate \eqref{eq-F decay-odd} of $F^\pm_n$, for  $l=0,\cdots,\frac{n+3}{2}$, there exists
$$
\left| \partial_\lambda^l\Big[e^{\mp i\lambda|y|}(G^\pm_n)^{(s)}(\lambda|y|)\Big]\right| \lesssim \lambda^{-l}.
$$
Hence, for $0\le s \le \vartheta(j)-1$,  $|y|\ge 1$ , $\lambda\langle y\rangle\ge 1$ and $l=0,\cdots,\frac{n+3}{2}$,  Proposition \ref{corcancell} derives that
\begin{equation*}
	\begin{aligned}
		&\left\| \partial_\lambda^l\Big[\lambda^{\frac{n+1}{2}-2m+s} e^{\mp i\lambda|y|}\big(G^\pm_n\big)^{(s)}(\lambda|y|)\cdot Q_jv\Big(\frac{(|\cdot-y|-|y|)^s}{|\cdot-y|^{\frac{n-1}{2}}s!}\Big)(x)\Big]\right\|_{L_x^2}\\
		\lesssim& \left|\partial_\lambda^l\left( \lambda^{\frac{n+1}{2}-2m+s} e^{\mp i\lambda|y|}\big(G^\pm_n\big)^{(s)}(\lambda|y|)\right)\right|\cdot \left\|Q_jv\Big(\frac{(|\cdot-y|-|y|)^s}{|\cdot-y|^{\frac{n-1}{2}}}\Big)(x) \right\|_{L_x^2}\\
		\lesssim& \lambda^{\frac{n+1}{2}-2m+s-l}\langle y\rangle^{s-\frac{n-1}{2}-\vartheta(j)}\lesssim \lambda^{\frac{n+1}{2}-2m+\vartheta(j)-l}\langle y\rangle^{-\frac{n-1}{2}},
	\end{aligned}
\end{equation*}
where in the last line, we apply the facts $\lambda \langle y\rangle \ge 1$ and $s<\vartheta(j)$.
Combining this with\eqref{eq3.903} and \eqref{eq3.905}, one can find that for $\lambda\langle y\rangle \ge 1$ and $\lambda\in(0,1)$,
\begin{equation}\label{eq3.906}
	\big\|\partial^l_\lambda \omega_{j}^{\pm}(\lambda,x,y)\big\|_{L_x^2}\lesssim \lambda^{\frac{n+1}{2}-2m+\vartheta(j)-l}\langle y\rangle^{-\frac{n-1}{2}}, \quad l=0,\cdots,\frac{n+3}{2}.
\end{equation}
Thus, we obtain \eqref{eq-k1 estimate-odd-1}, \eqref{eq-k1 estimate-odd-3} and \eqref{eq-k1 estimate-odd-2} in the region $\lambda\langle y\rangle\ge 1$. 

\vskip0.3cm
To sum up, \eqref{eq-k1 estimate-odd-1}, \eqref{eq-k1 estimate-odd-3} and \eqref{eq-k1 estimate-odd-2} are derived from \eqref{eq3.9031} and \eqref{eq3.906}.
\vskip0.3cm
Now  we define the function $\phi_0 \in C^{\infty}(\R)$ such that $\phi_0(z)=1$ when $|z|\leq 1/2$ and $\phi_0(z)=0$ when $|z|\geq 1$. Let $\phi_1(z)=1-\phi_0$.
Note that by the expansion \eqref{eq-free expansion-odd-4} when $\theta=\vartheta(j)$ and the cancellation property of $Q_j$, one has 
\begin{equation*}
	\Big(Q_jv\big(R_0^+(\lambda^{2m})(\cdot,\,y)-R_0^-(\lambda^{2m})(\cdot,\,y)\big)\Big)(x)=\lambda^{n-2m}
	\Big(Q_jv\big(r_{\vartheta(j)}^+(\lambda|\cdot-y|)-r_{\vartheta(j)}^-(\lambda|\cdot-y|)\big)\Big)(x).
\end{equation*}
Then we define 
$$\widetilde{\omega}_{j}^{\pm}(\lambda,x,y)=\pm \left(e^{\mp i\lambda|y|} \phi_0(\lambda \langle y\rangle)\lambda^{n-2m}
\big(Q_jvr_{\vartheta(j)}^\pm(\lambda|\cdot-y|)\big)+\phi_1(\lambda \langle y\rangle){\omega}_{j}^{\pm}(\lambda,x,y)\right). $$
Therefore, by a similar analysis as before, we can obtain that \eqref{eq3.900}, \eqref{eq-k2-odd-1} and \eqref{eq-k2-odd-2} hold.

\subsection{Proof of Lemma \ref{lem3.10}}\label{secA.2}
Recall \eqref{eq-QjR kernel-odd} states that
\begin{equation*}
	\Big(Q_jvR_0^\pm(\lambda^{2m})(\cdot,\,y)\Big)(x)=e^{\pm i\lambda |y|}\omega_{j}^{\pm}(\lambda,x,y).
\end{equation*}
We first split $\omega_{j}^{\pm}(\lambda,x,y)$ to sum of the following three part:
\begin{equation*}
	\phi_0(\lambda\langle y\rangle)\omega_{j}^\pm(\lambda,x,y)+\mathbf{1}_{\{|y|\le 1\}}\phi_1(\lambda\langle y\rangle)\omega_{j}^\pm(\lambda,x,y)+\mathbf{1}_{\{|y|> 1\}}\phi_1(\lambda\langle y\rangle)\omega_{j}^\pm(\lambda,x,y).
\end{equation*}
By \eqref{eq3.9031} in the proof of Lemma \ref{lemma-QjvR-odd}, it follows that the function $\phi_0(\lambda\langle y\rangle)\omega_{j}^\pm(\lambda,x,y)$ satisfies \eqref{eq3.1001}, that is,
\begin{equation*}
	\Big\|\partial^l_\lambda \left( \phi_0(\lambda\langle y\rangle)\omega_{j}^\pm(\lambda,x,y)\right) \Big\|_{L_x^2}\lesssim \begin{cases}
		\lambda^{\frac{n-1}{2}-2m+\vartheta(j)+\varepsilon-l}\langle y\rangle^{-\frac{n+1}{2}+\varepsilon} \quad & j<2m-\frac n2~~ \text{and}~~ \varepsilon\in [0,1] ,\\	
		\lambda^{-l}\langle y\rangle^{-\frac{n+1}{2}}, \quad & j=2m-\frac n2.
	\end{cases}	
\end{equation*}
And by \eqref{eq3.906}, one has that $\phi_1(\lambda\langle y\rangle)\omega_{j}^\pm(\lambda,x,y)$ satisfies \eqref{eq3.10011} in the region where $|y|\le 1$, i.e., for any  $l=0,\cdots,\frac{n+3}{2}$,
\begin{equation*}
	\Big\|\partial^l_\lambda \Big( \mathbf{1}_{\{|y|\le 1\}}\phi_1(\lambda\langle y\rangle)\omega_{j}^\pm(\lambda,x,y)\Big) \Big\|_{L_x^2}\lesssim \lambda^{\frac{n+3}{2}-2m+\vartheta(j)}\langle y\rangle^{-\frac{n-1}{2}}.	
\end{equation*}

Therefore, we next concentrate on the analysis of $\mathbf{1}_{\{|y|> 1\}}\phi_1(\lambda\langle y\rangle)\omega^\pm_{j}(\lambda,x,y)$.
Recall that in the proof of Lemma \ref{lemma-QjvR-odd},
\begin{equation*}
	\begin{aligned}
		\omega^\pm_{j}(\lambda,x,y)=&\lambda^{\frac{n+1}{2}-2m}e ^{\mp i\lambda|y|}Q_jv\Big(\frac{e^{\pm i\lambda(|\cdot-y|)}}{|\cdot-y|^{\frac{n-1}{2}}}F_n^\pm(\lambda|\cdot-y|)\Big)(x).
	\end{aligned}
\end{equation*}
Using the Taylor expansion for $G_n^\pm=e^{\pm iz}F_n^\pm(z)$ of $\vartheta(j)$-th order, one has
\begin{equation}\label{Fexpansion}
	G^\pm_n(\lambda|x-y|) =\sum_{s=0}^{\vartheta(j)}\big(G^\pm_n\big)^{(s)}(\lambda|y|)\frac{(\lambda|x-y|-\lambda|y|)^s}{s!}+g_{\vartheta(j)+1}^\pm(\lambda,|x-y|,|y|),
\end{equation}
where
\begin{align*}
	g_{\vartheta(j)+1}^\pm(&\lambda,|x-y|,|y|)\\
	&=\frac{(\lambda|x-y|-\lambda|y|)^{\vartheta(j)+1 }}{\big(\vartheta(j)-1\big)!}\int_0^1(1-t)^{\vartheta(j)-1}\big(G^\pm_n\big)^{({\vartheta(j)}+1)}(\lambda|y|+t\lambda(|x-y|-|y|))dt.
\end{align*}
Similar to \eqref{eq3.904}, one has 
\begin{equation} \label{eqeq3.904-1}
	(G^\pm_n)^{(s)}(z)=\sum_{k=0}^s(\pm i)^k C_s^k e^{\pm iz}(F_n^\pm)^{(s-k)}(z), \quad z\in \R,~ s=0,\cdots, \vartheta(j)+1.
\end{equation}
Now, the function $\mathbf{1}_{\{|y|> 1\}}\phi_1(\lambda\langle y\rangle)\omega_{j}^{\pm}(\lambda,x,y)$ can be written as a sum of the following three terms:
\begin{equation*}
	\begin{split}
		&L_{j,0}^\pm=\mathbf{1}_{\{|y|> 1\}}\phi_1(\lambda\langle y\rangle)e^{\mp i\lambda|y|}\lambda^{\frac{n+1}{2}-2m} \sum_{s=0}^{\vartheta(j)-1}\lambda^s\big(G^\pm_n\big)^{(s)}(\lambda|y|) Q_jv\Big( \frac{(|\cdot-y|-|y|)^s}{s!|\cdot-y|^{\frac{n-1}{2}}}\Big),\\
		&L_{j,1}^\pm= \mathbf{1}_{\{|y|> 1\}}\phi_1(\lambda\langle y\rangle)e^{\mp i\lambda|y|}\lambda^{\frac{n+1}{2}-2m+\vartheta(j)} \big(G^\pm_n\big)^{(\vartheta(j))}(\lambda|y|) Q_jv\Big( \frac{(|\cdot-y|-|y|)^{\vartheta(j)}}{\vartheta(j)!|\cdot-y|^{\frac{n-1}{2}}}\Big),\\
		&L_{j,2}^{\pm}=\mathbf{1}_{\{|y|> 1\}}\phi_1(\lambda\langle y\rangle)e^{\mp i\lambda|y|}\lambda^{\frac{n+1}{2}-2m}Q_jv\Big(\frac{g_{\vartheta(j)+1}^\pm(\lambda,|\cdot-y|,|y|)}{|\cdot-y|^{\frac{n-1}{2}}}\Big).
	\end{split}
\end{equation*}

\underline{{\bf Step 1:  the analysis of $L_{j,0}^\pm(\lambda,x,y)$.}}

We first deal with the term $L_{j,0}^\pm(\lambda,x,y)$.
Similar to \eqref{eq3.905}, for all $0\le s \le \vartheta(j)-1$,  Proposition \ref{corcancell} indicates that if $l=0,\cdots,\frac{n+3}{2}$, in the region $\lambda\langle y\rangle \geq 1/2$, we have
\begin{equation}\label{eq-est-5.001}
	\begin{aligned}
		&\left\| \partial_\lambda^l\Big[\lambda^{\frac{n+1}{2}-2m+s}e^{\mp i\lambda|y|}\big(G^\pm_n\big)^{(s)}(\lambda|y|)\cdot Q_jv\Big(\frac{(|\cdot-y|-|y|)^s}{|\cdot-y|^{\frac{n-1}{2}}s!}\Big)(x)\Big]\right\|_{L_x^2}\\
		&\ \ \ \ \lesssim \lambda^{\frac{n+1}{2}-2m+s-l}\langle y\rangle^{s-\frac{n-1}{2}-\vartheta(j)} \lesssim
		\lambda^{\frac{n-1}{2}-2m+\vartheta(j)-l}\langle y\rangle^{-\frac{n+1}{2}}.	
	\end{aligned}
\end{equation} 
Thus, by \eqref{eq-est-5.001}, for $l=0,1,\cdots,\frac{n+3}{2}$, we have
\begin{equation}\label{eq3.1005}
	\left\|\partial_\lambda^lL_{j,0}^\pm(\lambda,x,y) \right\| \lesssim
	\lambda^{\frac{n-1}{2}-2m+\vartheta(j)-l}\langle y\rangle^{-\frac{n+1}{2}}.	
\end{equation}
This, combined with the fact $\lambda\langle y\rangle \ge 1/2$ in the support of $L_{j,0}^\pm$, shows that $L_{j,0}^\pm(\lambda,x,y)$ also satisfies \eqref{eq3.1001}.\\

\indent \underline{{\bf Step 2:  the analysis of $L_{j,1}^\pm(\lambda,x,y)$.}}

To proceed, by \eqref{eqeq3.904-1}  one has 
\begin{equation}
	\begin{aligned}
		e^{\mp i\lambda|y|}\big(G^\pm_n\big)^{(\vartheta(j))}(\lambda|y|)&=F_n^\pm(\lambda|y|)+\sum_{s=0}^{\vartheta(j)-1} (\pm i)^s C_{\vartheta(j)}^s(F_n^\pm)^{(\vartheta(j)-s)}(\lambda|y|).\\
	\end{aligned}
\end{equation}
By the definition of $F_n^\pm$ and \eqref{eq2.8}, we know that
$F_n^\pm(z)=A^{\pm}+\widetilde{F}_n^\pm(z),	$
where the constant $A^{\pm}=\lim\limits_{z\rightarrow+\infty}F_n^\pm(z)$, and 
$\widetilde{F}_n^\pm(z)$ satisfy that 
$$\Big|\frac{d^k}{dz^k}\widetilde{F}_n^\pm(z)\Big|\lesssim_k
\langle z\rangle^{-1-k}, \quad \text{for}~k\in \mathbb{N}_0.
$$
Let ${B}_n^\pm(z)=\widetilde{F}_n^\pm(z)+\sum_{s=0}^{\vartheta(j)-1} (\pm i)^s C_{\vartheta(j)}^s(F_n^\pm)^{(\vartheta(j)-s)}(z)$.
Then, by \eqref{eq-F decay-odd}, there exists 
$$\Big|\frac{d^k}{dz^k}{B}_n^\pm(z)\Big|\lesssim_k
\langle z\rangle^{-1-k}, \quad \text{for}~k\in \mathbb{N}_0.
$$
We denote
\begin{equation*}
	L_{j,1}^{\pm,0}(\lambda,x,y)=\frac{A^\pm}{\vartheta(j)!}\,\mathbf{1}_{\{|y|> 1\}}\phi_1(\lambda\langle y\rangle)\lambda ^{\frac{n+1}{2}-2m+\vartheta(j)}\, Q_jv\left( \frac{(|\cdot-y|-|y|)^{\vartheta(j)}}{|\cdot-y|^\frac{n-1}{2}}\right)(x),
\end{equation*} 
\begin{equation*}
	L_{j,1}^{\pm,1}(\lambda,x,y)=\mathbf{1}_{\{|y|> 1\}}\phi_1(\lambda\langle y\rangle)\lambda^{\frac{n+1}{2}-2m+\vartheta(j)}{B}^\pm_n(\lambda|y|) Q_jv\Big(\frac{(|\cdot-y|-|y|)^{\vartheta(j)}}{\vartheta(j)!|\cdot-y|^{\frac{n-1}{2}}}\Big)(x).
\end{equation*} 
It immediately follows from Proposition \ref{corcancell} that for $l=0,\cdots,\frac{n+3}{2}$, in the support of $L_{j,1}^{\pm,1}(\lambda,x,y)$,   there is 
\begin{equation}\label{eq3.1006}
	\begin{aligned}
		\left\| \partial_\lambda^lL_{j,1}^{\pm,1}(\lambda,x,y)\right\|_{L_x^2}
		\lesssim \lambda^{\frac{n-1}{2}-2m-\vartheta(j)-l}\langle y\rangle^{-\frac{n+1}{2}}.
	\end{aligned}
\end{equation}
Therefore, $L_{j,1}^{\pm,1}(\lambda,x,y)$ satisfies \eqref{eq3.1001} because $\lambda\langle y\rangle \ge 1/2$ in the support of $L_{j,1}^{\pm,1}(\lambda,x,y).$

We now come to consider the term $L_{j,1}^{\pm,0}(\lambda,x,y)$.
Applying the following algebra identity:
\begin{equation*}
	a^k-b^k=(a-b)(a^{k-1}+a^{k-2}b+\cdots+b^{k-1}), \quad \text{for}\ k\in\mathbb{N}_0,
\end{equation*}
we have 
\begin{equation*}
	\begin{aligned}	
		\frac{1}{|x-y|^\frac{n-1}{2}}=\Big(\frac{1}{|x-y|^\frac{n-1}{2}}-\frac{1}{|y|^\frac{n-1}{2}}\Big)+\frac{1}{|y|^\frac{n-1}{2}} 
		=-\sum_{l=0}^{\frac{n-1}{2}-1}\frac{|x-y|-|y|}{|x-y|^{\frac{n-1}{2}-l}|y|^{l+1}}+\frac{1}{|y|^\frac{n-1}{2}}.
	\end{aligned}
\end{equation*}
Similarly, we have  
\begin{equation*}
	\begin{aligned}
		\frac{1}{|y|^\frac{n-1}{2}}=&\frac{1}{|y|^\frac{n-1}{2}}\Big(1- \frac{(|y|+|x-y|)^{\vartheta(j)}}{2^{\vartheta(j)}|y|^{\vartheta(j)}}\Big)+\frac{(|y|+|x-y|)^{\vartheta(j)}}{2^{\vartheta(j)}|y|^{\frac{n-1}{2}+\vartheta(j)}}\\
		=&-\sum_{l=0}^{\vartheta(j)-1}\frac{(|x-y|-|y|)(|y|+|x-y|)^{l}}{2^{l+1}|y|^{\frac{n+1}{2}+l}}+\frac{(|y|+|x-y|)^{\vartheta(j)}}{2^{\vartheta(j)}|y|^{\frac{n-1}{2}+\vartheta(j)}},
	\end{aligned}
\end{equation*}
Set $$\tilde{L}_{j,1}^{\pm,0,0}(x,y)=-\sum_{l=0}^{\vartheta(j)-1}\frac{(|x-y|-|y|)^{\vartheta(j)+1}(|y|+|x-y|)^{l}}{2^{l+1}|y|^{\frac{n+1}{2}+l}}-\sum_{l=0}^{\frac{n-1}{2}-1}\frac{(|x-y|-|y|)^{\vartheta(j)+1}}{|x-y|^{\frac{n-1}{2}-l}|y|^{l+1}}.$$
and 
\begin{equation*}
	L_{j,1}^{\pm,0,0}(x,y)=\frac{A^\pm}{\vartheta(j)!}\,\mathbf{1}_{\{|y|> 1\}}\phi_1(\lambda\langle y\rangle)\lambda ^{\frac{n+1}{2}-2m+\vartheta(j)}\cdot Q_jv\left( \tilde{L}_{j,1}^{\pm,0,0}(\cdot,y)\right)(x).
\end{equation*}
Then,  we have 
\begin{equation*}
	\left|\tilde{L}_{j,1}^{\pm,0,0}(x,y) \right|\lesssim  |y|^{-\frac{n+1}{2}}\langle x\rangle^{2\vartheta(j)}+\sum_{l=0}^{\frac{n-1}{2}-1}\frac{\langle x\rangle^{\vartheta(j)+1}}{|x-y|^{\frac{n-1}{2}-l}|y|^{l+1}}.	
\end{equation*}
Note that $|v(x)|\lesssim\langle x\rangle^{-\beta/2}$ and  $\beta/2> 2\vartheta(j)+\frac{n}{2}$ for a fixed $\mathbf{k}$ by our assumption \eqref{decayonV} and the definitions \ref{eq-Jk-odd} and \eqref{eq-Jk-odd-2} of $J_{\k}$. A direct computation using Lemma \ref{lem-est-integral-1} shows that when $|y|>1$, one has
\begin{equation*}
	\left\| Q_jv\left( \tilde{L}_{j,1}^{\pm,0,0}(\cdot,y)\right)(x)\right\|_{L_x^2} \lesssim \left\|v(x) \tilde{L}_{j,1}^{\pm,0,0}(x,y)\right\|_{L_x^2}\lesssim \langle y\rangle^{-\frac{n+1}{2}}.
\end{equation*}
Thus,  for any $l=0,\cdots,\frac{n+3}{2}$, it follows that
\begin{equation}\label{eq3.1007}
	\left\|\partial_\lambda^lL_{j,1}^{\pm,0,0}(\lambda,x,y)\right\|_{L_x^2} \lesssim\lambda ^{\frac{n+1}{2}-2m+\vartheta(j)-l} \langle y\rangle^{-\frac{n+1}{2}},
\end{equation}
which shows that $L_{j,1}^{\pm,0,0}(\lambda,x,y)$ satisfies \eqref{eq3.10011}.

Next, we have 
\begin{equation*}
	\frac{(|x-y|^2-|y|^2)^{\vartheta(j)}}{2^{\vartheta(j)}|y|^{\frac{n-1}{2}+\vartheta(j)}}=\frac{(-2xy+|x|^2)^{\vartheta(j)}}{2^{\vartheta(j)}|y|^{\frac{n-1}{2}+\vartheta(j)}}.
\end{equation*}
The cancellation property of $Q_j$  that $Q_j(vx^\alpha)=0$ for $|\alpha|\le \vartheta(j)-1 $ and $\alpha\in {\mathbb{N}_0}^n$ shows that
{\small
	\begin{equation*}
		\begin{aligned}
			Q_j\Bigg( v(x)\frac{(-2xy+|x|^2)^{\vartheta(j)}}{2^{\vartheta(j)}|x|^{\frac{n-1}{2}+\vartheta(j)}}\Bigg) =&\sum_{|\alpha|=|\beta|=\vartheta(j)}\tilde{C}_{\alpha,\beta}\frac{y^\alpha Q_j(vz^\beta)(x)}{|y|^{\frac{n-1}{2}+\vartheta(j)}}+\sum_{\substack{|\alpha|<\vartheta(j)\\ |\alpha|+|\beta|=2\vartheta(j)}}\tilde{C}_{\alpha,\beta}\frac{y^\alpha Q_j(vz^\beta)(x)}{|y|^{\frac{n-1}{2}+\vartheta(j)}}\\
			=&\tilde{L}_{j,1}^{\pm,0,1}(x,y)+\tilde{L}_{j,1}^{\pm,0,2}(x,y),
		\end{aligned}
	\end{equation*}
}
where $Q_j$ acts on the functions of variable $z$.
Denote, respectively
\begin{equation*}
	\begin{aligned}
		L_{j,1}^{\pm,0,1}(\lambda,x,y)&=\frac{A^\pm}{\vartheta(j)!}\mathbf{1}_{\{|y|> 1\}}\phi_1(\lambda\langle y\rangle)\lambda ^{\frac{n+1}{2}-2m+\vartheta(j)}\cdot\tilde{L}_{j,1}^{\pm,0,1}(x,y), \\	
		L_{j,1}^{\pm,0,2}(\lambda,x,y)&=\frac{A^\pm}{\vartheta(j)!}\mathbf{1}_{\{|y|> 1\}}\phi_1(\lambda\langle y\rangle)\lambda ^{\frac{n+1}{2}-2m+\vartheta(j)}\cdot\tilde{L}_{j,1}^{\pm,0,2}(x,y).
	\end{aligned}
\end{equation*}
At this time, we have 
$$L_{j,1}^{\pm}=L_{j,1}^{\pm,0,0}+L_{j,1}^{\pm,0,1}+L_{j,1}^{\pm,0,2}+L_{j,1}^{\pm,1}.
$$
It is easy to see that, under the assumption \ref{decayonV},  for $|y|>1$ and $l=0,\cdots,\frac{n+3}{2}$, in the support of $L_{j,1}^{\pm,0,2}$, we have 
\begin{equation}\label{eq3.1008}
	\left\|\partial_\lambda^l  L_{j,1}^{\pm,0,2}(\lambda,x,y)  \right\|_{L^2_x} \lesssim
	\lambda^{\frac{n+1}{2}-2m+\vartheta(j)-l} \langle y\rangle^{-\frac{n+1}{2}}\lesssim \lambda^{\frac{n+3}{2}-2m+\vartheta(j)-l} \langle y\rangle^{-\frac{n-1}{2}}.
\end{equation}
which shows that $L_{j,1}^{\pm,0,2}(\lambda,x,y)$ satisfies \eqref{eq3.10011}.

Meanwhile, note that $\phi_1=1-\phi_0$. Thus, we have
\begin{equation*}
	L_{j,1}^{\pm,0,1}(\lambda,x,y)=L_{j,1}^{\pm,0,1,0}(\lambda,x,y)+L_{j,1}^{\pm,0,1,1}(\lambda,x,y),
\end{equation*}
where
\begin{equation}\label{w1from}
	\begin{cases}
		L_{j,1}^{\pm,0,1,0}(\lambda,x,y)&=\frac{A^\pm}{\vartheta(j)!}\mathbf{1}_{\{|y|> 1\}}\lambda ^{\frac{n+1}{2}-2m+\vartheta(j)}\cdot\tilde{L}_{j,1}^{\pm,0,1}(x,y),\\
		L_{j,1}^{\pm,0,1,1}(\lambda,x,y)&=-\frac{A^\pm}{\vartheta(j)!}\mathbf{1}_{\{|y|> 1\}}\phi_0(\lambda\langle y\rangle)\lambda ^{\frac{n+1}{2}-2m+\vartheta(j)}\cdot\tilde{L}_{j,1}^{\pm,0,1}(x,y).
	\end{cases}	
\end{equation}
Since
\begin{equation*}
	\left\| \tilde{L}_{j,1}^{\pm,0,1}(x,y)\right\|_{L^2_x} \lesssim \langle y\rangle^{-\frac{n+1}{2}}, \quad |y|>1,	
\end{equation*}
and the $\text{supp} L_{j,1}^{\pm,0,1,1}(\lambda,x,y)\subseteq\{(\lambda,x,y); \lambda\langle y\rangle \le 1 \}$, it follows that for $l=0,\cdots,\frac{n+3}{2}$,
\begin{equation}\label{eq3.10081}
	\left\|\partial_\lambda^l \left( L_{j,1}^{\pm,0,1,1}(\lambda,x,y)\right)  \right\|_{L^2_x} \lesssim
	\lambda^{\frac{n-1}{2}-2m+\vartheta(j)+\varepsilon-l}\langle y\rangle^{-\frac{n+1}{2}+\varepsilon}
\end{equation}
holds uniformly for $\varepsilon\in [0,1]$, which means that $L_{j,1}^{\pm,0,1,1}$ satisfies \eqref{eq3.1001}. 

In summary, we have that
$$L_{j,1}^\pm=L_{j,1}^{\pm,0}+L_{j,1}^{\pm,1}=L_{j,1}^{\pm,0,0}+L_{j,1}^{\pm,0,1,0}+L_{j,1}^{\pm,0,1,1}+L_{j,1}^{\pm,0,2}+L_{j,1}^{\pm,1},$$
where $L_{j,1}^{\pm,0,1,1}(\lambda,x,y)$, $L_{j,1}^{\pm,1}(\lambda,x,y)$  satisfy \eqref{eq3.1001}, and  $L_{j,1}^{\pm,0,0}(\lambda,x,y)$, $L_{j,1}^{\pm,0,2}(\lambda,x,y)$ satisfy \eqref{eq3.10011}.
\vskip0.2cm
\indent \underline{{\bf Step 3:  the analysis of $L_{j,2}^\pm(\lambda,x,y)$.}}

It is left to consider the term $L_{j,2}^\pm(\lambda,x,y)$.
For  $l=0,\cdots,\frac{n+3}{2}$, we have
\begin{equation}\label{eq3.10061}
	\begin{aligned}
		\left| \partial_\lambda^l\Big[(G^\pm_n)^{(\vartheta(j)+1)}\big(\lambda|y|+t\lambda(|x-y|-|y|)\big)\Big]\right| \lesssim
		\langle x\rangle^{\frac{n+3}{2}}.
	\end{aligned}
\end{equation}
This implies that for  $l=0,\cdots,\frac{n+3}{2}$ and $|y|>1$, one has
$$\left| \partial_\lambda^l\Big[\lambda^{\frac{n+1}{2}-2m}\phi_1(\lambda\langle y\rangle)\cdot  e^{\mp i\lambda|y|} g_{\vartheta(j)+1}^\pm(\lambda,|x-y|,|y|)]\right|\lesssim\lambda^{\frac{n+3}{2}-2m+\vartheta(j)-l}\langle x\rangle^{\vartheta(j)+\frac{n+3}{2}}.$$
Therefore,by the decay assumption \eqref{decayonV} and Lemma \ref{lem-est-integral-1},  for  $l=0,\cdots,\frac{n+3}{2}$ and $\lambda\in(0,1)$, the following holds,
\begin{equation}\label{eq.estforg}
	\Big\|\partial^l_\lambda L_{j,2}^\pm(\lambda,x,y)\Big\|_{L_x^2}
    \lesssim \lambda^{\frac{n+3}{2}-2m+\vartheta(j)-l}
    \left\| \frac{v(x)}{|x-y|^{\frac{n-1}{2}}}\right\|_{L_x^2} 
    \lesssim \lambda^{\frac{n+3}{2}-2m+\vartheta(j)-l}\langle y\rangle^{-\frac{n-1}{2}},
\end{equation}
which shows that $L_{j,2}^\pm(\lambda,x,y)$ satisfies \eqref{eq3.10011}.

To sum up, 
let 
$$\omega_{j,1}^{\pm}(\lambda,x,y)=L_{j,1}^{\pm,0,1,0}(\lambda,x,y),$$
$$\omega_{j,2}^{\pm}(\lambda,x,y)=\phi_0(\lambda\langle y\rangle)\omega_{j}^{\pm}+L_{j,0}^\pm+L_{j,1}^{\pm,0,1,1}+L_{j,1}^{\pm,1},$$
$$\omega_{j,3}^\pm(\lambda,x,y)=\mathbf{1}_{\{|y|\le 1\}}\phi_1(\lambda\langle x\rangle)\omega_{j}^{\pm}+L_{j,1}^{\pm,0,0}+L_{j,1}^{\pm,0,2}+L_{j,2}^\pm.$$
Then we have obtained \eqref{eq3.1000}, i.e., 
\begin{equation}\label{eq3.1000'}
	\Big(Q_jvR_0^\pm(\lambda^{2m}\Big)(x,y)=e^{\pm i\lambda|y|}\Big(\omega_{j,1}^\pm(\lambda,x,y)+\omega_{j,2}^\pm(\lambda,x,y)+\omega_{j,3}^\pm(\lambda,x,y)\Big).
\end{equation}
In addition, \eqref{eq3.100001} is derived from  \eqref{w1from}.
The estimate \eqref{eq3.1001} can be derived from \eqref{eq3.1005}, \eqref{eq3.1006} and \eqref{eq3.10081}.
The estimate \eqref{eq3.10011} can be derived from \eqref{eq3.1007}, \eqref{eq3.1008} and \eqref{eq.estforg}.

\begin{appendix}
\section{The asymptotic expansions of $(M^\pm(\lambda))^{-1}$}\label{sec:Minverse}
In this appendix, we outline the processes how to obtain the asymptotic expansions of $(M^\pm(\lambda))^{-1}$ as $\lambda$ near zero, i.e.  Theorem \ref{thm-M inverse-odd}, which actually was proved by Cheng et al \cite{CHHZ-2024}. It is noteworthy that the study of the asymptotic expansions for $(M^\pm(\lambda))^{-1}$ can be traced back to Jensen and Nenciu's work \cite{JN01}, where they firstly established a unified approach to study the asymptotic expansions of $(M^\pm(\lambda))^{-1}$ for the case when $m=1$. Based on the work \cite{JN01}, in the recent year, there exist many results about the asymptotic expansions of $(M^\pm(\lambda))^{-1}$ for $H=(-\Delta)^m+V$ with  $m\ge 2$, see \cite{Green-Toprak-4D,EGT21,FSWY20,SWY22,LSY23}.

Recall that
$$M^\pm(\lambda)=U+vR_0^\pm(\lambda^{2m})v,\,\,\, \lambda>0,$$
then by the expansions of $R_0^{\pm}(x,y)$  in Lemma \ref{lemma-free resolvent expansion},  $M^\pm(\lambda)$ can be expanded on $L^2(\mathbb{R}^n)$ as follows:
\begin{itemize}
\item If zero is a regular point (i.e. $\mathbf{k}=0$), taking $\theta=\max\{0,2m-n\}+1$ in \eqref{eq-free expansion-odd-2}, then
\begin{equation}\label{eq-M-regular}
\begin{aligned}
M^\pm(\lambda)&=\sum\limits_{0\le j\le \lfloor\frac{\theta-1}{2}\rfloor}a_j^\pm\lambda^{n-2m+2j}vG_{2j}v+T_0+\lambda^{n-2m}vr_{\theta}^\pm(\lambda)v,
\end{aligned}
\end{equation}
\item If zero is a $\mathbf{k}$-th resonance with $1\le\mathbf{k}\le m_n,$ taking $\theta=\max\{0,2m-n\}+2\mathbf{k}$ in \eqref{eq-free expansion-odd-2}, then
\begin{equation}\label{eq-M-resonance}
\begin{aligned}
M^\pm(\lambda)&=\sum\limits_{0\le j\le \lfloor\frac{\theta-1}{2}\rfloor}a_j^\pm\lambda^{n-2m+2j}vG_{2j}v+T_0+\lambda^{n-2m}vr_{\theta}^\pm(\lambda)v,
\end{aligned}
\end{equation}
\item If zero is an eigenvalue (i.e. $\mathbf{k}= m_n+1$), taking $\theta=4m-n+1$ in \eqref{eq-free expansion-odd-5}, then
\begin{equation}\label{eq-M-eigenvalue}
\begin{aligned}
M^\pm(\lambda)&=\sum\limits_{0\le j\le \lfloor\frac{\theta-1}{2}\rfloor}a_j^\pm\lambda^{n-2m+2j}vG_{2j}v+T_0+b_1\lambda^{2m}vG_{4m-n}v+\lambda^{n-2m}vr_{4m-n+1}^\pm(\lambda)v.
\end{aligned}
\end{equation}
\end{itemize}
In the above expansions, $G_{2j}\ (j\in J_{\mathbf{k}})$ are the operators which were defined in \eqref{eq-G-odd}, $$T_0=U+b_0vG_{2m-n}v,$$  
and $r_{\theta}^\pm(\lambda)$ are the operators with integral kernels $r_{\theta}^\pm(\lambda|x-y|)$. Indeed, by \emph{(\romannumeral2)} in  Assumption \ref{Assumption}, we have $\beta-2\theta>n$, thus each term except for $T_0$ in \eqref{eq-M-regular}, \eqref{eq-M-resonance} and \eqref{eq-M-eigenvalue} is a Hilbert-Schmidt operator.
Moreover, for any $l=0,1,\cdots,\frac{n+3}{2},$  by \eqref{eq-r estimate-odd}, one has
\begin{equation}\label{equ3.53}
	\|\partial_\lambda^lvr_{\theta}^\pm(\lambda)v\|_{L^2}\lesssim \lambda^{n-2m+\theta-l}.
\end{equation}
The following lemma (see \cite[Lemma 3.12]{JK}) is the main method which is used to analyze the asymptotic expansions of $(M^\pm(\lambda))^{-1}$ as $\lambda$ near zero.
\begin{lemma}\label{lemma-inverse}
Let $\lambda>0$ and $Q_j\ (j\in J_{\mathbf{k}})$ be the projection operators in \eqref{eq-Q_j-odd}. Define the operator $B=(\lambda^{-j}Q_j)_{j\in J_{\mathbf{k}}}:\,\bigoplus_{j\in J_{\mathbf{k}}}Q_jL^2\to L^2$ by
$$
Bf=\sum\limits_{j\in J_{\mathbf{k}}}\lambda^{-j}Q_jf_j,\quad f=\left(f_j\right)_{j\in J_{\mathbf{k}}}\in \bigoplus_{j\in J_{\mathbf{k}}}Q_jL^2.
$$
Let $B^*$ be the dual operator of $B$ and define $\mathbb{A^\pm}$ on $\bigoplus_{j\in J_{\mathbf{k}}}Q_jL^2$ by
\begin{equation}\label{equ3.54}
	\mathbb{A^\pm}=\lambda^{2m-n}B^* M^\pm(\lambda) B.
\end{equation}
Assume that $B$ is surjective and  $\mathbb{A^\pm}$ are invertible on $\bigoplus_{j\in J_{\mathbf{k}}}Q_jL^2$, then $M^\pm(\lambda)$ are invertible on $L^2$. Especially, one has 
\begin{equation}\label{equ3.54.000}
	(M^\pm(\lambda))^{-1}=\lambda^{2m-n}B(\mathbb{A^\pm})^{-1}B^*.
\end{equation}
\end{lemma}

The following lemma contains the abstract form of the Feshbach formula (see \cite[Lemma 2.3]{JN01}) which will be used to study the inverse of the operator matrices $\mathbb{A^\pm}$.

\begin{lemma}\label{lemma-Feshbach formula}
Let $\mathbb{A^\pm}$ be an operator matrix on $\mathcal{H}=\mathcal{H}_1\oplus\mathcal{H}_2:$
\begin{equation}
\mathbb{A^\pm}=\begin{pmatrix}
a_{11} & a_{12}\\
a_{21} & a_{22}
\end{pmatrix},\ \ \ a_{ij}:\ \mathcal{H}_j\rightarrow\mathcal{H}_i,
\end{equation}
where $a_{11},a_{22}$ are closed and $a_{12},a_{21}$ are bounded. Assume that $a_{11}$ has a bounded inverse, then $\mathbb{A^\pm}$ has a bounded inverse if and only if the operator
\begin{equation}
d:= a_{22}-a_{21}a_{11}^{-1}a_{12}
\end{equation}
has a bounded inverse.  Furthermore, if $d$ has a bounded inverse, we then have 
\begin{equation}\label{eq-inverse-feshbach}
\left(\mathbb{A^\pm} \right)^{-1}=
\begin{pmatrix}
a_{11}^{-1}a_{12}d^{-1}a_{21}a_{11}^{-1}+a_{11}^{-1} & -a_{11}^{-1}a_{12}d^{-1}\\
-d^{-1}a_{21}a_{11}^{-1} & d^{-1}
\end{pmatrix}.
\end{equation}
\end{lemma}

Now we turn to prove Theorem \ref{thm-M inverse-odd}.
\begin{proof}[{\bf Proof of Theorem \ref{thm-M inverse-odd}:}]
By Proposition \ref{prop-resonance-odd},  we know that $\{Q_j\}_{j\in J_{\mathbf{k}}}$ is a direct sum decomposition of $L^2$ ( i.e. $\sum_{j\in J_{\k}} Q_jL^2=L^2$ ) when zero is the $\mathbf{k}$-th kind resonance, which implies that the operator $B=(\lambda^{-j}Q_j)_{j\in J_{\k}}$ is a bijection.
Then, by Lemma \ref{lemma-inverse}, to study the inverse of $M^\pm(\lambda)$ on $L^2$, it suffices to analyze the inverse of the operator matrices $\mathbb{A^\pm}=\lambda^{2m-n}B^* M^\pm(\lambda) B$ on  $\bigoplus_{j\in J_{\k}}Q_jL^2$. In what follows, we give the proof case by case.

\underline{{\textbf{Case~1: Regular Case (i.e. $ \mathbf{k}=0$).}}}

Since $J_{\mathbf{0}}=\{m-\frac{n}{2}\}$ when $2m+1\le n\le 4m-1$, the analysis is easier than the case when $1\le n\le 2m-1$. Thus, we only analyze the inverse of $\mathbb{A^\pm}$ when $1\le n\le 2m-1$.

By \eqref{equ3.54} and \eqref{eq-M-regular}, using the orthogonal properties of $Q_j$, we obtain that $\mathbb{A^{\pm}}:=\Big(a_{i,j}^\pm(\lambda)\Big)_{i,j\in J_{\k}}$ satisfies
\begin{equation}
\begin{split}
 a_{i,j}^\pm(\lambda)&=\sum\limits_{l_0\le l\le \tilde{m}_n}a_j^\pm\lambda^{2l-i-j}Q_ivG_{2l}vQ_j+\lambda^{2m-n-i-j}Q_{i}T_0Q_j+\lambda^{2m-n-i-j}Q_{i}vr_{\theta}^\pm(\lambda)vQ_{j},
 \end{split}
 \end{equation}
 where $\tilde{m}_n=\max\{0,m-\frac{n-1}{2}\}$ and $l_0=\lfloor\frac{\vartheta(i)+\vartheta(j)+1}{2}\rfloor$. Thus, we can write $\mathbb{A^\pm}$ in the following form
\begin{equation}\label{eq-D-Regular}
	\mathbb{A^{\pm}}=
		D^{\pm}+(r_{i,j}^\pm(\lambda))_{i,j\in J_{\mathbf{0}}},
\end{equation}
where  $D^\pm=(d_{i,j}^\pm)_{i,j\in J_{\mathbf{0}}}$ is given by
\begin{equation}\label{equ3.64}
	d_{i,j}^\pm=\begin{cases}
		a_{\frac{i+j}{2}}^\pm Q_ivG_{i+j}vQ_j,\quad&\mbox{if}\,\,i+j\,\,\mbox{is even},\\
		Q_iT_0Q_j,&\mbox{if}\,\,i+j=2m-n,\\
		0,&\mbox{else}.
	\end{cases}
\end{equation}
Then by \eqref{equ3.53} and \eqref{equ3.64}, it follows that for any $l=0,1,\cdots,\frac{n+3}{2},$ the terms $r_{i,j}^\pm(\lambda)$ satisfy
\begin{equation}\label{eq-rij-regular}
	|\partial_\lambda^lr_{i,j}^\pm(\lambda)|\lesssim
	\begin{cases}
	\lambda^{\max\{1,\, n-2m\}-l}, &\text{if}\,\,i=j=m-\frac{n}{2},\\
\lambda^{\frac12-l},  &\text{else}.
	\end{cases}
\end{equation}
Now the main difficulty is to show that the vector-valued operator $D^{\pm}$ is invertible. 

Define $D=(d_{i,j})_{i,j\in J_{\mathbf{0}}}$ by
\begin{equation*}
	d_{i,j}=\begin{cases}
		(-\i)^{i+j}(-1)^ja_{\frac{i+j}{2}} Q_ivG_{i+j}vQ_j,&\mbox{if}\,\,i+j\,\,\mbox{is even},\\
		(-\i)^{i+j}(-1)^jQ_iT_0Q_j,&\mbox{if}\,\,i+j=2m-n,\\
		0,&\mbox{else},
	\end{cases}
\end{equation*}
where $a_j=e^{\frac{+\pi \i}{m}(n-2m+2j)} a_j^+= a_j=e^{\frac{-\pi \i}{m}(n-2m+2j)} a_j^- $ (See Lemma 2.3 in  \cite{CHHZ-2024}). Then we have
\begin{equation*}\label{equ3.58}
	d_{i,j}=(-\i)^{i+j}(-1)^j e^{\frac{\pm \i\pi}{2m}(n-2m+i+j)}d_{i,j}^\pm,
\end{equation*}
so it follows  that
\begin{equation}\label{eq-D-regular}
	D=U_0^\pm D^\pm U_0^\pm U_1,
\end{equation}
where $U_0^\pm=\text{diag}\{e^{\pm\frac{\i\pi}{2m}(j+\frac{n}{2}-m)}(-\i)^jQ_j\}_{j\in J_{\mathbf{0}}}$, $U_1=\text{diag}\{(-1)^jQ_j\}_{j\in J_{\mathbf{0}}}$.

Now, we can rewrite $D$ as follows:
$$D=\begin{pmatrix}
	\textbf{D}_{00} & 0 \\[0.2cm]
	0 & Q_{m-\frac{n}{2}}T_0Q_{m-\frac{n}{2}} \\[0.2cm]
\end{pmatrix},$$
where $\textbf{D}_{00}=(d_{i,j})_{i,j\in J_{\mathbf{0}}\setminus\{m-\frac{n}{2}\} }.$ By the fact that $\textbf{D}_{00}$ is strictly positive (see Lemma 2.6 in \cite{CHHZ-2024}) and $Q_{m-\frac{n}{2}}T_0Q_{m-\frac{n}{2}}$ is invertible when zero is a regular point, we know that $D$ is invertible on $\bigoplus_{j\in J_{\mathbf{0}}}Q_jL^2$, which implies that $D^\pm$ are invertible on $\bigoplus_{j\in J_{\mathbf{0}}}Q_jL^2$. Moreover, by \eqref{eq-D-regular} we have
$$(D^\pm)^{-1}=U_0^\mp U_1D^{-1}U_0^\mp:=(M_{i,j}^\pm)_{i,j\in J_{\mathbf{0}}}.$$
Thus, by Neumann series expansion, there exists some small $\lambda_0\in(0,1)$ such that
\begin{equation*}\label{eq-A inverse-regular}
	(\mathbb{A^{\pm}})^{-1}=\big(D^\pm+(r_{i,j}^\pm(\lambda))\big)^{-1}=(D^\pm)^{-1}+\Gamma^{\pm}(\lambda),
\end{equation*}
where $\Gamma^{\pm}(\lambda)=(D^\pm)^{-1}\sum\limits_{l=1}^\infty\big(-(r_{ij}^\pm(\lambda))(D^\pm)^{-1}\big)^l$.
This, together with \eqref{equ3.54.000} and \eqref{eq-rij-regular}, yields
\begin{equation}\label{eq-M inverse-regular}
		\big(M^\pm(\lambda)\big)^{-1}=\lambda^{2m-n}B(\mathbb{A^{\pm}})^{-1}C
		=\sum\limits_{i,j\in J_{\mathbf{0}}}\lambda^{2m-n-i-j}Q_j\Big(M_{i,j}^\pm+\Gamma_{i,j}^\pm(\lambda)\Big)Q_j,
\end{equation}
where $\Gamma_{ij}^\pm(\lambda)$ satisfies the first estimate of \eqref{equ3.47}.

\underline{{\textbf{Case~2: Resonance Case (i.e. $ 1\le \mathbf{k}\le m_n$).}}}

By the orthogonal properties of $Q_j$, we obtain from \eqref{equ3.54} and \eqref{eq-M-resonance} that $\mathbb{A^{\pm}}:=\Big(a_{i,j}^\pm(\lambda)\Big)_{i,j\in J_{\k}}$ satisfies
\begin{equation}\label{equ3.53.111}
\begin{split}
 a_{i,j}^\pm(\lambda)&=\sum\limits_{l_0\le l\le \tilde{m}_n+\k}a_j^\pm\lambda^{2l-i-j}Q_ivG_{2l}vQ_j+\lambda^{2m-n-i-j}Q_{i}T_0Q_j+\lambda^{2m-n-i-j}Q_{i}vr_{\theta}^\pm(\lambda)vQ_{j},
 \end{split}
 \end{equation}
 where $\tilde{m}_n=\max\{0,m-\frac{n-1}{2}\}$ and $l_0=\lfloor\frac{\vartheta(i)+\vartheta(j)+1}{2}\rfloor$. Thus, we can write $\mathbb{A^\pm}$ in the following form
\begin{equation}\label{equ3.63}
	\mathbb{A^{\pm}}=
		D^{\pm}+(r_{i,j}^\pm(\lambda))_{i,j\in J_{\k}},
\end{equation}
where  $D^\pm=(d_{i,j}^\pm)_{i,j\in J_{\k}}$ and $d_{i,j}$ is given by \eqref{equ3.64}. Moreover, for any $l=0,1,\cdots,\frac{n+3}{2},$ the terms $r_{i,j}^\pm(\lambda)$ satisfy
\begin{equation}\label{equ3.65.55}
	|\partial_\lambda^lr_{i,j}^\pm(\lambda)|\lesssim
	\begin{cases}
	\lambda^{\max\{1,\, n-2m\}-l}, &\text{if}\,\,i=j=m-\frac{n}{2},\\
	\lambda^{1-l},   &\text{if}\,\,i=j=2m-\frac{n}{2},\\
\lambda^{\frac12-l},  &\text{else}.
	\end{cases}
\end{equation}
Now we turn to show that the operator $D^{\pm}$ is invertible. Since they have quite different structures in dimensions $n<2m$ and $n>2m$,
thus, we divide the proof into $1\le n\le 2m-1$ and $2m+1\le n\le 4m-1$ respectively.

When $1\le n\le 2m-1$, we similarly define $D=(d_{i,j})_{i,j\in J_{\k}}$ by
\begin{equation}
	d_{i,j}=\begin{cases}
		(-\i)^{i+j}(-1)^ja_{\frac{i+j}{2}} Q_ivG_{i+j}vQ_j,&\mbox{if}\,\,i+j\,\,\mbox{is even},\\
		(-\i)^{i+j}(-1)^jQ_iT_0Q_j,&\mbox{if}\,\,i+j=2m-n,\\
		0,&\mbox{else},
	\end{cases}
\end{equation}
 where $a_j=e^{\frac{+\pi \i}{2m}(n-2m+2j)} a_j^+=e^{\frac{-\pi \i}{2m}(n-2m+2j)} a_j^- $ (See Lemma 2.3 in  \cite{CHHZ-2024}).
Then it follows  that
\begin{equation}\label{equ3.64'}
	D=U_0^\pm D^\pm U_0^\pm U_1,
\end{equation}
where $U_0^\pm=\text{diag}\{e^{\pm\frac{\i\pi}{2m}(j+\frac{n}{2}-m)}(-\i)^jQ_j\}_{j\in J_{\k}}$, $U_1=\text{diag}\{(-1)^jQ_j\}_{j\in J_{\k}}$.

And we can rewrite $D$ in the following form:
$$D=\begin{pmatrix}
	\textbf{D}_{00} & 0 & D_{01}\\[0.2cm]
	0 & Q_{m-\frac{n}{2}}T_0Q_{m-\frac{n}{2}} & 0\\[0.2cm]
	D_{10} & 0 & \mathbf{D}_{11}
\end{pmatrix},$$
where $D_{10}=D_{01}^*$, and $\textbf{D}_{00}=(d_{i,j})_{i,j\in J'_{\k} }, \textbf{D}_{11}=(d_{i,j})_{i,j\in J''_{\k} }$ with 
\begin{equation*}\label{equ3.12.11}
	J'_{\k}=\{j\in J_{\k};\,j<\mbox{$m-\frac{n}{2}$}\},\quad	J''_{\k}=\{j\in J_{\k};\,\mbox{$m-\frac{n}{2}$}<j<2m-\mbox{$\frac n2$}\}.
\end{equation*}
Particularly, it is known that $\textbf{D}_{00}$ is strictly positive and $\textbf{D}_{11}$ is strictly negative, see Lemma 2.6 in \cite{CHHZ-2024}.

By the definition of $Q_{m-\frac{n}{2}}$, it is known that $Q_{m-\frac{n}{2}}T_0Q_{m-\frac{n}{2}}$ is invertible on $Q_{m-\frac{n}{2}}L^2$.
Then, by Lemma \ref{lemma-Feshbach formula} , $D$ is invertible on $\bigoplus_{j\in J_{\k}}Q_jL^2$ if and only if
\begin{equation}\label{equ3.66}
	d:=\textbf{D}_{11}-D_{01}^\ast \textbf{D}_{00}^{-1}D_{01},
\end{equation}
is invertible on $\bigoplus_{j\in J_\k''}Q_{j}L^2$.
Now we decompose
$Q_j=Q_{j,1}+Q_{j,2}$,
where
\begin{equation*}\label{equ3.67}
	Q_{j,2}L^2=Q_jL^2\bigcap \{x^\alpha v;\,\,|\alpha|\le j\}^\perp.
\end{equation*}
Since $Q:=\text{diag}\{Q_{j}\}_{j\in J_\k''}$ is the identity operator on $\bigoplus_{j\in J_{\k}''}Q_{j}L^2$, and
$Q_ivG_{i+j}vQ_jQ_{j,2}=0$ for $i\in J_\k', j\in J_\k''$,
we obtain
\begin{equation}\label{equ3.66'}
		d=Q\mathbf{D}_{11}Q-D_{01}^\ast \mathbf{D}_{00}^{-1}D_{01}
		=Q'\mathbf{D}_{11}Q'-D_{01}^\ast \mathbf{D}_{00}^{-1}D_{01},
\end{equation}
where $Q'=\text{diag}\{Q_{j,1}\}_{j\in J_\k''}$.
It suffices to prove that $d$ is injective on $\bigoplus_{j\in J_\k''}Q_{j}L^2$. Assume that $f=(f_j)_{j\in J_\k''}\in  \bigoplus_{j\in J_\k''}Q_{j}L^2$ and
\begin{equation}\label{equ3.68.eq}
	0=\langle df,f\rangle=\langle Q'\mathbf{D}_{11}Q'f,f\rangle-\langle D_{01}^\ast \mathbf{D}_{00}^{-1}D_{01}f,f\rangle.
\end{equation}
Note the fact that  $\textbf{D}_{00}$ is strictly positive, $\textbf{D}_{11}$ is strictly negative (see Lemma 2.6 in [4]). We have
\begin{equation}\label{equ3.69}
	Q'f=0\quad\mbox{ and }\quad D_{01}f=0.
\end{equation}
First, $Q'f=0$ implies $f_j\in Q_{j,2}L^2$, i.e.,  $f_j\in \{x^\alpha v;\,\,|\alpha|\le j\}^\perp$, $j\in J_\k''$. Second, observe that
\begin{equation*}\label{equ3.70}
	D_{01}f=\begin{pmatrix}
		(-\i)^{2m-n}(-1)^{m-\frac{n+1}{2}+\k}Q_{m-\frac{n-1}{2}-\k}T_0Q_{m-\frac{n+1}{2}+\k}f_{m-\frac{n+1}{2}+\k}\\[0.2cm]
		\vdots\\[0.2cm]
		(-\i)^{2m-n}(-1)^{m-\frac{n-1}{2}}Q_{m-\frac{n+1}{2}}T_0Q_{m-\frac{n-1}{2}}f_{m-\frac{n-1}{2}}
	\end{pmatrix},
\end{equation*}
thus $D_{01}f=0$ and the fact that $f_j\in Q_{j}L^2\subset S_{m-\frac{n}{2}}L^2$ imply
\begin{align*}\label{equ.Qj}
		Q_{2m-n-j}T_0f_{j}=0.
\end{align*}
Since  $Q_{j}L^2\subset S_{j-1}L^2$, by definition we have  $S_{2m-n-j}T_0f_j=0$. Hence $S_{2m-n-j-1}T_0f_j=0$.
Combining the above, we have $f_j\in S_{j}L^2$.
On the other hand, since $f_j\in Q_{j}L^2$, it follows that $f_j\in (S_{j}L^2)^\bot$. We conclude that
$f_j\equiv0$ for all $j\in J_\k''$ (i.e., $f\equiv0$)
provided \eqref{equ3.68.eq} holds. Therefore we have proved that $d$ is invertible as well as rigidly negative definite on  $\bigoplus_{j\in J_\k''}Q_{j}L^2$, and that $D^\pm$ are invertible on $\bigoplus_{j\in J_{\k}}Q_jL^2$. Moreover, by \eqref{eq-inverse-feshbach} and \eqref{equ3.64'} we have
\begin{equation}\label{eq: inverse of D^pm}
(D^\pm)^{-1}=U_0^\pm U_1D^{-1}U_0^\pm:=(M_{i,j}^\pm)_{i,j\in J_{\k}},
\end{equation}
where 
\begin{equation}\label{eq: struce of the inverse of D}
D^{-1}=
\begin{pmatrix}
\textbf{D}_{00}^{-1}D_{01}d^{-1}D_{10}\textbf{D}_{00}^{-1}+\textbf{D}_{00}^{-1} & -\textbf{D}_{00}^{-1}D_{01}d^{-1}\\
-d^{-1}D_{10}\textbf{D}_{00}^{-1} & d^{-1}
\end{pmatrix}.
\end{equation}
Thus, by Neumann series expansion, there exists some small $\lambda_0\in(0,1)$ such that
\begin{equation*}\label{equ3.71}
	(\mathbb{A^{\pm}})^{-1}=\big(D^\pm+(r_{i,j}^\pm(\lambda))\big)^{-1}=(D^\pm)^{-1}+\Gamma^{\pm}(\lambda),
\end{equation*}
where $\Gamma^{\pm}(\lambda)=(D^\pm)^{-1}\sum\limits_{l=1}^\infty\big(-(r_{ij}^\pm(\lambda))(D^\pm)^{-1}\big)^l$.
This, together with \eqref{equ3.54.000} and \eqref{equ3.65.55}, yields
\begin{equation}\label{equ3.72}
		\big(M^\pm(\lambda)\big)^{-1}=\lambda^{2m-n}B(\mathbb{A^{\pm}})^{-1}C
		=\sum\limits_{i,j\in J_{\k}}\lambda^{2m-n-i-j}Q_j\Big(M_{i,j}^\pm+\Gamma_{i,j}^\pm(\lambda)\Big)Q_j,
\end{equation}
where $\Gamma_{ij}^\pm(\lambda)$ satisfies the first estimate of \eqref{equ3.47}.

When $2m+1\le n\le 4m-1$, we notice that the relation \eqref{equ3.64'} still holds, however, $D$ has the following form:
\begin{equation*}\label{equ3.82}
	\begin{split}
		D&=\begin{pmatrix}
			Q_{m-\frac{n}{2}}T_0Q_{m-\frac{n}{2}} & 0\\
			0 & \mathbf{D}_{11}
		\end{pmatrix},
	\end{split}
\end{equation*}
where $\textbf{D}_{11}=(d_{i,j})_{i,j\in J_{\k''}}$ is defined by
\begin{equation*}
	d_{i,j}=\begin{cases}
		(-\i)^{i+j}(-1)^ja_{\frac{i+j}{2}} Q_ivG_{i+j}vQ_j,\quad& \mbox{if}\,\, i+j\, \,\mbox{is even},\\
		0,&\mbox{if}\,\, i+j\,\, \mbox{is odd}.
	\end{cases}
\end{equation*}
Since $\mathbf{D}_{11}$ is strictly negative  on $\bigoplus_{j\in J_{\k}''}Q_{j}L^2$ (see Lemma 2.6 in [4]), which implies that $\mathbf{D}_{11}$ is invertible on $\bigoplus_{j\in J_{\k}''}Q_{j}L^2$. Since  $Q_{m-\frac{n}{2}}T_0Q_{m-\frac{n}{2}}$ is invertible on $Q_{m-\frac{n}{2}}L^2$, we obtain the invertibility of  $D$. By \eqref{equ3.64'}, $D^{\pm}$ is invertible, then  Neumann series expansion yields \eqref{equ3.72}

\underline{{\textbf{Case~3: Eigenvalue Case (i.e. $ \mathbf{k}=m_n+1$).}}}

Similarly, by the orthogonal properties of $Q_j$, in this case, we obtain from \eqref{equ3.54} and \eqref{eq-M-eigenvalue} that $\mathbb{A^{\pm}}:=\Big(a_{i,j}^\pm(\lambda)\Big)_{i,j\in J_{m_n+1}}$ with $a_{i,j}^\pm(\lambda)$ satisfies
\begin{equation}
\begin{split}
a_{i,j}^\pm(\lambda)=\sum\limits_{l_0\le l\le 2m-\frac{n+1}{2}}a_j^\pm\lambda^{2l-i-j}Q_ivG_{2l}vQ_j&+\lambda^{2m-n-i-j}Q_{i}T_0Q_j+b_1\lambda^{2m-n-i-j}Q_{i}vG_{2m-n}vQ_{j}\\
& +\lambda^{2m-n-i-j}Q_{i}vr_{4m-n+1}^\pm(\lambda)vQ_{j},
 \end{split}
 \end{equation}
 where $l_0=\lfloor\frac{\vartheta(i)+\vartheta(j)+1}{2}\rfloor$. Thus, we can write $\mathbb{A^\pm}$ in the following form
\begin{equation}\label{eq-A-eigenvalue}
	\mathbb{A^{\pm}}=
		\text{diag}(D^{\pm}, Q_{2m-\frac{n}{2}}vG_{4m-n}vQ_{2m-\frac{n}{2}})+(r_{i,j}^\pm(\lambda))_{i,j\in J_{m_n+1}},
\end{equation}
where  $D^\pm=(d_{i,j}^\pm)_{i,j\in J_{m_n+1}}$ is given in \eqref{equ3.64}, and the terms $r_{i,j}^\pm(\lambda)$ satisfy the same estimates as in \eqref{equ3.65.55}. 
By the similar process as in the case of $0\le\mathbf{k}\le m_n$, we know that $D^\pm$ is invertible, and the invertibility of $Q_{2m-\frac{n}{2}}vG_{4m-n}vQ_{2m-\frac{n}{2}}$ can be obtained by the same arguments in \cite{JN01,SWY22}, thus $\mathbb{A^{\pm}}$ is invertible. Then  Neumann series expansion and Lemma \ref{lemma-inverse} yields \eqref{equ3.72}.
Thus, the proof of Theorem \ref{thm-M inverse-odd} is completed.
\end{proof}

\end{appendix}

\noindent{\bf Acknowledgments:} 
Avy Soffer is partially supported by NSF-DMS (No. 2205931) and the Simon’s Foundation (No. 395767).  
Zhao Wu and Xiaohua Yao are partially supported by NSFC (No. 12171182, 42450275).


\end{document}